\documentclass[letterpaper,11 pt]{article}
\usepackage{braket,amsfonts,color,makecell}
\usepackage{xcolor}

\usepackage{amsmath, amssymb, bbm, xspace}

\newcommand{\Real}{\ensuremath{\mathbb{R}}}

\newcommand{\minimize}[1]{\displaystyle\minim_{#1}}
\newcommand{\minim}{\mathop{\hbox{\rm minimize}}}
\newcommand{\maximize}[1]{\displaystyle\maxim_{#1}}
\newcommand{\maxim}{\mathop{\hbox{\rm maximize}}}
\DeclareMathOperator*{\st}{subject\;to}

\def\spose#1{\hbox to 0pt{#1\hss}}

\def\text #1{\hbox{\quad#1\quad}}

\def\nthinsp{\mskip -2   mu}

\def\superstar{^{\raise 0.5pt\hbox{$\nthinsp *$}}}
\def\SUPERSTAR{^{\raise 0.5pt\hbox{$*$}}}

\def\lamstarT {\lambda^{\raise 0.5pt\hbox{$\nthinsp *$}T}}

\def\Cscr{{\cal C}}

\def\Nscr{{\cal N}}

\def\Nscr{{\cal N}}

\def\Xscr{{\cal X}}
\def\Yscr{{\cal Y}}

\def\hbar{\skew{4.2}\bar h}

		\def\bk1{{\rm 1\kern-.17em l}}
		\def\bkD{{\rm I\kern-.17em D}}
		\def\bkR{{\rm I\kern-.17em R}}
		\def\bkP{{\rm I\kern-.17em P}}
		\def\bkY{{\bf \kern-.17em Y}}
		\def\bkZ{{\bf \kern-.17em Z}}

		\def\beq{\begin{eqnarray}}
		\def\bc{\begin{center}}
		\def\be{\begin{enumerate}}
		\def\bi{\begin{itemize}}
		\def\bs{\begin{small}}
		\def\bS{\begin{slide}}
		\def\ec{\end{center}}
		\def\ee{\end{enumerate}}
		\def\ei{\end{itemize}}
		\def\es{\end{small}}
		\def\eS{\end{slide}}
		\def\eeq{\end{eqnarray}}
		\def\qed{\quad \vrule height7.5pt width4.17pt depth0pt}

	\def\cp2problem#1#2#3#4{\fbox
		 {\begin{tabular*}{0.9\textwidth}
			{@{}l@{\extracolsep{\fill}}l@{\extracolsep{6pt}}l@{\extracolsep{\fill}}c@{}}
				#1 & & $#4 $
			\end{tabular*}}}
\def\z{\phantom 0}
		
		\renewcommand{\emph}[1]{\textbf{#1}}

		\def\bk1{{\rm 1\kern-.17em l}}
		\def\bkD{{\rm I\kern-.17em D}}
		\def\bkR{{\rm I\kern-.17em R}}
		\def\bkP{{\rm I\kern-.17em P}}
		
		\def\bkZ{{\bf{Z}}}

\newcommand {\beeq}[1]{\begin{equation}\label{#1}}
\newcommand {\eeeq}{\end{equation}}
\newcommand {\bea}{\begin{eqnarray}}
\newcommand {\eea}{\end{eqnarray}}

\def\texitem#1{\par\smallskip\noindent\hangindent 25pt
               \hbox to 25pt {\hss #1 ~}\ignorespaces}

\def\st{\mbox{subject to}}

\usepackage{amsthm}
\usepackage{multirow}
\usepackage{framed}
\usepackage{float}
\usepackage{shadethm}
\usepackage{epstopdf}
\usepackage[para,online,flushleft]{threeparttable}
\usepackage{algorithmic}
\usepackage{url}
\usepackage{verbatim}
\usepackage{amsmath} 
\usepackage{amssymb}  
\usepackage{amsfonts}
\usepackage{acronym}
\usepackage{mathtools}
\usepackage[normalem]{ulem}
\usepackage{graphicx}
\usepackage{epsfig}
\usepackage{cite}
\usepackage{algorithm}
	\newshadetheorem{thm}{theorem}
	\definecolor{shadethmcolor}{HTML}{EDF8FF}
	\definecolor{shaderulecolor}{HTML}{45CFFF}
	\definecolor{shaderulecolor}{gray}{0}
	\colorlet{shadecolor}{orange!15}

	\newshadetheorem{alg}{Algorithm}
	\definecolor{shadethmcolor}{HTML}{EDF8FF}
	\definecolor{shaderulecolor}{HTML}{45CFFF}
	\definecolor{shaderulecolor}{gray}{0}
	\colorlet{shadecolor}{orange!15}

	\newshadetheorem{prop}{proposition}
	\definecolor{shadethmcolor}{HTML}{EDF8FF}
	\definecolor{shaderulecolor}{HTML}{45CFFF}
\definecolor{shaderulecolor}{gray}{0}
 	\colorlet{shadecolor}{orange!15}

\newtheorem{theorem}{Theorem}
\usepackage{hyperref}
\newtheorem{assumption}{Assumption}
\newtheorem{corollary}{Corollary}
\newtheorem{definition}{Definition}

\newtheorem{lemma}{Lemma}
\newtheorem{proposition}{Proposition}
\newtheorem{remark}{Remark}
\def\bko{{\rm 1\kern-.17em l}}

\newcommand{\fy}[1]{{\color{black}#1}}
\newcommand{\us}[1]{{\color{black}#1}}
\newcommand{\uss}[1]{{\color{black}#1}}
\newcommand{\usv}[1]{{\color{black}#1}}

\newcommand{\fyy}[1]{{\color{black}#1}}
\newcommand{\fyyy}[1]{{\color{black}#1}}
\newcommand{\ic}[1]{{\color{black}#1}}

\def\Nscr{{\mathcal N}}

\newcommand{\ww}{\mathbf{w}}
\newcommand{\uu}{\mathbf{u}}
\newcommand{\vx}{\mathbf{x}}

\newcommand{\vz}{\mathbf{z}}

\def\be{\begin{enumerate}}
\def\ee{\end{enumerate}}
\def\Nscr{{\mathcal N}}
\def\st{\mbox{subject to}}

\newcommand{\todo}[1]{\vspace{5 mm}\par \noindent \marginpar{{ToDo}}
\framebox{\begin{minipage}[c]{0.9\linewidth} \tt #1
\end{minipage}}\vspace{5 mm}\par}

\def\argmin{\mathop{\rm argmin}}

 \newcommand{\remove}[1]{}
\newcommand{\EXP}[1]{\mathsf{E}\!\left[#1\right] }

\newcommand{\pmat}[1]{\begin{pmatrix} #1 \end{pmatrix}}

\newcommand{\x}{{\mathbf{x}}}
\newcommand{\y}{{\mathbf{y}}}

\usepackage{multirow}

\def\sF{\mathcal{F}}

\def\Real{\mathbb{R}}

\def\argmin{\mathop{\rm argmin}}

\newcounter{algsubstate}
\renewcommand{\thealgsubstate}{\alph{algsubstate}}

\begin{document}
\sloppy
%
%

\title{Complexity guarantees for an implicit smoothing-enabled method for stochastic MPECs}

\author{Shisheng Cui\thanks{Industrial \& Manufacturing Engineering,
	Pennsylvania State University,  University Park, State College, PA
		16802, USA, \texttt{suc256@psu.edu};}  \and Uday V. Shanbhag\thanks{Industrial \& Manufacturing Engineering,
	Pennsylvania State University,  University Park, State College, PA
		16802, USA, \texttt{udaybag@psu.edu}; Shanbhag gratefully acknowledges the support from NSF CMMI-1538605 and DOE ARPA-E award DE-AR0001076. } \and Farzad Yousefian\thanks{Department of Industrial \& Systems Engineering, Rutgers University, Piscataway, NJ 08854, USA, \texttt{farzad.yousefian@rutgers.edu}; 
		Yousefian  gratefully acknowledges the support of the NSF through CAREER grant ECCS-1944500.		
		} }
 
\maketitle

\begin{abstract}  
Mathematical programs with equilibrium constraints (MPECs) represent {a class of hierarchical programs that allow for
modeling problems} in engineering, economics, finance, and
statistics. {While stochastic generalizations} have been
assuming increasing relevance, {there is a pronounced} absence of
    efficient first/zeroth-order schemes with non-asymptotic rate guarantees
    for resolving even deterministic variants of such problems. We consider a
    subclass of stochastic MPECs (\fyyy{SMPECs}) where the parametrized lower-level equilibrium
    problem is  given by a  deterministic/stochastic variational inequality
    (VI) problem whose mapping is strongly monotone, uniformly in upper-level
    decisions. Under suitable assumptions, 
    {this paves the way for resolving the implicit problem with a Lipschitz continuous objective} via a gradient-free zeroth-order method by
    leveraging a locally randomized spherical smoothing framework. In this
    setting, we {present schemes for single-stage and two-stage stochastic MPECs when the upper-level problem is either convex or nonconvex.  {\bf (I). Single-stage SMPECs.} In single-stage SMPECs, in convex regimes, our proposed inexact schemes are characterized by a complexity in upper-level projections, upper-level samples, and lower-level projections   of $\mathcal{O}(\tfrac{1}{\epsilon^2})$, $\mathcal{O}(\tfrac{1}{\epsilon^2})$, and $\mathcal{O}(\tfrac{1}{\epsilon^2}\ln(\tfrac{1}{\epsilon}))$, respectively. Analogous bounds for the nonconvex regime are  $\mathcal{O}(\tfrac{1}{\epsilon})$, $\mathcal{O}(\tfrac{1}{\epsilon^2})$, and $\mathcal{O}(\tfrac{1}{\epsilon^3})$, respectively. {\bf (II). Two-stage SMPECs.} In two-stage SMPECs, in convex regimes, our proposed inexact schemes have a complexity in upper-level projections, upper-level samples, and lower-level projections   of $\mathcal{O}(\tfrac{1}{\epsilon^2}), \mathcal{O}(\tfrac{1}{\epsilon^2})$, and $\mathcal{O}(\tfrac{1}{\epsilon^2}\ln(\tfrac{1}{\epsilon}))$ while the corresponding bounds in the nonconvex regime are  $\mathcal{O}(\tfrac{1}{\epsilon})$,$\mathcal{O}(\tfrac{1}{\epsilon^2})$, and $\mathcal{O}(\tfrac{1}{\epsilon^2}\ln(\tfrac{1}{\epsilon}))$, respectively. In addition, we derive statements for exact as well as accelerated counterparts. Preliminary
    numerics suggest that the schemes scale with problem size, are relatively
    robust to modification of algorithm parameters, {show distinct benefits in obtaining near-global minimizers for convex implicit problems in contrast with competing solvers}, and  provide solutions of
    similar accuracy in a fraction of the time taken by sample-average
    approximation (SAA).  
    }  
\end{abstract}

\section{Introduction}
In this paper, we consider the resolution of variants and stochastic generalizations of the mathematical program with equilibrium  constraints (MPEC), given by 
\begin{align}\tag{MPEC} \label{MPEC}
\begin{aligned}
    \min_{\x, \y} & \quad f(\x, \y) \\
    \st & \quad \y \in \mbox{SOL}(\Yscr, F(\x,\bullet)), \\
     & \quad \x \in \Xscr, 
\end{aligned}
\end{align}
where $f: \Real^n \times \fy{\Real^m} \to \Real$ is a real-valued function, $F:
 \Xscr \times \Yscr \to \Real^m$, $\Xscr \subseteq \Real^n$ and  $\Yscr
\subseteq \Real^m$ denote closed and convex sets, and $\fy{\mbox{SOL}(\Yscr,
F(\x,\bullet))}$ denotes the solution set of the parametrized variational
inequality problem VI$(\Yscr, F(\x,\bullet))$, given an upper-level decision
$\x$. Recall that the variational inequality problem VI$(\Yscr, F(\x,\bullet))$ requires a {vector} $\y$ in {the set} $\Yscr$ such that 
\begin{align}\tag{VI$(\Yscr, F(\x,\bullet))$} 
    (\tilde{\y}-\y)^T F(\x,\y) \ \geq \ 0, \qquad \forall \ \tilde{\y} \ \in \ \Yscr. 
\end{align} MPECs have a broad range of applications arising in hierarchical
optimization, {frictional contact} problems, power systems~\cite{hobbs-strategic}, traffic
equilibrium problems~\cite{hearn04mpec}, and Stackelberg equilibrium
problems~\cite{sherali83stackelberg}. A comprehensive survey of models,
analysis, and algorithms can be found in~\cite{luo96mathematical} while a
subsequent monograph emphasized the implicit
framework~\cite{outrata98nonsmooth}. 

\medskip

The MPEC is an ill-posed generalization of a nonconvex and nonlinear program,  an observation that follows from considering the setting where $\Yscr$ is a convex cone in $\Real^m$. In such an instance, (MPEC) reduces to a mathematical program with complementarity constraints (MPCC) since $\y$ solves VI$(\Yscr, F(\x,\bullet))$ if and only if \fyyy{$\y$} solves CP$(\Yscr,F(\x,\bullet)\fy{)}$, defined as the problem of finding \fy{a vector} $\y$ such that   
\begin{align}\tag{CP$(\Yscr, F(\x,\bullet))$} 
    \Yscr \ni \y \ \perp \ F(\x, \y) \in \Yscr^*,
\end{align}
where $\Yscr^* \triangleq \{u \mid y^Tu \geq 0, y \in \Yscr\}$. When $\Yscr$ is the nonnegative orthant, then (MPEC) reduces to the following MPCC, {which can be cast as an ill-posed nonlinear program.} 
\begin{align} \tag{MPCC} \label{MPCC}
\begin{aligned}
    \min_{\x, \y} \quad  &  f(\x, \y)  \\
    \st  \quad 0  \leq \y & \perp F(\x, \y) \geq 0, \\
    \quad \x  & \in \Xscr.
\end{aligned}
\end{align}
Ill-posedness of (MPCC) arises from noting that standard constraint qualifications (such as the Mangasarian-Fromovitz constraint qualification) fail to hold at any feasible point. This has led to a concerted effort in developing weaker stationarity conditions for MPECs~\cite{scheel2000mathematical} as well as a host of regularization~\cite{ragunathan05interior,fletcher02local,anitescu-solving,jiang00smooth,leyffer2006interior} and penalization~\cite{xu04convergence} schemes. 

\medskip

 Yet {an enduring} gap {persists} in the development of
algorithms for such problems. Despite a wealth of developments in the field of
zeroth and first-order algorithms for deterministic and stochastic convex and
nonconvex optimization, there are no available non-asymptotic rate guarantees
for either zeroth or first-order schemes for MPECs or their stochastic
variants. In particular, our interest lies in two distinct stochastic variants presented as follows. 

\smallskip

\noindent {\bf \large 1.1. Problems of interest.} We focus on the problem (MPEC) where the lower-level map $F(\x,\bullet)$ is strongly monotone over $\Yscr$ {uniformly in $\x$}. This ensures that the solution of VI$(\Yscr,F(\x,\bullet))$ is a singleton for every $\x \in \Xscr$. We consider two {settings.}

\smallskip

\noindent {(i) \bf {Single-stage SMPECs.}\footnote{\fyyy{In some of the literature on stochastic programming, this class of problems is also known as {\em one-stage} SMPEC.  However, inspired by this paper~\cite{rockafellar17stochastic} and for expository reasons, we have adopted {\em single-stage} SMPEC.}}} {Single-stage MPECs capture a class of stochastic MPECs with constraints given by parametrized  variational inequality problems {with} expectation-valued maps}. Such problems assume relevance in modeling a range of
stochastic equilibrium problems; more specifically, such problems represent the
necessary and sufficient equilibrium conditions of smooth stochastic convex
optimization problems and smooth stochastic convex Nash equilibrium
problems~\cite{Houyuan08,Nem11}. \fyyy{They can also be employed for modeling settings in power systems~\cite{fang2015coupon,baringo2013strategic}, structural optimization~\cite{evgrafov2003stochastic}, and transportation science~\cite{patriksson2008applicability,migdalas1998multilevel}.} More formally, suppose the
variational inequality problem VI$(\Yscr,F(\x,\bullet))$ is characterized by a
map $F$ whose components are expectation-valued, i.e.
\begin{align}
    \label{def-SF}
    F(\x,\y) \triangleq \pmat{ \mathbb{E}[G_1(\x,\y,\xi(\omega))] \\
                            \vdots \\
                        \mathbb{E}[G_m(\x,\y,\xi(\omega))]}, 
\end{align}
where {$G_i: \Real^n \times {\Real^m} \times \Real^d \to \Real$ and $\xi: \Omega\to\mathbb{R}^d$ denotes a random variable associated with the probability space $(\Omega, \mathcal{F},\mathbb{P})$}. {Note that the expectations in \eqref{def-SF} are taken with respect to the probability distribution $\mathbb{P}$.} {For the ease of presentation, throughout the paper, we} refer to the integrand $G_i(\x,\y, \xi(\omega){)}$ by $G_i(\x,\y, \omega)$. {In effect, the lower-level problem is a stochastic variational inequality problem~\cite{Houyuan08,FarzadAngeliaUday_MathProg17}}. In addition, the objective may also be expectation-valued and the \usv{pessimistic version of the}  resulting problem is defined as follows. 
\begin{align}\label{eqn:SMPECepx}\tag{SMPEC$^{\bf 1s}$}
\begin{aligned}
    \min_{\x,\y} & \quad f(\x,\y)\triangleq  \mathbb{E}[{\tilde f}(\x, \y,\omega)] \\
    \st & \quad \y \in \mbox{SOL}(\Yscr, \mathbb{E}[G(\x, \bullet, \omega)]), \\
     & \quad \x \in \Xscr\fy{.}
\end{aligned}
\end{align}
{An instance where \fyy{\eqref{eqn:SMPECepx}} emerges is when the lower-level equilibrium problem {captures} the equilibrium conditions of a convex stochastic optimization problem given by 
\begin{align}
    \min_{\y \in \Yscr} \,  \mathbb{E}[h(\x,\y,\omega)], 
\end{align}
where $F(\x,\y) \triangleq \mathbb{E}[\nabla_\y h(\x,\y,\omega)].$ {A more general instance is when {a solution to the} lower-level equilibrium problem {is a Nash equilibrium of a noncooperative game with expectation-valued objectives}, as given by 
\begin{align}
    \min_{\y_i \in \Yscr_i} \,  \mathbb{E}[h_i(\x,(\y_i;\y_{-i}),\omega)], 
\end{align}
where $i \in \{1,\ldots,N\}$, $N$ denotes the number of players, $\y_i \in \Yscr_i$ and $h_i(\x,(\bullet;\y_{-i}),\omega)$ denote the strategy set and the cost function of player $i \in \{1,\ldots,N\}$, respectively, and $\y_{-i}$ denotes the strategies of the other players than player $i$. Under some mild conditions, it is known that the equilibrium conditions of the aforementioned game can be characterized as VI$(\Yscr,F(\x,\bullet))$ where $\Yscr \triangleq\prod_{i=1}^N \Yscr_i$ and $F(\x,\y)  \triangleq \prod_{i=1}^N \mathbb{E}[\nabla_{\y_i}h_i(\x,(\y_i;\y_{-i}),\omega)]$ (cf. Chap. 1 in \cite{facchinei02finite}).} 

{An alternate approach for modeling uncertainty in MPECs is provided in the
next \fy{model},} where the lower-level problem constraints are
imposed in an almost sure (a.s.) sense~\cite{demiguel09stochastic}.

\smallskip

\noindent {(ii)  {\bf Two-stage SMPECs.}  Two-stage stochastic  MPECs are characterized by  equilibrium constraints VI$(\Yscr, F(\x, \bullet, \omega))$ for almost every $\omega \in \Omega$.} {We provide motivation by considering the following two-stage leader-follower game in which the follower makes  a {\em second-stage} decision} $\y$ contingent on the
leader's {decision} $\x$ and the realization of uncertainty is denoted by
$\omega$. Consequently, the leader's {first-stage}  problem requires minimizing her expected
cost $\mathbb{E}[\tilde{f}(\x,\y(\omega),\omega)]$ where $\y(\omega)$
represents follower's {second-stage (i.e. recourse)} decision, given $\x$ and $\omega$. \usv{A pessimistic version of this problem} can be
compactly represented as \fyy{\eqref{eqn:a_s_prob}}, defined next. 
\begin{align}\tag{SMPEC$^{\bf 2s}$}\label{eqn:a_s_prob}
\begin{aligned}
    \min_{\x, \y(\omega)} & \quad \mathbb{E}[\tilde f(\x, \y(\omega),\omega)] \\
    \st & \quad \y(\omega) \in \mbox{SOL}(\Yscr{(\x,\omega)}, {G(\x, \bullet,\omega)}), \mbox{ for almost every } \omega \in \Omega \\
     & \quad \x \in \Xscr. 
\end{aligned}
\end{align}
{In regimes where }VI$(\Yscr(\x,\omega), G(\x,\bullet,\omega))$ has a unique solution for any $\x \in \Xscr$ and any $\omega \in \Omega$, \usv{the pessimistic and optimistic versions of the SMPECs coincide and} we may recast \fyy{\eqref{eqn:a_s_prob}} as the following {\it implicit} stochastic optimization problem where $\y: \Xscr \times \Omega \to \Real^m$ denotes a single-valued solution map \usv{of VI$(\Yscr, F(\x,\bullet,\omega))$}. 
\begin{align}\tag{SMPEC$^{\bf imp,2s}$}\fy{\label{eqn:a_s_prob_imp}}
\begin{aligned}
    \min_{\x} & \quad   {f^{{\bf imp}}(\x)}  \triangleq  \mathbb{E}[\fyy{ \tilde f}(\x, \y(\x,\omega),\omega)] \\
     \st & \quad \x \in \Xscr. 
\end{aligned}
\end{align}
{The implicit counterpart of (SMPEC$^{\bf 1s}$), denoted by (SMPEC$^{\bf imp,1s}$), is defined analogously. }

\smallskip 
\noindent {\bf \large 1.2. Gaps and Contributions.} {\usv{The lower-level parametrized variational inequality problem can often be recast as a parametrized complementarity problem (e.g. when the VI admits a suitable regularity condition~\cite{luo96mathematical}). The MPEC then reduces to a mathematical program with complementarity constraints (MPCC).} 
Nonlinear programming (NLP) approaches aligned around sequential quadratic
    programming~\cite{fletcher02local} and interior-point schemes~\cite{leyffer2006interior,ragunathan05interior,anitescu-solving}  have been applied for resolving MPCCs (See~\cite{luo96mathematical} for a survey). \usv{This represents a dominant algorithmic thread for resolving MPECs while a second lies in implicit programming approaches~\cite{implicit1,implicit2,implicit3,implicit4,implicit5,implicit6,luo96mathematical}. Yet, there are some key shortcomings of such avenues in such regimes, motivating the present research.}}

\smallskip 

\smallskip

\noindent (a) {\em Limited convergence guarantees for existing NLP/regularization/penalization
schemes.} \usv{Most interior-point~\cite{leyffer2006interior,ragunathan05interior,anitescu-solving}, sequential quadratic programming (SQP)~\cite{fletcher02local}, and penalization/regularization
schemes~\cite{anitescu-solving,leyffer2006interior,demiguel2005two}} for
resolving MPECs are characterized by convergence {to}
strong-stationary or C-stationary points in the full space of upper and
lower-level decisions with rate guarantees {only} available in a local sense. \usv{Such schemes do not leverage any convexity properties in obtaining stronger guarantees.}  {In particular, there appear to be no efficient schemes that can provide convergence guarantees to global minimizers (\usv{in an implicit sense}) in either deterministic or stochastic regimes.} 

\smallskip

\noindent (b) {\em Implementability concerns with existing implicit approaches.} \usv{Existing implicit programming approaches (cf.~\cite{implicit1,implicit2,implicit3,implicit4,implicit5,implicit6,implicit7}) require exact resolution of the lower-level problem (precluding the resolution of lower-level stochastic variational inequality problems), can generally not accommodate uncertainty in their lower/upper-level, and are not equipped with non-asymptotic rate and complexity guarantees, particularly when the implicit problem is nonconvex.} 

\smallskip

\noindent (c) {\em Lack of efficient first/zeroth-order schemes.} While there has been a
tremendous amount of advances in providing non-asymptotic rate guarantees for
efficient first/zeroth-order algorithms for convex and {nonconvex} 
optimization problems~\cite{nesterov1998introductory,ConnScheVice09,flaxman2005online,ghadimi15,nesterov17}, the
resolution of MPECs via such avenues has been largely ignored. In fact, we are
unaware of any efficient first/zeroth-order scheme for deterministic MPECs even
under strong monotonicity assumptions at the lower-level.    

\smallskip 

\noindent (d) {\em Lack of scalability and convergence of schemes for stochastic MPECs.}
Sample-average
approximation~\cite{shapiro08stochastic,chen15regularized,liu11convergence} and
smoothing schemes~\cite{lin09solving} for \fyy{\eqref{eqn:a_s_prob}} have been studied
extensively. While SAA schemes provide an avenue for approximation, the SAA
problems become increasingly difficult to solve since the number of constraints
grows linearly with the sample-size. Absent such sampling, then such avenues
can generally contend with finite sample-spaces. However, no efficient
stochastic approximation schemes are available for contending with the
stochastic analogs.  

\medskip 

Collectively, these gaps motivate the development of tools and techniques for
this challenging class of stochastic nonconvex problems. To this end, {we
    develop a zeroth-order algorithmic framework equipped with convergence rate
    guarantees that is applied on the implicit formulation of the problem. In
    the implicit formulation, the objective {function} is viewed as a function in terms of the variable $\x$. While the implicit
    programming approach has been utilized before~\cite{luo96mathematical,lin09solving,xu06implicit}, several challenges arise
    when considering the development of iterative solution methods: (i) a
    closed-form characterization for $\y(\bullet)$ \fyy{(or $\y(\bullet,\omega)$)} is possibly unavailable
    which in turn, {precludes} the applicability of the standard first-order
    schemes; (ii) the implicit function is possibly nondifferentiable and
    nonconvex in $\x$ which complicates the convergence analysis and, in
particular, {the derivation of} rate statements. {In fact, one cannot compute subgradients or Clarke generalized gradients easily in such settings;} (iii) in inexact regimes where
there is lack of access to an oracle for computing $\y(\bullet)$ {(or $\y(\bullet,\omega)$)}, 
    standard zeroth-order methods may not be directly applied.  This is
    primarily because an inexact value of $\y(\bullet)$ may lead to a biased
    zeroth-order gradient approximation for the implicit function and the level
    of bias may even grow undesirably, as the parameters are updated
    iteratively; (iv) {finally, in settings {where the implicit problem is convex}, asymptotically convergent accelerated schemes with rate statements are unavailable.}

\begin{table}
\begin{center}{\scriptsize 
            \caption{Complexity guarantees for solving single-stage {SMPECs}\label{table:1s}}
\begin{tabular}{ |l| l||c|c|c|c|} 
\hline
\multicolumn{2}{|c||}{Single-stage SMPECs} & \multicolumn{2}{c|}{Convex implicit} & \multicolumn{2}{c|}{Nonconvex implicit}\\
\multicolumn{2}{|c||}{}  & Inexact & Exact &   Inexact & Exact\\
\hline
\multirow{2}{2em}{Upper level} & $\#$ projections & $n^4L_0^2\tilde{L}_0^4\epsilon^{-2}$  & $n^2L_0^2 \epsilon^{-2}$  &           $n^2L_0^2\tilde{L}_0^2\epsilon^{-1}$ & $n^2L_0^2\epsilon^{-1}$\\ 
& $\#$ samples  & $n^4L_0^2\tilde{L}_0^4\epsilon^{-2}$& $n^2L_0^2 \epsilon^{-2}$  &            $n^4L_0^4\tilde{L}_0^4\epsilon^{-2}$ & $n^4L_0^4\epsilon^{-2}$ \\ 
\hline
\multirow{2}{2.2em}{Lower level} & $\#$ projections & $n^4L_0^2\tilde{L}_0^4\epsilon^{-2} \ln\left( n^2L_0\tilde{L}_0^2\epsilon^{-1} \right)$ &  
--  &                $n^6L_0^6\tilde{L}_0^6\epsilon^{-3}$ & --\\ 
&$\#$ samples  & $  n^{4\bar \tau}L_0^{2\bar \tau}\tilde{L}_0^{4\bar \tau}\epsilon^{-2\bar \tau}$ & --                                    & $n^6L_0^6\tilde{L}_0^6\epsilon^{-3}$ &-- \\ 
\hline
\end{tabular}}
\end{center}
\end{table}
\begin{table}
\begin{center}{\scriptsize 
            \caption{Complexity guarantees for solving two-stage {SMPECs}\label{table:2s}}
\begin{tabular}{ |l| l||c|c|c|c|c|c|} 
\hline
\multicolumn{2}{|c||}{Two-stage SMPECs} & \multicolumn{3}{c|}{Convex implicit} & \multicolumn{2}{c|}{Nonconvex implicit}\\
\multicolumn{2}{|c||}{}  & Inexact & Exact & Accelerated &  Inexact & Exact\\
\hline
\multirow{2}{2em}{Upper level} & $\#$ projections & $n^4L_0^2{\tilde{L}_0^4}\epsilon^{-2}  $ &  $n^2L_0^2\epsilon^{-2}$  &  $\epsilon^{-1}$ &               $n^2L_0^2\tilde{L}_0^2\epsilon^{-1}$ & $n^2L_0^2\epsilon^{-1}$\\ 
& $\#$ samples  & $n^4L_0^2{\tilde{L}_0^4}\epsilon^{-2} $& $n^2L_0^2 \epsilon^{-2}$  &    $\epsilon^{-(2+\delta)}$ & $n^4L_0^4{\tilde{L}^4_0} \epsilon^{-2}$ & $n^4L_0^4\epsilon^{-2}$ \\ 
\hline
{Lower level} & {$\#$ projections} & $n^4L_0^2\tilde{L}_0^4\epsilon^{-2}  \ln\left( n^2L_0\tilde{L}_0^2\epsilon^{-1}  \right)$ &  
--  &     --  &            $n^4L_0^4{\tilde{L}^4_0}\epsilon^{-2}\ln(n^2L_0^2{\tilde{L}^2_0}\epsilon^{-1})$ & --\\ 
\hline
\end{tabular}}
\end{center}
\end{table}
    \medskip 
    \noindent {\bf Contributions.} In this paper, we aim at addressing these challenges through
    the development of a locally randomized zeroth-order scheme where the
    gradient of the implicit function is approximated at perturbed and possibly
    inexact evaluations of $\y(\bullet)$ \fyy{(single-stage) and $\y(\bullet,\omega)$ (two-stage)}. {Tables~\ref{table:1s} and~\ref{table:2s} provide the new complexity statements derived in this work for single-stage and two-stage SMPECs, respectively. The contributions in different regimes are as follows}.

{\noindent {\bf (1) Single-stage SMPECs.} We consider the single-stage problem~\eqref{eqn:SMPECepx} in Section~\ref{sec:1s}.

    \smallskip 
    \noindent {\bf (1-i)} {\em Inexact convex settings:} We develop (\texttt{ZSOL$^{\bf 1s}_{\rm cnvx}$}), {defined in} Algorithm~\ref{algorithm:inexact_zeroth_order_SVIs} where we employ a zeroth-order {method} for minimizing the implicit function. In the inexact variant of this method, to solve the stochastic VI at the lower-level and approximate $\y(\bullet)$, we employ a variance-reduced stochastic
    approximation method presented by Algorithm~\ref{algorithm:inexact_lower_level_SA}. In Theorem
    \ref{thm:ZSOL_convex}, we derive non-asymptotic convergence rates and also
    obtain an overall iteration complexity of
    $\mathcal{O}\left(n^4L_0^2\tilde{L}_0^4\epsilon^{-2}  \right)$ and
    $\mathcal{O}\left(n^4L_0^2\tilde{L}_0^4\epsilon^{-2} \ln\left( n^2L_0\tilde{L}_0^2\epsilon^{-1} \right)\right)$ for the projections on the set $\Xscr$ and $\Yscr$,
    respectively, where $L_0$ and $\tilde{L}_0$ are defined by Assumption~\ref{ass-1}. Importantly, both the stepsize and smoothing parameters are updated iteratively using prescribed rules allowing for establishing convergence to an optimal solution of the original single-stage SMPEC. 

    \noindent {\bf(1-ii)} {\em Exact convex settings:} The convergence statements for the exact variant of (\texttt{ZSOL$^{\bf 1s}_{\rm cnvx}$}) are provided in Corollary~\ref{thm:convex_exact}. In particular, we derive the iteration complexity of
    $\mathcal{O}\left(n^2L_0^2\epsilon^{-2}\right)$. This implies that to obtain an $\epsilon$-solution,
    the number of oracle calls to the solution of the lower-level variational inequality problem is at most $\mathcal{O}\left(n^2L_0^2\epsilon^{-2}\right)$.

    \noindent {\bf (1-iii)} {\em Inexact nonconvex settings:} In the case where the implicit
    function is nonconvex, we develop (\texttt{ZSOL$^{\bf 1s}_{\rm ncvx}$}), \uss{defined in} Algorithm~\ref{algorithm:ZSOL_nonconvex}. We analyze the convergence properties of this
zeroth-order scheme under a constant stepsize and smoothing parameter. In
Theorem \ref{thm:inexact_nonconvex}, to obtain an $\epsilon$-solution (characterized by mean norm-squared of a residual mapping) to the smoothed approximate SMPEC, we derive non-asymptotic convergence rates
for solving the smoothed implicit problem and obtain an overall iteration complexity of $\mathcal{O}\left(n^2L_0^2\tilde{L}_0^2\epsilon^{-1}\right)$ and $\mathcal{O}\left(n^4L_0^4\tilde{L}_0^4\epsilon^{-2}\right)$ for the projections on the set $\Xscr$ and $\Yscr$,
respectively.

\noindent {\bf (1-iv)} {\em Exact nonconvex settings:} In Corollary \ref{cor:exact_nonconvex} we provide the results for the exact variant of (\texttt{ZSOL$^{\bf 1s}_{\rm ncvx}$}). To obtain an $\epsilon$-solution (characterized by mean norm-squared of a residual mapping), we derive the iteration complexity of $\mathcal{O}\left(n^2L_0^2\epsilon^{-1}\right)$ for solving the smoothed approximate SMPEC. The number of oracle calls to the solution of the lower-level variational inequality problem is at most $\mathcal{O}\left(n^4L_0^4\epsilon^{-2}\right)$.

\medskip 

\noindent {\bf (2) Two-stage SMPECs.} We consider the two-stage problem~\eqref{eqn:a_s_prob} in Section~\ref{sec:2s}.

\noindent {\bf (2-i)} {\em Inexact convex settings}: We \uss{present} (\texttt{ZSOL$^{\bf 2s}_{\rm cnvx}$}), \uss{defined in} Algorithm~\ref{algorithm:inexact_zeroth_order_SVIs_2S}, for addressing two-stage SMPECs with a convex implicit objective function. In Theorem~\ref{thm:ZSOL_convex_2s}, for the inexact setting, we derive an overall iteration complexity of
    $\mathcal{O}\left(n^4L_0^2\tilde{L}_0^4\epsilon^{-2}  \right)$ and
    $\mathcal{O}\left(n^4L_0^2\tilde{L}_0^4\epsilon^{-2} \ln\left( n^2L_0\tilde{L}_0^2\epsilon^{-1} \right)\right)$ for the projections on the set $\Xscr$ and $\Yscr$,
    respectively. These \uss{statements} are similar to those obtained in the single-stage model. However, unlike in the single-stage case, the inexact variant of (\texttt{ZSOL$^{\bf 2s}_{\rm cnvx}$}) does not require any new samples in solving the lower-level problem, i.e., in Algorithm~\ref{algorithm:Inexact_lower_level_SA_twostage}, a \uss{parametrized} deterministic variational inequality problem is solved. 

    \noindent {\bf (2-ii)} {\em Exact convex settings}: In Corollary~\ref{cor:exact_nonconvex_2s}, we provide the iteration complexity of
    $\mathcal{O}\left(n^2L_0^2\epsilon^{-2}\right)$, similar to that of the single-stage counterpart. This implies that the number of oracle calls to the solution of the lower-level variational inequality problem is at most $\mathcal{O}\left(n^2L_0^2\epsilon^{-2}\right)$.
    
    \noindent {\bf (2-ii-a)} {\em Accelerated exact convex settings}: We develop  a variance-reduced accelerated zeroth-order scheme called (\texttt{ZSOL$^{\bf 2s}_{\rm cnvx, acc}$}), \uss{formally specified} by Algorithm~\ref{algorithm:acc-inexact_zeroth_order_SVIs}. In Proposition~\ref{prop:acc_convex_exact}, we improve the complexity to $\mathcal{O}(1/\epsilon)$ in terms of upper-level projection steps while the number of lower-level variational inequality problems is no worse than $\mathcal{O}(1/\epsilon^{2+\delta})$ for $\delta > 0$. 

    \noindent {\bf (2-iii)} {\em Inexact nonconvex settings}: In addressing two-stage models with a nonconvex implicit objective function, we develop (\texttt{ZSOL$^{\bf 2s}_{\rm ncnvx}$}), \uss{a variance-reduced} zeroth-order method. This scheme is presented by Algorithm~\ref{algorithm:ZSOL_nonconvex_2s}. In Theorem~\ref{thm:inexact_nonconvex_2s} we obtain non-asymptotic convergence rates
for solving the smoothed implicit problem and derive an overall iteration complexity of $\mathcal{O}\left(n^2L_0^2\tilde{L}_0^2\epsilon^{-1}\right)$ and $\mathcal{O}\left(n^4L_0^4\tilde{L}_0^4\epsilon^{-2}\right)$ for the projections on the set $\Xscr$ and $\Yscr$, respectively. These results are similar to those we obtained for the single-stage counterpart. However, in computing an approximate $\y(\bullet,\omega)$ in the lower-level problem in Algorithm~\ref{algorithm:Inexact_lower_level_SA_twostage}, unlike in the single-stage regime, we solve a deterministic variational inequality problem.

\noindent {\bf (2-iv)} {\em Exact nonconvex settings:} Lastly, in Corollary~\ref{cor:exact_nonconvex_2s}, we consider the exact variant of (\texttt{ZSOL$^{\bf 2s}_{\rm ncnvx}$}). Similar to the single-stage case, to obtain an $\epsilon$-solution (characterized by mean norm-squared of a residual mapping),  we derive the iteration complexity of $\mathcal{O}\left(n^2L_0^2\epsilon^{-1}\right)$ for solving the smoothed approximate SMPEC. The number of oracle calls to the solution of the lower-level variational inequality problem is at most $\mathcal{O}\left(n^4L_0^4\epsilon^{-2}\right)$. 

\smallskip

\noindent {\bf (3)} {\bf Comprehensive numerics.} \usv{In Section~\ref{sec:5},
we provide a comprehensive set of numerics where we provide empirical support
for the scalability and convergence claims for inexact schemes for single and
two-stage SMPECs. Such investigations also suggest the limited scalability of
SAA schemes as well as the ability of the proposed schemes to compute
near-global solutions under convexity of the implicit problems, in contrast with their SAA counterparts. Finally, the benefits of acceleration in terms of accuracy is observed as promised by theoretical claims.} 

\smallskip

To the best of our knowledge, all the above-mentioned rate and
complexity results in addressing both the single-stage and two-stage
SMPECs appear to be novel.}  

\noindent {\bf Notation.} {Throughout, we use the following notation and
definitions. We let $\Xscr^*$ and $f^*$ denote the optimal solution set and the
optimal objective value of a corresponding implicit problem, respectively.} We
define $D_\Xscr   \triangleq \frac{1}{2}\sup_{\x \in
\Xscr}\textrm{dist}^2(\x,\Xscr^*)$. We let $\mathbb{B}$ denote the unit
ball defined as $\mathbb{B}\triangleq \{u \in  \mathbb{R}^n\mid \|u\| \leq 1\}$
and $\mathbb{S}$ denote the surface of the ball $\mathbb{B}$, i.e.,
$\mathbb{S}\triangleq \{v \in \mathbb{R}^n\mid \|v\| = 1\}$. Given a set $\Xscr
\subseteq \mathbb{R}^n$ and a scalar $\eta>0$, we let $\Xscr_\eta$ denote the
expanded set $\Xscr +\eta \mathbb{B}$. Given a function $f:\mathbb{R}^n \to
\mathbb{R}$ and a set $\Xscr \subseteq \mathbb{R}^n$, we write $f \in
C^{0,0}(\Xscr)$ if $f$ is \fyy{Lipschitz} continuous on the set $\Xscr$,
i.e., $|f(\x)-f(\tilde \x)| \leq L_0\|\x-\tilde \x\|$ for all $\x,\tilde \x \in
\Xscr$ and some $L_0>0$.  In the case where $f$ is globally Lipschitz, i.e.,
$\Xscr = \mathbb{R}^n$, we write $f \in C^{0,0}$.  Given a continuously
differentiable function and a set $\Xscr \subseteq \mathbb{R}^n$, we write $f
\in C^{1,1}(\Xscr)$ if $\nabla f$ is \fyy{Lipschitz} continuous on the set
$\Xscr$, i.e., $\|\nabla f(\x)-\nabla f(\tilde \x)\| \leq L_1\|\x-\tilde \x\|$
for all $\x,\tilde \x \in \Xscr$ and some $L_1>0$.  Similarly,  we write $f \in
C^{1,1}$ to denote that $\nabla f$ is globally Lipschitz. {We  denote the
Euclidean projection of a vector $\x$ on a set $\Xscr$ {by 
$\Pi_\Xscr(\x)$}, i.e., $\|\x-\Pi_\Xscr(\x)\| = \min_{\bar{\x} \in
\Xscr}\|\x-\bar{\x}\|$. {Throughout, unless otherwise specified, for the ease
of presentation we use $\mathbb{E}[\bullet]$ to denote the expectation with
respect to all the random variables under discussion. We use conditional
expectations to specifically take expectations with respect to a subgroup of
random variables.}}

\section{Preliminaries}
In this section, {we begin by outlining the key assumptions imposed on \fyy{\eqref{eqn:SMPECepx}} and \fyy{\eqref{eqn:a_s_prob}}
in Section~\ref{sec:2.1}}. {Our treatment and analysis differ based on whether}  the implicit function {$f^{\bf imp} $} is
either convex or nonconvex. In the latter case, the resulting problem reduces
to a nonsmooth nonconvex program with possibly expectation-valued objectives.
In such settings, we provide a brief discussion of stationarity conditions in
Section~\ref{sec:2.2} {while a discussion of locally randomized spherical
smoothing techniques is presented in Section~\ref{sec:2.3}.}

\subsection{Problem definition}\label{sec:2.1}
Throughout this paper, we assume that in the case of \fyy{\eqref{eqn:SMPECepx}}, the set $\Yscr$ is closed and convex in $\Real^m$ and the parametrized map $F(\x,\bullet)$ is {strongly} monotone on $\Yscr$ uniformly in $\x$. An analogous assumption for \fyy{\eqref{eqn:a_s_prob}} requires that $G(\x,\bullet,\omega)$ is strongly monotone on $\Yscr$ for every $\omega \in \Omega$. Since {the} lower-level problem is strongly monotone, the solution map of the lower-level problem is single-valued. Consequently, we may recast \fyy{\eqref{eqn:a_s_prob}} as the following implicit program in $\x$.
\begin{align}\label{prob:mpec_as_imp}\tag{SMPEC$^{\bf imp, 2s}$} 
    \min_{\x \in \Xscr} \ f^{\bf imp}(\x) \triangleq \mathbb{E}[\tilde{f}(\x,\y(\x,\omega),\omega)], 
\end{align}
where ${f^{\bf imp}}(\bullet)$ is assumed to be Lipschitz continuous on a closed and convex set $\Xscr$. {Note that such a property on $f^{\bf imp}$ holds if $f^{\bf imp}$ is locally Lipschitz on a compact set}. In the case of (SMPEC$^{\bf 1s})$, the \fy{implicit} problem reduces to   
\begin{align}\label{prob:mpec_exp_imp}\tag{SMPEC$^{\bf imp, 1s}$} 
    \min_{\x \in \Xscr} \ f^{\bf imp}(\x)  \triangleq \mathbb{E}[\tilde{f}(\x,\y(\x),\omega)], 
\end{align}
where $\y(\x)$ represents the solution to a variational inequality problem VI$(\Yscr, F(\x,\bullet))$. {Note that this problem subsumes \eqref{eqn:SMPECepx} by suppressing the expectation in the upper-level.} We now formalize the assumptions {on the problems of interest}.


\begin{assumption}[{\bf Properties of $f, F, \Xscr, \Yscr$}]\z 
\label{ass-1}\em 
\noindent {(a)} Consider the problem {\eqref{prob:mpec_exp_imp}}.

\noindent (a.i) {$\tilde{f}(\bullet,\y(\bullet),\omega)$ is $L_0(\omega)$-Lipschitz continuous on $\Xscr+\eta_0 \mathbb{B}$ for every $\omega \in \Omega$ and for some $\eta_0 > 0$, where $L_0\triangleq \sqrt{\mathbb{E}[L_0^2(\omega)]}<\infty$.  Also, $\fyy{\tilde f}(\x,\bullet,\omega)$ is $\tilde{L}_0(\omega)$-Lipschitz for all $\x \in \Xscr+\eta_0 \mathbb{B}$ for every $\omega \in \Omega$ and for some $\eta_0 > 0$, where $\tilde{L}_0\triangleq \sqrt{\mathbb{E}[\tilde{L}_0^2(\omega)]}<\infty$.}
    
\noindent (a.ii) $\Xscr \subseteq \Real^n$ and $\Yscr \subseteq \Real^m$ are \fy{nonempty,} closed, \fyyy{bounded}, and convex sets.

\noindent (a.iii) $F(\x,\bullet)$ is a $\mu_F$-strongly monotone and $L_F$-Lipschitz continuous map on $\Yscr$ uniformly in $\x \in \Xscr$.
\\  

\noindent \fy{(b)} Consider the problem {\eqref{prob:mpec_as_imp}}.

\noindent (b.i) {$\tilde{f}(\bullet,\y(\bullet,\omega),\omega)$ is $L_0(\omega)$-Lipschitz continuous on $\Xscr+\eta_0 \mathbb{B}$ for every $\omega \in \Omega$ and for some $\eta_0 > 0$, where $L_0\triangleq \sqrt{\mathbb{E}[L_0^2(\omega)]}<\infty$. Also, $\fyy{\tilde f}(\x,\bullet,\omega)$ is $\tilde{L}_0(\omega)$-Lipschitz for all $\x \in \Xscr+\eta_0 \mathbb{B}$ for every $\omega \in \Omega$ and for some $\eta_0 > 0$, where $\tilde{L}_0\triangleq \sqrt{\mathbb{E}[\tilde{L}_0^2(\omega)]}<\infty$.}
    
\noindent (b.ii) $\Xscr \subseteq \Real^n$ and $\Yscr \subseteq \Real^m$ are \fyyy{nonempty,} closed, \fyyy{bounded}, and convex sets.

\noindent (b.iii) $G(\x,\bullet,\omega)$ is a {$\mu_F(\omega)$-strongly monotone and $L_F(\omega)$-Lipschitz continuous map on $\Yscr$ uniformly in $\x \in \Xscr$ for every $\omega \in \Omega$, and there exist scalars $\mu_F,L_F \in (0,+\infty)$ such that $\inf_{\omega \in \Omega}\mu_F(\omega) \geq  \mu_F$ and $\sup_{\omega \in \Omega}L_F(\omega) \leq L_F$.}
\qed  

\end{assumption}
{\begin{remark} \em
        As outlined in Assumption~\ref{ass-1}, throughout we assume that the {mapping in the lower-level parametrized by $\x$} is strongly monotone { on $\Yscr$ uniformly in $\x$}. The assumption is inherent to most implicit methods for resolving MPECs and our
proposed schemes inherit that characteristic. When considering sample-average
approximation schemes in the context of SMPECs, we observe that similar assumptions have been adopted in a subset of prior work including~\cite{shapiro2006,xuye11,lin09solving}. \fyy{In fact, lower-level uniqueness is by no means a rarely seen phenomenon. It is inherent to a host of problems in practice~\cite{sherali83stackelberg,murphy82mathematical,su2007analysis,demiguel09stochastic} and there is a significant body of research on implicit methods for solving MPECs in a range of settings~\cite{implicit1,implicit2,implicit3,implicit4,implicit5,implicit6,implicit7}.} In the current work, we intend to assess the fundamental gaps on the performance under a requirement on lower-level uniqueness but we allow for far more generality in the lower-level problem (e.g., in terms of accommodating expectation-valued maps) and either convexity or nonconvexity in terms of the upper-level problem.
\end{remark}}
We observe that the requirement that $f$ is Lipschitz continuous on $\Xscr + \eta_0 \mathbb{B}$ (rather than $\Xscr$) is a consequence of employing a smoothed approximation of $f$ in our algorithm development. A natural question is whether the Lipschitz continuity of the objective $f$ over $\Xscr$ in the implicit problem  follows under reasonable conditions. The next result {addresses} precisely such a concern. 

\begin{proposition}\em  Consider the problem {\eqref{eqn:SMPECepx}}. {Let} Assumption~\ref{ass-1} {(a.ii, a.iii)} hold. Suppose $\tilde{f}(\bullet,\bullet,\omega)$ is continuously differentiable on $\Cscr \times \Real^m$ where $\Cscr$ is an open set containing $\Xscr$. Then the function {$f^{\bf imp}$}, defined as ${f^{{\bf imp}}(\x)} \triangleq \mathbb{E}[\tilde{f}(\x,{\y(\x)}, \omega)]$, is Lipschitz and directionally differentiable on $\Xscr$.
\end{proposition}

\begin{proof} This result follows from invoking {~\cite[Cor.~4.2]{patriksson99stochastic}} together with the compactness of $\Xscr$. 
\end{proof}

\begin{proposition}\em  Consider the problem {\eqref{eqn:a_s_prob}}. {Let} Assumption~\ref{ass-1} {(b.ii, b.iii)} hold. Suppose $\tilde{f}(\bullet,\bullet,\omega)$ is continuously differentiable on $\Cscr \times \Real^m$ where $\Cscr$ is an open set containing $\Xscr$. Then the function {$f^{\bf imp}$}, defined as ${f^{{\bf imp}}(\x)} \triangleq \mathbb{E}[\tilde{f}(\x,\y(\x,\omega), \omega)]$, is Lipschitz and directionally differentiable on $\Xscr$.
\end{proposition}

\begin{proof} This result follows from invoking ~\cite[Cor.~4.3]{patriksson99stochastic} together with the compactness of $\Xscr$. 
\end{proof}
In a subset of regimes, {$f^{{\bf imp}}$} is captured by the next assumption.

\begin{assumption}[{\bf Convexity of $f$ in implicit problem}] \label{ass-2} \em {Consider any of the implicit problems \fy{\eqref{prob:mpec_as_imp} or \eqref{prob:mpec_exp_imp}}. {Then the implicit function $f^{{\bf imp}}$} is convex on $\Xscr$.} 
\end{assumption}

We note that  there has been extensive study of conditions under which the implicit function {$f^{{\bf imp}}$}  is indeed convex {(for example, see~\cite{patriksson99stochastic,xu06implicit,demiguel09stochastic})}.
{In fact,  the convexity of the implicit function can be proven in MPECs arising in a host of  application-driven regime~\cite{sherali83stackelberg,sherali84multiple,su07analysis,xu06implicit,demiguel09stochastic}, there appear to be no explicit conditions to the best of our knowledge.}

\subsection{Stationarity conditions}\label{sec:2.2}
While {the implicit function $f^{\bf imp}$} can be shown to be convex in some {specific settings}, 
the function {$f^{\bf imp}$} is Lipschitz continuous on $\Xscr$ in more general settings. Consequently, the problem can be compactly stated as 
\begin{align} \label{Opt} 
    \min_{\x \in \Xscr} \ h(\x)\triangleq f^{{\bf imp}}(\x) \fy{.}
\end{align}
We observe that $h$ is a nonsmooth and
possibly nonconvex function on $\Xscr$. In the remainder of
this subsection, we recap some of the concepts of Clarke's nonsmooth calculus
that will facilitate the development of stationarity conditions. We begin by defining the directional derivative, a key object necessary in addressing nonsmooth and possibly nonconvex optimization problems. 
\begin{definition}[cf. \cite{clarke98}] \em 
    The directional derivative of $h$ at $\x$ in a direction $v$ is defined as 
    \begin{align}
        h^{\circ}(\x,v) \triangleq  \limsup_{\y \to \x, t \downarrow 0} \left(\frac{h(\y+tv)-h(\y)}{t}\right).
\end{align}
The Clarke generalized gradient at $\x$ can then be defined as 
\begin{align}
    \partial h(\x) \triangleq  \left\{ \fyyy{\zeta} \in \Real^n \mid h^{\circ}(\x,v) \geq \langle \fyyy{\zeta}, v\rangle, \quad \forall v \in \Real^n\right\}.
\end{align}
In other words, $h^{\circ}(\x,v) = \displaystyle \sup_{g \in \partial h(\x)} \langle g,v\rangle.$\qed
\end{definition}

\medskip

 If $h$ is continuously differentiable at $\x$, we have that the Clarke generalized gradient reduces to the standard gradient, i.e. $\partial h(\x) = \nabla_{\x} h(\x).$ If $\x$ is a minimal point of $h$, then we have that $0 \in \partial h(\x)$. For purposes of completeness, we recap some properties of $\partial h(\x)$. Recall that if $h$ is locally Lipschitz on an open set $\Cscr$ containing $\Xscr$, then $h$ is differentiable almost everywhere on $\Cscr$ by Rademacher's theorem~\cite{clarke98}. Suppose $\Cscr_h$ denotes the set of points where $h$ is not differentiable. We may then recall some properties of Clarke generalized gradients.   
 \begin{proposition}[Properties of Clarke generalized gradients~\cite{clarke98}] \em
Suppose $h$ is Lipschitz continuous on $\Real^n$. Then the following hold.
\begin{enumerate}
\item[(i)] $\partial h(\x)$ is a nonempty, convex, and compact  set and $\|g \| \leq L$ for any $g \in \partial h(\x)$. 
\item[(ii)] $h$ is differentiable almost everywhere. 
\item[(iii)] $\partial h(\x)$ is an upper semicontinuous map defined as 
    $$\partial h(\x) = \mbox{conv}\left\{g \mid g = \lim_{k \to \infty} \nabla_{\x} h(\x_k), \Cscr_h \not \owns \x_k \to \x\right\}.$$
\end{enumerate}
\end{proposition}
We may also define the \fyyy{$\delta$}-generalized gradient~\cite{goldstein77} as
\begin{align}
    \partial_\fyyy{\delta} h(\x) \triangleq \mbox{conv}\left\{ \fyyy{\zeta}: \fyyy{\zeta} \in \partial h(\y), \|\x-\y\| \leq \fyyy{\delta}\right\}.
\end{align}
Under the assumption that $h$ is globally bounded from below and Lipschitz
continuous on $\Xscr$, our interest \fy{in the nonconvex regimes} lies in developing techniques for computing
an {\em approximate} stationary point. For instance, when $h$ is $L$-smooth,
then computing an approximate stationary point in unconstrained regimes such
that $\|\nabla_{\x} h(\x)\| \leq \epsilon$ requires at most
$\mathcal{O}(1/\epsilon^2)$ gradient steps. Much of the prior work in the
computation of stationary points of nonconvex and nonsmooth functions is either
asymptotic~\cite{burke05robust,chen12smoothing} or relies on some structure~\fyyy{\cite{beck17fom,xu2019stochastic,liu2019successive}} where the nonconvex part
is smooth while the convex part may be {closed and proper}. However, the
question of computing approximate stationary points for functions that are both
nonconvex and nonsmooth has been less studied.  

\subsection{Properties of spherical \fy{smoothing} of $f$}\label{sec:2.3}

We consider an iterative smoothing approach in this paper where a smoothed
approximation of $h$ is minimized and the smoothing parameter is progressively
reduced. This avenue has a long history, beginning with the efforts by
Steklov~\cite{steklov1} leading to significant efforts in both convex~\cite{DeFarias08,Farzad1,Duchi12} and
nonconvex~\cite{nesterov17} regimes. In this paper, we consider the
following smoothing of $h$, given by $h_{\eta}$ where \begin{align} 
    h_{\eta}(\x) \triangleq  \mathbb{E}_{u \in \mathbb{B}}[h(\x+\eta u)], \label{def-smooth}
\end{align} where $u$ is a random vector in the unit ball $\mathbb{B}$, defined as $\mathbb{B}\triangleq \{u \in  \mathbb{R}^n\mid \|u\| \leq 1\}$.  Throughout, we let $\mathbb{S}$ denote the surface of the ball $\mathbb{B}$, i.e., $\mathbb{S}\triangleq \{v \in \mathbb{R}^n\mid \|v\| = 1\}$. We also let $\eta\mathbb{B}$ and $\eta\mathbb{S}$ denote the ball with radius $\eta$ and its surface, respectively. Recall that if $h$ is locally Lipschitz over a compact set $\Xscr$, it is globally Lipschitz on $\Xscr$. We may derive the following properties on $h_{\eta}$.  

\begin{lemma}[{\bf Properties of spherical smoothing\footnote{We note that while spherical \fy{smoothing} have apparently been studied in~\cite{nemirovskij1983problem}, we did not have access to this text. Part (i) of our {lemma} is inspired by Flaxman et al.~\cite{flaxman2005online} while other parts either  follow in  a fashion similar to Gaussian smoothing~\cite{nesterov17} or are directly proven.}}] \label{lemma:props_local_smoothing}\em Suppose $h:\mathbb{R}^n \to \mathbb{R}$ is {a continuous function and $\eta>0$ is a given scalar}.  Let $h_{\eta}$ be defined as \eqref{def-smooth}. Then the following hold.

    \noindent (i) The smoothed function $h_{\eta}$ is continuously differentiable over $\Xscr$. In particular,  for any $\x \in \Xscr$, we have that 
    \begin{align}
        \nabla_{\x} h_{\eta}(\x) = \left(\tfrac{n}{\eta}\right) \mathbb{E}_{v \in \eta \mathbb{S}} \left[h(\x+v) \tfrac{v}{\|v\|}\right]. 
    \end{align}
    {Suppose $h \in C^{0,0}(\Xscr_{\eta})$ with parameter $L_0$.} For any $\x, \y \in \Xscr$, we have that (ii) -- (iv) hold. 

\begin{enumerate}
    \item[ (ii)] $| h_{\eta}(\x)-h_{\eta}(\y) | \leq L_0 \|\x-\y\|.$ 

    \item[ (iii)] $| h_{\eta}(\x) - h(\x)| \leq L_0 \eta.$ 

    \item[ (iv)] $\| \nabla_{\x} h_{\eta}(\x) -\nabla_{\x} h_{\eta}(\us{\y})\| \leq 
\tfrac{L_0n}{\eta}\|\x-\y\|.$ 
    \item[(v)] If $h$ is convex {and $h \in C^{0,0}(\Xscr_{\eta})$ with parameter $L_0$}, then $h_{\eta}$ is convex and satisfies the following for any $\x \in \Xscr$. 
        \begin{align}
                h(\x) \leq h_{\eta}(\x) \leq h(\x) + \eta L_0.
            \end{align} 
    \item[(vi)] If $h$ is convex {and $h \in C^{0,0}(\Xscr_{\eta})$ with parameter $L_0$}, then $\nabla_x h_{\eta}(\x) \in \partial_{\fyyy{\delta}} h(\x)$ where {$\fyyy{\delta} \triangleq \eta L_0$}.
        

    \item[(vii)] If {$h \in C^{1,1}(\Xscr_{\eta})$} with constant {$L_1$}, then $\|\nabla_x h_{\eta}(\x)-\nabla_x h(\x) \| \leq {\eta L_1 n}.$ 
 
    \item [(viii)] {Suppose $h \in C^{0,0}(\Xscr_{\eta})$ with parameter $L_0$. Let us define  for $v \in \eta\mathbb{S}$
\begin{align*}
g_{\eta}(\x,v) \triangleq \left(\tfrac{n}{\eta}\right) \tfrac{(h(\x+v) - h(\x))v}{\|v\|}.
\end{align*}
Then, for} any $\x \in \Xscr$, {we have that} ${\mathbb{E}_{v \in \eta \mathbb{S}}}[\|{g_\eta}(\x,v)\|^2] \leq L_0^2 n^2$. 
\end{enumerate}
\end{lemma}
\begin{proof} (i) We elaborate on the proof sketch provided in~\cite{flaxman2005online}. By definition, we have that 
    \begin{align*}
        h_{\eta}(\x) =   \mathbb{E}_{u \in \eta \mathbb{B}}[h(\x+u)] =\int_{\eta \mathbb{B}} h(\x+u) p(u) du.
    \end{align*}
Let $p(u)$ denote the probability density function of $u$. Since $u$ is uniformly distributed in the ball $\eta \mathbb{B}$, we have that $p(u) = \tfrac{1}{\mbox{Vol}(\eta \mathbb{B})}$ for any $u \in \eta \mathbb{B}$. Consequently, 
\begin{align*}
h_{\eta}(\x) =    \int_{\eta \mathbb{B}} h(\x+u) p(u) du = \frac{\int_{\eta \mathbb{B}} h(\x+u)  du}{\mbox{Vol$_n$}(\eta \mathbb{B})}.    
\end{align*}
We may then compute the derivative $\nabla_{\x} h_{\eta}(\x)$ by leveraging Stoke's theorem and by defining $\tilde{p}(v) = 
\tfrac{1}{\mbox{Vol$_{n-1}$}(\eta \mathbb{S})}$ for all $v$.
\begin{align*}
    \nabla_{\x} h_{\eta}(\x)&  =   \nabla_{\x} \left[ \frac{\int_{\eta \mathbb{B}} h(\x+u)  du}{\mbox{Vol$_{n}$}(\eta \mathbb{B})}\right] 
     \overset{\tiny \mbox{Stoke's theorem}}{=} 
\left[ \frac{\int_{\eta \mathbb{S}} h(\x+v) \tfrac{v}{\|v\|} dv}{\mbox{Vol$_{n}$}(\eta \mathbb{B})}\right] 
     = \left[ \frac{\int_{\eta \mathbb{S}} h(\x+v) \tfrac{v}{\|v\|} dv}{\mbox{Vol$_{n}$}(\eta \mathbb{B})}\right] \frac{\mbox{Vol$_{n-1}$}(\eta \mathbb{S})}{\mbox{Vol$_{n-1}$}(\eta \mathbb{S})} \\
    & = \left[ \frac{\int_{\eta \mathbb{S}} h(\x+v) \tfrac{v}{\|v\|} dv}{\mbox{Vol$_{n-1}$}(\eta \mathbb{S})}\right] \frac{\mbox{Vol$_{n-1}$}(\eta \mathbb{S})}{\mbox{Vol$_{n}$}(\eta \mathbb{B})} 
     =  \left[   \int_{\eta \mathbb{S}}h(\x+v) \tfrac{v}{\|v\|} \tilde{p}(v) dv\right] \frac{n}{\eta} 
     =  \frac{n}{\eta} \mathbb{E}_{ v \in \eta\mathbb{S}} \left[ h(\x+v) \tfrac{v}{\|v\|} \right]. 
        \end{align*}

    \noindent (ii) {We have
    \begin{align*}
        | h_{\eta}(\x)-h_{\eta}(\y) | &= \left| \mathbb{E}_{u \in \mathbb{B}}[h(\x+\eta u)] -  \mathbb{E}_{u \in \mathbb{B}}[h(\y+\eta u)]\right| \overset{\tiny \mbox{Jensen's ineq.}}{\leq} \mathbb{E}_{u \in \mathbb{B}}[|h(\x+\eta u) -h(\y+\eta u) |] \\
                                      &\overset{\tiny h \in C^{0,0}(\fy{\Xscr_\eta})}{\leq} \mathbb{E}_{u \in \mathbb{B}}[L_0 \|\x-\y\|] 
        = L_0 \|\x-\y\|.
    \end{align*}

}
    \noindent (iii) Next, we show that $|h_{\eta}(\x)-h(\x)|$ can be bounded in terms of $\eta$ and $L_0$. 
    \begin{align*}
        | h_{\eta}(\x) - h(\x)| & = \left| \int_{\eta \mathbb{B}} (h(\x+u) - h(\x))p(u) du \right| \\
                        & \leq \int_{\eta \mathbb{B}} \left| (h(\x+u) - h(\x))\right| p(u) du  \\
                        & \leq L_0 \int_{\eta \mathbb{B}} \|u\|  p(u) du  
                         \leq L_0 \eta \int_{\eta \mathbb{B}} p(u) du = L_0\eta. 
    \end{align*}

    \noindent (iv) 
{Note that we have $\Xscr+\eta\mathbb{S} \subseteq \Xscr+\eta\mathbb{B}$. Thus,  from the definition of $\Xscr_\eta$ and $h \in C^{0,0}(\Xscr_\eta)$, we have $h \in C^{0,0}(\Xscr+\eta\mathbb{S})$.  As such, we have}
    \begin{align*}
         \left\| \nabla_{\x} h_{\eta}(\x) - \nabla_{\x} h_{\eta}(\y)\right \| 
        & =  \left\| \tfrac{n}{\eta} \mathbb{E}_{ v \in \eta\mathbb{S}} \left[ h(\x+v) \tfrac{v}{\|v\|} \right] - \tfrac{n}{\eta} \mathbb{E}_{ v \in \mathbb{S}} \left[ h(\y+v)\tfrac{v}{\|v\|} \right] \right\| \\
        & \leq \tfrac{n}{\eta}  \mathbb{E}_{ v \in \eta\mathbb{S}} \left[ \left\|(h(\x+v) - h(\y+v)) \tfrac{v}{\|v\|}  \right\|\right]
        \\
        & \leq \tfrac{L_0n}{\eta} \|\x-\y\|  \mathbb{E}_{ v \in\eta \mathbb{S}} \left[\tfrac{\|v\|}{\| v \|}\right] 
         = \tfrac{L_0n}{\eta}\|\x-\y\| .  
    \end{align*}
  
    \noindent   {  (v)  First, note that from $h \in C^{0,0}(\Xscr_{\eta})$,  we have that $h \in C^{0,0}(\mbox{int}(\Xscr_\eta))$. Noting that $\mbox{int}(\Xscr_\eta)$ is an open set, from part (b) of Theorem 3.61 in~\cite{beck17fom}, we have that $\|\tilde g\|\leq L_0$ for all $\x \in \mbox{int}(\Xscr_\eta)$ and $\tilde g \in \partial h(\x)$. The desired statements then follow from part (a) and part (b) of Lemma 2~\cite{yousefian10convex}.}

 \noindent   {(vi) From part (v), function $h_\eta$ is convex and $h(\y) + \eta {L_0} \geq h_{\eta}(\y)$ for any $\y \in \Xscr$. Thus, for all $\x,\y \in \Xscr$ we have
   \begin{align*}
   h(\y) + \eta {L_0} \geq h_{\eta}(\y) \geq  h_{\eta}(\x) +\nabla h_\eta(\x)^T(\y-\x) \geq h(\x) +\nabla h_\eta(\x)^T(\y-\x).
\end{align*}   }   
\noindent   {(vii) Note that we can show that $\int_{\eta \mathbb{S}}vv^T  p_v(v) dv = \tfrac{\eta^2}{n}\mathbf{I}$. We may then express $\nabla_x h(x)$ as 
  \begin{align*}
  \nabla_x h(\x) &=\tfrac{n}{\eta^2} \left(\int_{\eta \mathbb{S}}vv^T  p_v(v) dv\right)  \nabla_x h(\x)  = \tfrac{n}{\eta^2} \left(\int_{\eta \mathbb{S}}v^T\nabla_x h(\x)  v p_v(v) dv\right)\\
  &= \tfrac{n}{\eta} \left(\int_{\eta \mathbb{S}}v^T\nabla_x h(\x)  \tfrac{v}{\|v\|} p_v(v) dv\right)=\tfrac{n}{\eta}\mathbb{E}_{ v \in \eta\mathbb{S}} \left[ \left(\nabla_x h(\x)^Tv\right)\tfrac{v}{\|v\|} \right],
  \end{align*}
  {where the third inequality follows from $\|v\| = \eta$ for $v\in \eta\mathbb{S}$.}
  From this relation, part (i), and by recalling that $\tfrac{n}{\eta} \mathbb{E}_{ v \in \eta\mathbb{S}} \left[ h(\x)\tfrac{v}{\|v\|} \right]=0$, we can write 
  \begin{align*}
  \| \nabla_x h_{\eta}(\x)-\nabla_x h(\x)\| &= \left\|\tfrac{n}{\eta} \mathbb{E}_{ v \in \eta\mathbb{S}} \left[ \left(h(\x+v)-h(\x) \right)\tfrac{v}{\|v\|} \right]-\tfrac{n}{\eta}\mathbb{E}_{ v \in \eta\mathbb{S}} \left[ \left(\nabla h(\x)^Tv\right)\tfrac{v}{\|v\|} \right] \right\|\\
  & \leq \tfrac{n}{\eta} \mathbb{E}_{ v \in \eta\mathbb{S}} \left[\left|h(\x+v)-h(\x)-\nabla h(\x)^Tv \right|\tfrac{\|v\|}{\|v\|} \right]\\
    & \leq \tfrac{n}{\eta} \mathbb{E}_{ v \in \eta\mathbb{S}} \left[L_1\|v\|^2\right] = n\eta L_1.
\end{align*}    }     
   
    \noindent (viii) We observe that for any $\x$, ${\mathbb{E}_{v \in \eta \mathbb{S}}}[\|{g_{\eta}}(\x,v)\|^2]$ may be bounded as follows. 
    \begin{align*}
{\mathbb{E}_{v \in \eta \mathbb{S}}}[\|{g_{\eta}}(\x,v)\|^2] &  = \frac{n^2}{\eta^2} \int_{\eta \mathbb{S}} \frac{\| (h(\x+v)-h(\x))v\|^2}{\|v\|^2}  p_v(v) dv \\
            & \leq \frac{n^2}{\eta^2} \int_{\eta \mathbb{S}} L_0^2 \| v\|^2  p_v(v) dv \leq n^2 \int_{\eta \mathbb{S}} p_v(v) dv = n^2 L_0^2.
    \end{align*}

    \end{proof} 

\begin{remark}[Local vs global smoothing] \em Gaussian smoothing as employed
    in~\cite{nesterov17} allows for unbounded random variables as part of the
    smoothing process. However, this precludes contending with compact regimes
    which we may require to impose Lipschitzian assumptions. Furthermore, in
    many settings, the domain of the function is compact and Gaussian smoothing
    cannot be adopted. Instead, local smoothing requires that the smoothing
    random variable have compact support. In~\cite{yousefian10convex,Farzad1},
    we examine smoothing schemes based on random variables defined on a cube or a
    sphere. However, most of the results of the previous \fyy{lemma} are novel with
    respect to~\cite{yousefian10convex}. 
\end{remark}

    We intend to develop schemes for computing approximate stationary points of \eqref{Opt} by an iterative smoothing scheme. However, this needs formalizing the relationship between the original problem and its smoothed counterpart. Before proceeding, we define $\fyyy{\delta}$-Clarke generalized gradient of $h$, denoted by $\partial_{\fyyy{\delta}} h(\x)$ at $\x$, as follows~\cite{goldstein77}.
    \begin{align}
        \partial_{\fyyy{\delta}} h(\x) \triangleq \mbox{conv} \left\{ \fyyy{\zeta} \mid  \fyyy{\zeta} \in \partial h(\y), \|\y - \x\| \leq \fyyy{\delta} \right\}. 
    \end{align} 
    It was first shown by Goldstein~\cite{goldstein77} that $\partial_{\fyyy{\delta}} h(\x)$ is nonempty, compact, and convex set. 
    \begin{proposition}\label{prop_equiv} \em 
        Consider the problem \eqref{Opt} where $h$ is a locally Lipschitz continuous function and $\Xscr$ is a closed, convex, and bounded set in $\Real^n$. 


        \noindent (i) For any $\eta > 0$ and any $\x \in \Real^n$, $\nabla h_{\eta}(\x) \in \partial_{2\eta} h(\x)$. Furthermore, if $0 \not \in \partial h(\x)$, then there exists an $\eta$ such that $\nabla_{\x} h_{\tilde \eta} (\x) \neq 0$ for $\tilde{\eta} \in (0,\eta]$.    

        \noindent (ii) For any $\eta > 0$ and any $\x \in \Xscr$, 
        \begin{align}
            \left[ 0 \in \nabla_{\x} h_{\eta}(\x) + \mathcal{N}_{\Xscr}(\x) \right] \  \implies \ \left[ 0 \in \partial_{2\eta} h(\x)+ \mathcal{N}_{\Xscr}(\x)\right].
        \end{align}
    \end{proposition}
    \begin{proof} (i) and (ii) represent a constrained counterparts of \cite[Prop.~2.2 and Cor.~2.1]{mayne84}.
    \end{proof}

Lemma~\ref{lemma:props_local_smoothing} (v) provides a statement that relates the true objective
to its smoothed counterpart in convex regimes. This provides an avenue for
developing finite-time schemes for computing approximate solutions to the {\em
original problem}. \fyy{Prop.~\ref{prop_equiv} (ii)} provides a relationship in
settings where $h$ is locally Lipschitz; in particular, it is shown that if
$\x$ satisfies stationarity of the $\eta$-smoothed problem, it satisfies a
suitable $2\eta-$stationarity property for the original problem.

\section{Zeroth-order methods for single-stage {SMPECs}}\label{sec:1s}
In this section, we present a zeroth-order framework for contending
{with \eqref{prob:mpec_exp_imp}.} The remainder of this section is organized as
follows. In Section~\ref{sec:3.1}, we introduce an implicit zeroth-order scheme
that can {allow} for \fyyy{constructing a smoothed zeroth-order gradient through leveraging} inexact solutions of the lower-level problem. To address settings
where the implicit problem is convex, we derive rate and complexity guarantees
for an iteratively smoothed gradient framework in Section~\ref{sec:3.2} when
the lower-level problem is either inexactly or exactly resolved. \fyyy{In these settings, the smoothing parameter is progressively reduced at each iteration.} {Lastly in Section~\ref{sec:nonconvex_single_stage}, we
derive iteration complexity in addressing the nonconvex case \fyyy{under a constant smoothing parameter}.}

\subsection{An implicit zeroth-order scheme}\label{sec:3.1}
Since the {implicit} function is merely Lipschitz continuous, we employ a zeroth-order framework that relies on computing a zeroth-order approximation of the gradient. Consider the implicit problem \eqref{prob:mpec_exp_imp}. G}iven the function ${f^{\bf imp}}$ {and a scalar $\eta$, we consider a spherical smoothing {denoted by $f^{\bf imp}_{\eta}$} based on \eqref{def-smooth}, defined {as}
\begin{align}\label{eqn:G-Smooth}\tag{G-Smooth{$^{\bf 1s}$}}
  { f^{\bf imp}_{\eta}(\x) \triangleq \mathbb{E}_{u\in \mathbb{B}}[f^{\bf imp}(\x+\eta u)]= \mathbb{E}_{u\in \mathbb{B}}[\mathbb{E}[\tilde{f}(\x+\eta u,\y(\x+\eta u),\omega)]]},
\end{align}
where $u$ is uniformly distributed in {the unit} ball $\mathbb{B}$.} Let $g_{\eta}(\x)$ denote a zeroth-order approximation of the  gradient {of ${f^{\bf imp}_{\eta}(\x)}$}. {Invoking Lemma \ref{lemma:props_local_smoothing},} one choice for $g_{\eta}$ {is given as follows for any $\x$}. 
\begin{align}\fy{\label{eqn:g_eta}}
    g_{\eta}(\x) = \left( \frac{n}{\eta} \right)\mathbb{E}_{v \in \eta \mathbb{S}}\left[\frac{\left(f^{\bf imp}(\x+v) - f^{\bf imp}(\x)\right)v}{\|v\|} \right]. 
\end{align} 
{In general}, {given the presence of the expectation}, $g_{\eta}(\x)$ is challenging to evaluate and a common approach has been in utilizing an unbiased estimate given by ${g_{\eta}(\x,v,\omega)}$ defined as 
\begin{align}\label{eqn:g_v_eta}
 {g_{\eta}(\x,v,\omega)} \triangleq \left( \frac{n}{\eta} \right)\left[\frac{\left({\tilde f}(\x+v, \y(\x+v),\omega) -{\tilde f}(\x, \y(\x),\omega)\right)v}{\|v\|} \right]. 
\end{align}
Given a vector $\x_0 \in \Xscr$, we may employ \eqref{eqn:g_v_eta} in 
constructing a sequence $\{\x_k\}$ where $\x_k$ satisfies the following projected stochastic gradient update. 
\begin{align}\label{eqn:fixed_smoothing_scheme}
    \x_{k+1} := \Pi_{\Xscr} \left[ \x_k - \gamma_k {g}_{\eta}(\x_k,v_k,\omega_k) \right].
\end{align}

{Motivated by the development of the stochastic approximation (SA) scheme~\cite{robbins51sa}, the projected stochastic gradient and gradient-free schemes have been studied extensively in convex and nonconvex regimes (e.g., see~\cite{nemirovski_robust_2009,Farzad1,ghadimi13zeroth,ghadimi15} and the references therein). Recall that in the SA schemes, the standard requirements on the stepsize sequence include $\sum_{k=0}^\infty \gamma_k = \infty$ and $\sum_{k=0}^\infty \gamma_k^2 < \infty$.} The scheme \eqref{eqn:fixed_smoothing_scheme} has been studied for addressing
    nonsmooth convex and nonconvex optimization problems~\cite{nesterov17}
    while unconstrained nonconvex regimes were examined
    in~\cite{ghadimi13zeroth}. In particular, in the work by Nesterov and
    Spokoiny~\cite{nesterov17}, zeroth-order randomized smoothing gradient
    schemes are proposed under a single sample with a fixed smoothing parameter
    $\eta$ with the assumption that the smoothing random variable $v$ has a
Gaussian distribution. Importantly, a direct adoption of such smoothing schemes
to address the hierarchical problems studied in this work is afflicted by
several challenges.

\medskip

\noindent (i) {\em Lack of asymptotic guarantees.} When $\eta > 0$, the scheme generates a sequence that is convergent to an approximate solution, at best. In addition, the choice of $\eta$ is contingent on accurate estimates of other problem parameters (such as $L_0$), in the absence of which, $\eta$ may be chosen to be extremely small. This often afflicts the practical behavior of the scheme. Moreover, employing a fixed $\eta$ precludes asymptotic convergence to the true counterpart. Instead, in most of our schemes, we employ a mini-batch approximation of $g_{\eta}(\x)$, denoted by $g_{\eta, N}(\x)$ and defined as  
\begin{align}\label{eqn:g_eta_and_N}
 g_{\eta,N}(\x) \triangleq \frac{\sum_{j=1}^N g_{\eta}(\x,v_j, \omega_j)}{N}. 
\end{align}
Furthermore, we replace a fixed $\eta$ by a diminishing sequence $\{\eta_k\}$, the resulting iterative smoothing scheme being articulated as follows. 
\begin{align}\label{eqn:method_iter_smooth}
    \x_{k+1} := \Pi_{\Xscr} \left[ \x_k - \gamma_k {g}_{\eta_k,N_k}(\x_k) \right].
\end{align}
\noindent (ii) {\em Unavailability of exact solutions of $\y(\x)$.}  Even if $\y(\bullet)$ is a single-valued map requiring the solution of a strongly monotone lower-level problem, computing a solution to this problem is not necessarily cheap. As a consequence, our scheme needs to account for \fyy{random} errors in the computation of $g_{\eta_k}(\x_k)$, denoted by $\fyyy{\tilde{b}_k}$. As a consequence, the resulting scheme is defined as follows.    
\begin{align}\label{eqn:inexact_smoothing_scheme}
    \x_{k+1} := \Pi_{\Xscr} \left[ \x_k - \gamma_k ({g}_{\eta_k,N_k}(\x_k)+\fyyy{\tilde{b}_k}) \right], {\qquad} \fyy{\hbox{for all }k\geq 0}.
\end{align}
{In particular}, when considering problems (SMPEC$^{\bf 1s}$), exact solutions of $\y(\x_k)$ are generally unavailable in finite time. Instead, one can take \fyy{$t_k$} steps of a standard projection scheme. 
\fyy{\begin{align} \label{upd-y}
    \y_{t+1} := \Pi_{\Yscr} \left[ \y_t - \beta_t \bar{F}(\x_k, \y_t)\right], {\qquad}  t = 0, \cdots, \fyy{t_k}-1,
\end{align}
where $\bar{F}(\x_k,\y_t) \triangleq \tfrac{\sum_{\ell=1}^{{M}_t} {G}(\x_k, \y_t, \fyy{\omega_{\ell,t}})}{{M}_t}$}. \fyy{In such} a variance-reduced scheme, \fyy{when $M_t$ grows at a geometric rate}, $\log(1/\epsilon_k)$ steps of \fyy{\eqref{upd-y}} are required to obtain an $\epsilon_k$-solution of $\y_k$\fyy{~\cite{iusem19variance}}. 

\smallskip 

\noindent (iii) {\em Bias in $\fyyy{\tilde{b}_k}$.} A key issue that arises
from (ii) emerges in the form of bias. In particular,
${g}_{\eta_k,N_k}(\x_k)+\fyyy{\tilde{b}_k}$ is
not necessarily an unbiased estimator of ${g}_{\eta_k}(\x_k)$. Further, it
remains unclear how the bias and variance of ${g}_{\eta_k,N_k}(\x_k)+\fyyy{\tilde{b}_k}$ 
propagate through this framework
\eqref{eqn:inexact_smoothing_scheme}-\eqref{upd-y} as $\gamma_k$, $\eta_k$, and
$N_k$ are updated iteratively in the outer loop
\eqref{eqn:inexact_smoothing_scheme}. Consequently, in the development of
the inexact smoothing scheme
\eqref{eqn:inexact_smoothing_scheme}-\eqref{upd-y}, it remains critical to
design prescribed stepsize, smoothing, and sample-size sequences to control the
accuracy of the estimator ${g}_{\eta_k,N_k}(\x_k)+\fyyy{\tilde{b}_k}$ and consequently,
ascertain the convergence of the generated iterate to an optimal solution of {the underlying MPEC}. This concern will be examined
in detail in the subsequent sections.  
\subsection{Convex single-stage regimes}\label{sec:3.2} 
{In this  subsection, we consider resolving the {implicit formulations when the
    implicit function} is convex.  \fy{As pointed out earlier,} the convexity of the implicit
problem often holds in practice
(cf.~\cite{patriksson99stochastic,xu06implicit,demiguel09stochastic}). We first consider the inexact case where the exact value of $\y(\bullet)$
is not necessarily available.  We then specialize our statements to settings
where exact solutions of lower-level problems can be employed.  

\subsubsection{An inexact zeroth-order scheme}
We now delve into developing and analyzing an inexact zeroth-order method for
resolving the implicit variant {\eqref{prob:mpec_exp_imp}.} We begin {by} providing the general setup and assumptions. Then, we provide some key results and algorithms.  Before proceeding, we consider the following assumption.
\begin{assumption}\label{assum:u_iter_smooth}\em
    Given a sequence $\{\eta_k\}$, let ${\{v_k\}} \in \mathbb{R}^n$ be {iid replicates uniformly distributed on} $\eta_k\mathbb{S}$ for all $k\geq 0$. {Also, let $\{\omega_{k}\}$ be {iid replicates}.}
\end{assumption}
\fyy{\begin{remark}\em
Throughout the paper, for the ease of presentation, we assume that there exists an oracle that returns the replicates of $\omega$ in the upper-level. The function $\tilde{f}(\bullet,\bullet, \omega)$ can then be evaluated using a second oracle. Note that this assumption is without loss of any generality and an alternative approach is to assume that there exists an oracle that generates the random realizations of the function $\tilde{f}(\bullet,\bullet, \omega)$ directly. 
\end{remark}}
Consider the implicit form of} \fyy{\eqref{eqn:SMPECepx}}, i.e. \fy{\eqref{prob:mpec_exp_imp}} where
the lower-level problem is complicated by the presence of expectation-valued
maps, i.e., $F$ is defined as \eqref{def-SF} and satisfies
Assumption~\ref{ass-1} (a.iii).
In such an instance, obtaining $\y(\x)$ is impossible in finite time unless the expectation can be tractably resolved. Instead, by employing stochastic approximation \fyy{methods} for addressing the lover-level problem, we consider the case where we have access to an approximate solution $\y_\fyyy{\tilde{\epsilon}_k}(\x_k)$ such that \fyy{the following holds a.s.}
\begin{align}
    \mathbb{E}[\|\y_\fyyy{\tilde{\epsilon}_k} (\x_k) - \y(\x_k)\|^2\mid \x_k] \leq \fyyy{\tilde{\epsilon}_k}, \mbox{ where } \y(\x_k) \in \mbox{SOL}(\Yscr, \fy{F(\x_k,\bullet)}).  
\end{align}
As a consequence, we may define {an inexact zeroth-order gradient mapping} $g_{\eta,\fyyy{\tilde{\epsilon}}}(\x,{v}\fyy{,\omega})$ as follows.
{\begin{align}\label{def-g-eta-eps}
 {g_{\eta,\fyyy{\tilde{\epsilon}}}(\x,v,\omega)  \triangleq } \frac{n({\tilde f}(\x+ v, \y_{\fyyy{\tilde{\epsilon}}}(\x+ v),\omega) - {\tilde f}(\x, \y_{\fyyy{\tilde{\epsilon}}}(\x),\omega))v}{\|v\|\eta}, 
\end{align}
where {$v \in \eta\mathbb{S}$}  and $\y_\fyyy{\tilde{\epsilon}_k} (\x_k)$ is an output of a
variance-reduced stochastic approximation scheme.} {The outline of the proposed
zeroth-order solver (\fyy{\texttt{ZSOL$^{\bf 1s}_{\rm cnvx}$}}) is presented in
Algorithm~\ref{algorithm:inexact_zeroth_order_SVIs} while an inexact solution
of $\y(\x)$ is computed by Algorithm~\ref{algorithm:inexact_lower_level_SA}.
We impose the following assumptions \fyy{on} the lower-level evaluations
    $G(\hat \x_k,\y_t,\omega_{\ell,t})$ \fyy{in Algorithm~\ref{algorithm:inexact_lower_level_SA}}.

\begin{assumption}\label{assum:lower_level_stoch}\em
   \fyy{Consider Algorithm~\ref{algorithm:inexact_lower_level_SA}.} Let the following hold for all $k\geq 0$, $t \geq 0$,  $\hat \x_k \in \Xscr$, $\y_t \in \Yscr$, and $1\leq \ell\leq M_t$ \fyy{where $M_t$ denotes the batch size at iteration $t$}.

\noindent (a)  \fyyy{The replicates $\{G(\bullet,\bullet, \omega_{\ell,t})\}_{\ell=1}^{M_t}$ are generated randomly and are iid.}

\noindent (b)  $\mathbb{E}[G(\hat \x_k,\y_t,\omega_{\ell,t})\mid \hat \x_k,\y_t]  = F(\hat \x_k,\y_t)$ {holds almost surely}.

\noindent (c)  $\mathbb{E}[\|G(\hat \x_k,\y_t,\omega_{\ell,t})-F(\hat \x_k,\y_t)\|^2 \mid \hat \x_k,\y_t]  \leq \nu_{\y}^2\|\y_t\|^2+\nu_G^2$ {holds almost surely} for some {deterministic scalars} \fyyy{$\nu_\y\geq 0$ and $\nu_G>0$}.
\end{assumption}

\begin{algorithm}[H]
    \caption{\texttt{ZSOL$^{\bf 1s}_{\rm cnvx}$}: Zeroth-order method for convex  \fyy{\eqref{eqn:SMPECepx}}}\label{algorithm:inexact_zeroth_order_SVIs}
    \begin{algorithmic}[1]
        \STATE\textbf{input:} {Given $\x_0 \in \Xscr$, ${\bar \x}_0: = \x_0$,  stepsize sequence $\{\gamma_k\}$,  smoothing parameter sequence  $\{\eta_k\}$}, {inexactness sequence $\{\fyyy{\tilde{\epsilon}_k}\}$}, $r \in [0,1)$,  and $S_0 : = \gamma_0^r$
    \FOR {$k=0,1,\ldots,{K}-1$}
\STATE Generate { iid replicates} $\omega_{k} \in \Omega$ and  $v_{k} \in \eta_k \mathbb{S}$ 
    		    \STATE  {Do one of the following, depending on the type of the scheme.
    		    
\vspace{-.1in}    	
	    
\begin{itemize}
\item   Inexact scheme: Call Algorithm \ref{algorithm:inexact_lower_level_SA} twice to obtain $\y_\fyyy{\tilde{\epsilon}_k}(\x_k)$ and $\y_\fyyy{\tilde{\epsilon}_k}(\x_k+v_k)$ 
\vspace{-.15in}    	

\item   Exact scheme: Evaluate $\y(\x_k)$ and $\y(\x_k+v_k)$
\end{itemize}}

\vspace{-.1in}    	

    		 \STATE Evaluate the inexact {or exact zeroth-order gradient approximation as follows.
    		 \vspace{-.1in}    	
	    
         \begin{align}
             g_{\eta_k,\fyyy{\tilde{\epsilon}_k}}(\x_k,v_k,\omega_k) & :=\tfrac{n\left({\tilde f}(\x_k+ v_k, \y_\fyyy{\tilde{\epsilon}_k}(\x_k+ v_k),\omega_k) - \fyy{\tilde f} (\x_k, \y_\fyyy{\tilde{\epsilon}_k}(\x_k),\omega_k)\right)v_k}{\|v_k\|\eta_k} \tag{Inexact} \\
             g_{\eta_k}(\x_k,v_k,\omega_k) & :=\tfrac{n\left({\tilde f}(\x_k+ v_k, \y(\x_k+ v_k){,\omega_k}) - {\tilde f} (\x_k, \y(\x_k),\omega_k)\right)v_k}{\|v_k\|\eta_k}. \tag{Exact}
\end{align}}
\STATE {Update $\x_k$ as follows.}
{
        \begin{align*}
            \x_{k+1} :=  \begin{cases}
                \Pi_{\Xscr} \left[ \x_k - \gamma_k g_{\eta_k,\fyyy{\tilde{\epsilon}_k}}(\x_k,v_k,\omega_k) \right] &  \hspace{2.8in} \mbox{ (Inxact)}  \\
                \Pi_{\Xscr} \left[ \x_k - \gamma_k g_{\eta_k}(\x_k,v_k,\omega_k) \right] &   \hspace{2.85in}\mbox{ (Exact)} 
            \end{cases}
        \end{align*}
    }


  \STATE Update the averaged iterate as follows.  $ S_{k+1} : = S_k+\gamma_{k+1}^r$ and  $\bar \x_{k+1}:=\tfrac{S_k\bar \x_k+\gamma_{k+1}^r\x_{k+1}}{S_{k+1}}$
    \ENDFOR
   \end{algorithmic}
\end{algorithm}

\begin{algorithm}[H]
    \caption{{Variance-reduced SA method for lower-level of convex \fyy{\eqref{eqn:SMPECepx}}}}\label{algorithm:inexact_lower_level_SA}
    \begin{algorithmic}[1]
        \STATE \textbf{input:} An arbitrary $\y_0 \in \Yscr$, vector $\hat{\x}_k$ \fyy{(that is either $\x_k$ or $\x_k+v_k$ from Alg.~\ref{algorithm:inexact_zeroth_order_SVIs})}, scalar $\rho \in (0,1)$, {stepsize $\alpha>0$, mini-batch sequence $\{M_t\}$ \fyy{with $M_t := \lceil M_0\rho^{-t} \rceil$}, integer $k$, and \fyy{scalars $M_0,\tau>0$ (see Def.~\eqref{def:algo_1})} }
        \STATE Compute $t_k:=\lceil\tau \ln(k+1) \rceil $ 
      \FOR {$t=0,1,\ldots,t_k-1$}
      \STATE Generate random realizations of the stochastic mapping $G(\hat{\x}_k,{\y_t}, \omega_{\ell,t})$ for $\ell=1,\ldots,M_t$
      \STATE {Update $\y_t$ as follows.}  $
        \y_{t+1} := \Pi_{\Yscr}\left[\y_t - \alpha \tfrac{\sum_{\ell = 1}^{M_t} G(\hat{\x}_k,\y_t,\omega_{\ell, t})}{M_t}\right]$
        \ENDFOR
        \STATE Return $\y_{t_k}$ 
   \end{algorithmic}
\end{algorithm}

Before analyzing \fyy{(\texttt{ZSOL$^{\bf 1s}_{\rm cnvx}$})}, we review the properties of the  exact zeroth-order {stochastic} gradient denoted by $g_{\eta}(\x,v,\omega)$ and show that {it} is an unbiased estimator of the gradient of the smoothed implicit function. We then derive a bound on the second moment of this stochastic gradient under the assumption that the implicit {stochastic} function is Lipschitz. 
{\begin{remark}\label{rem:g_eta}
Throughout, we use the definition $
g_{\eta}(\x,v)\triangleq \left(\tfrac{n}{\eta}\right) \frac{\left(f^{\bf imp}(\x+v) - f^{\bf imp}(\x)\right)v}{\|v\|}$, where $f^{\bf imp}(\bullet)$ is the implicit function defined by \eqref{prob:mpec_exp_imp} or \eqref{prob:mpec_as_imp} . 
\end{remark}}
\begin{lemma}[\bf Properties of the single-stage exact zeroth-order gradient]\label{lem:smooth_grad_properties}\em
Suppose Assumption~\ref{ass-1} (a) holds.  Consider \eqref{prob:mpec_exp_imp}. Given $\x \in \Xscr$ and $\eta>0$,  {consider the stochastic zeroth-order mapping $g_{\eta}(\x,v,\omega)$ defined by \eqref{eqn:g_v_eta} for $v \in \eta\mathbb{S}$ and $k\geq 0$, where $v$ and $\omega$ are independent}.
 Then,  ${\nabla f^{\bf imp}_{\eta}(\x)} =\mathbb{E}[g_{\eta}(\x,v,\omega)\mid \x]$ and $\mathbb{E}[\|g_{\eta}(\x,v,\omega)\|^2\mid \x] \leq {L_0^2n^2}$ {almost surely} for all $k\geq 0$. 
\end{lemma}
\begin{proof}
{From \eqref{eqn:g_v_eta} and that $f^{\bf imp}(\x)  \triangleq \mathbb{E}[\tilde{f}(\x,\y(\x),\omega)]$ we can write 
\begin{align*}
\mathbb{E}[g_{\eta}(\x,v,\omega)\mid \x]&
=\mathbb{E}_{v \in \eta\mathbb{S}}\left[\left(\tfrac{n}{\eta}\right) \frac{\left(f^{\bf imp}(\x+v) - f^{\bf imp}(\x)\right)v}{\|v\|}\mid \x\right]\\
&=\left(\tfrac{n}{\eta}\right)\mathbb{E}_{v \in \eta\mathbb{S}}\left[ f^{\bf imp}(\x+v) \frac{ v}{\|v\|}\mid \x\right]\overset{\tiny \mbox{Lemma }\ref{lemma:props_local_smoothing} (i)}{=}  \nabla f^{\bf imp}_{\eta}(\x) .
\end{align*}
We have
\begin{align*}
\mathbb{E}[\|g_{\eta}(\x,v,\omega)\|^2\mid \x,\omega]  
 &=\left(\tfrac{n}{\eta}\right)^2\mathbb{E}\left[\left\| \tfrac{\left({\tilde f}(\x+v, \y(\x+v),\omega) - {\tilde f}(\x, \y(\x),\omega)\right)v}{\|v\|}\right\|^2\mid \x,\omega\right]  \\
 & =\left(\tfrac{n}{\eta}\right)^2 \int_{\eta \mathbb{S}} \tfrac{\left\| \left({\tilde f}(\x+v, \y(\x+v),\omega) - {\tilde f}(\x, \y(\x),\omega)\right)v\right\|^2}{\|v\|^2}  p_v(v) dv \\
 &\overset{\tiny \mbox{Assumption \ref{ass-1} (a.i)}}{\leq} \frac{n^2}{\eta^2} \int_{\eta \mathbb{S}} L_0^2(\omega) \| v\|^2  p_v(v) dv \leq n^2 {L_0^2(\omega)} \int_{\eta \mathbb{S}} p_v(v) dv = n^2 L_0^2(\omega).
\end{align*}
Taking {expectations} with respect to $\omega$ {on} both sides of the preceding inequality and invoking $L_0^2 \triangleq \mathbb{E}[L_0^2(\omega)]<\infty$, we obtain the desired bound. 
}
\end{proof}}
We are now ready to present the properties of the inexact zeroth-order gradient mapping. 
\begin{lemma}[\bf Properties of the single-stage inexact zeroth-order gradient]{\label{lem:inexact_error_bounds}}\em Consider \eqref{prob:mpec_exp_imp}.  Suppose {Assumption~\ref{ass-1} (a)} holds.  {Let} $g_{\eta,\fyyy{\tilde{\epsilon}}}(\x,v,\omega)$ be defined as \eqref{def-g-eta-eps} for $\omega \in \Omega$ and $v \in \eta\mathbb{S}$ for $\eta, \fyyy{\tilde{\epsilon}} >0$. Suppose $\mathbb{E}[\|\y_\fyyy{\tilde{\epsilon}}(\x)-\y(\x)\|^2 \mid \x,\omega] \leq \fyyy{\tilde{\epsilon}}$ almost surely for all $\x \in \Xscr$. Then, the following hold for the single-stage model {for any $\x \in \Xscr$}. 

    \noindent {\bf (a)}  $\mathbb{E}[\|g_{\eta,\fyyy{\tilde{\epsilon}}}(\x,v,\omega)\|^2\mid\x]  \leq  3n^2\left(\tfrac{2{\tilde{L}_0^2} \fyyy{\tilde{\epsilon}}   }{\eta^2} + {L_0^2}\right),$ {almost surely}. 

 \noindent {\bf (b)}  $ \mathbb{E}\left[ \left\| g_{\eta,\fyyy{\tilde{\epsilon}}}(\x,v,\omega) - g_{\eta}(\x,v,\omega) \right\|^2 \mid \x\right] 
 \leq   \frac{{4\tilde{L}^2_0n^2}\fyyy{\tilde{\epsilon}}}{\eta^2} $,{ almost surely}.

\end{lemma}
\begin{proof}  
\noindent {\bf (a)} Adding and subtracting $ g_{\eta}(\x,{v},\omega)$, we obtain from \eqref{def-g-eta-eps}
\begin{align*}
   & \quad  \|  g_{\eta,\fyyy{\tilde{\epsilon}}}(\x,{v},\omega)\| \\
   &  =  \left\| \frac{n({ \tilde f}(\x+{v}, \y_{\fyyy{\tilde{\epsilon}}}(\x+{v}){,\omega}) - { \tilde f}(\x+{v}, \y(\x+{v}),\omega)){v}}{{\|v\|}\eta}  +{g_{\eta}(\x,{v},\omega)} +\frac{n({ \tilde f}(\x, \y(\x),\omega) - { \tilde f}(\x, \y_{\fyyy{\tilde{\epsilon}}}(\x),\omega)){v}}{{\|v\|}\eta}\right\| \\
   & \leq  \left\| \frac{n({ \tilde f}(\x+{v}, \y_{\fyyy{\tilde{\epsilon}}}(\x+{v}),\omega) - { \tilde f}(\x+{v}, \y(\x+{v}),\omega)){v}}{{\|v\|}\eta}\right\| + \left\|{g_{\eta}(\x,{v},\omega)} \right\|  \\
   &+ \left\| \frac{n({ \tilde f}(\x, \y(\x),\omega) - { \tilde f}(\x, \y_{\fyyy{\tilde{\epsilon}}}(\x),\omega)){v}}{{\|v\|}\eta} \right\| \\
   & \leq    \frac{\| { \tilde f}(\x+{v}, \y_{\fyyy{\tilde{\epsilon}}}(\x+{v}),\omega) - { \tilde f}(\x+{v}, \y(\x+{v}),\omega)\|n \|v\|}{\|v\|\eta} + \left\|{g_{\eta}(\x,{v},\omega)}  \right\| \\
                               & +  \frac{\|{ \tilde f}(\x , \y(\x) ,\omega) - { \tilde f}(\x , \y_{\fyyy{\tilde{\epsilon}}}(\x ),\omega)\|n\|v\|}{\eta\|v\|}  \\
                               & \leq    \frac{{\tilde{L}_0(\omega)}\|\y_{\fyyy{\tilde{\epsilon}}}(\x+v) - \y(\x+v)\| n }{\eta}   + \left\|{g_{\eta}(\x,v,\omega)}  \right\|  + \frac{\tilde{L}_0(\omega)\|\y_{\fyyy{\tilde{\epsilon}}}(\x)-\y(\x)\| n}{\eta} . 
\end{align*} 
{Invoking Lemma \ref{lem:smooth_grad_properties}, we} may then bound the second moment of $\|g_{\eta,\fyyy{\tilde{\epsilon}}}(\x,v,\omega)\|$ as follows.
\begin{align}\label{eqn:lemma_props_of_g_eta_bound2}
    \mathbb{E}[\|{g_{\eta,\fyyy{\tilde{\epsilon}}}(\x,{v},\omega)}\|^2] & \leq   3\mathbb{E}\left[\left(\frac{{\tilde{L}_0^2(\omega){n^2}}\|\y_{\fyyy{\tilde{\epsilon}}}(\x+v) - \y(\x+v)\|^2  }{\eta^2}\right) \mid \x\right]  + 3\mathbb{E}\left[\left\|g_{\eta}(\x,v,\omega)\right\|^2\mid \x\right] \notag\\
                                                     & + 3\mathbb{E}\left[\left(\frac{\tilde{L}_0^2(\omega){n^2}\|\y_{\fyyy{\tilde{\epsilon}}}(\x)-\y(\x)\|^2 }{\eta^2}\right) \mid \x\right]  \leq  6\left(\frac{{\tilde{L}_0^2{n^2}} \fyyy{\tilde{\epsilon}}   }{\eta^2}\right) + {3L_0^2n^2}, \mbox{ a.s.} 
\end{align}

{\noindent {\bf (b)}} 
{We first} derive a bound on $\left\| { g_{\eta,\fyyy{\tilde{\epsilon}}}(\x,{v},\omega) - g_{\eta}(\x,{v},\omega) } \right\|$. 
\begin{align*}
 & \quad  \left\| { g_{\eta,\fyyy{\tilde{\epsilon}}}(\x,{v},\omega) - g_{\eta}(\x,{v},\omega) } \right\|
\\& = \left\| \frac{n({ \tilde f}(\x+{v}, \y_{\fyyy{\tilde{\epsilon}}}(\x+{v}),\omega) - \fyyy{{ \tilde f}(\x, \y_{\fyyy{\tilde{\epsilon}}}(\x),\omega))}{v}}{\|v\|\eta}  -   \frac{n({ \tilde f}(\x+{v}, \y(\x+{v}),\omega) - \fyyy{{ \tilde f}(\x, \y(\x),\omega))}{v}}{\|v\|\eta}  \right\| \\
  & \leq \left\| \frac{n({ \tilde f}(\x+v, \y_{\fyyy{\tilde{\epsilon}}}(\x+v),\omega) - { \tilde f}(\x+v, \y(\x+v ),\omega))v}{\|v\|\eta}  \right\| + \left\|  \frac{n({ \tilde f}(\x, \y_{\fyyy{\tilde{\epsilon}}}(\x),\omega) - { \tilde f}(\x, \y(\x),\omega))v}{\|v\|\eta}  \right\| \\
  & \leq   \frac{{\tilde{L}_0{n}}\| \y_{\fyyy{\tilde{\epsilon}}}(\x+v) - \y(\x+v)\|}{\eta}   +   \frac{{\tilde{L}_0{n}} \| \y_{\fyyy{\tilde{\epsilon}}}(\x) - \y(\x)\|}{\eta},  
  \end{align*}
\fyyy{where in the last inequality we use the definition of $\tilde L_0$ in Assumption~\ref{ass-1} (a.i).} It follows that $\mathbb{E}\left[  \left\| { g_{\eta,\fyyy{\tilde{\epsilon}}}(\x,{v},\omega) - g_{\eta}(\x,{v},\omega) } \right\|^2 \mid \x\right] 
\leq \frac{{4\tilde{L}^2_0n^2}\fyyy{\tilde{\epsilon}}}{\eta^2}$ {holds almost surely.}
\end{proof}
{We make use of the following \fyyy{result} in the convergence and rate analysis. 
\begin{lemma}[Lemma~2.11 in \cite{KaushikYousefian2021}]\label{Lem:averaging_seq}\em
Let $\{\bar \x_k\}$ be generated by Algorithm \ref{algorithm:inexact_zeroth_order_SVIs}. \fyyy{Let} $\alpha_{k,N} \triangleq \frac{\gamma_k^r}{\sum_{j=0}^N \gamma_j^r}$ for $k \in \{0,\ldots, N\}$ and $N\geq 0$. Then, for any $N\geq 0$, we have $\bar{\x}_{N} = \sum_{k=0}^N \alpha_{k,N} \x_k$. Furthermore, \fyyy{if} $\Xscr$ is a convex set, \fyyy{then} $\bar \x_N \in \Xscr$.
\end{lemma}
\fyy{\begin{remark}\em Lemma~\ref{Lem:averaging_seq} allows for representing $\bar \x_k$ in Algorithm~\ref{algorithm:inexact_zeroth_order_SVIs} as a weighted average of the generated iterates $\{\x_k\}$. The term $\gamma_k^r$ in the last step of (\texttt{ZSOL$^{\bf 1s}_{\rm cnvx}$}) is employed to build the weights $\frac{\gamma_k^r}{\sum_{j=0}^N \gamma_j^r}$ where $0\leq r<1$ is a fixed parameter that can be arbitrarily chosen. This averaging scheme was studied earlier~\cite{KaushikYousefian2021,FarzadAngeliaUday_MathProg17} and allows for achieving the best convergence rate for SA methods.
\end{remark}}
}
{We are now in a position to develop rate and complexity statements for Algorithms~\ref{algorithm:inexact_zeroth_order_SVIs}--\ref{algorithm:inexact_lower_level_SA}.}
The algorithm parameters for both schemes are defined next. 

{\begin{definition}[{Parameters for Algorithms~\ref{algorithm:inexact_zeroth_order_SVIs}--\ref{algorithm:inexact_lower_level_SA}}]\label{def:algo_1}\em Let the stepsize and smoothing sequence in Algorithm~\ref{algorithm:inexact_zeroth_order_SVIs} be given by {$\gamma_k := \frac{\gamma_0}{(k+1)^a}$ and $\eta_k := \frac{\eta_0}{(k+1)^b}$}, respectively for all $k\geq 0$ where {$\gamma_0, \eta_0, a,$ and $b$} are strictly positive.  In Algorithm~\ref{algorithm:inexact_lower_level_SA}, suppose  $\alpha \leq \tfrac{\mu_F}{2L_F^2}$, $M_t := \lceil M_0\rho^{-t} \rceil$ for $t\geq 0$ for some $0<\rho<1$ where $M_0 \geq \tfrac{2\nu_\y^2}{L_F^2}$.  Let $t_k:=\lceil\tau \ln(k+1) \rceil $ where $\tau \geq \frac{-{2(a+b)}}{\ln(\max\{1-\mu_F\alpha,\rho\})}$.
Finally, suppose $r \in [0,1)$ is an arbitrary scalar. 
\end{definition}}
{\begin{theorem}[\bf{{Rate and complexity statements and almost sure convergence for inexact \texttt{ZSOL$^{\bf 1s}_{\rm cnvx}$}}}]\label{thm:ZSOL_convex}\em
Consider \fy{the} sequence {$\{\bar \x_k\}$} generated by applying Algorithm~\ref{algorithm:inexact_zeroth_order_SVIs} on {\eqref{prob:mpec_exp_imp}}. Suppose {Assumptions~\ref{ass-1}--~\ref{assum:lower_level_stoch}} hold and algorithm parameters are defined by Def.~\ref{def:algo_1}.  

        \noindent {\bf(a)} {Suppose \fyy{$\hat \x_k \in \Xscr+\eta_k\mathbb{S}$} and let $\{\y_{t_k}\}$ \fy{be the} sequence generated by Algorithm \ref{algorithm:inexact_lower_level_SA}. Then for  suitably defined $\tilde{d} < 1$ and $B\fyyy{(\hat \x_k)} > 0$, the following holds for $t_k \geq 1$. }  
\begin{align*}
 \mathbb{E}[\|\y_{t_k} - \y(\hat \x_k)\|^2] \leq \fyyy{\tilde{\epsilon}_k} \triangleq  B\fyyy{(\hat \x_k)} \tilde{d}^{\fy{t_k}}.
\end{align*}

\noindent {\bf{(b)}}  {Let $a=0.5$ and $b \in [0.5,1)$ and $0\leq r< 2(1-b)$. Then, for all $K \geq 2^{\frac{1}{1-r}}-1$ we have
}  
{\begin{align*} \mathbb{E}\left[f^{\bf imp}(\bar \x_K)\right]-f^*  \leq (2-r)\left(\tfrac{D_\Xscr }{\gamma_0}+\tfrac{2\theta_0\fyyy{(\hat \x_k)}\gamma_0}{1-r}\right)\tfrac{1}{\sqrt{K+1}}+(2-r)\left(\tfrac{\eta_0L_0}{1-0.5r-b}\right)\tfrac{1}{(K+1)^{b}},\end{align*}
where $\theta_0\fyyy{(\hat \x_k)} \triangleq D_\Xscr  +\tfrac{\left(2+3\gamma_0^2\right)n^2\tilde{L}_0^2B}{\eta_0^2\gamma_0^2}+ 1.5n^2L_0^2$.
In particular, when $b:=1-\delta$ and $r=0$, where $\delta>0$ is a small scalar,  we have for all $K\geq 1$
\begin{align*} 
\mathbb{E}\left[f^{\bf imp}(\bar \x_K)\right]-f^* &\leq 2\left(\tfrac{D_\Xscr }{\gamma_0}+2\theta_0\fyyy{(\hat \x_k)}\gamma_0\right)\tfrac{1}{\sqrt{K+1}}+\left(\tfrac{2\eta_0L_0}{\delta}\right)\tfrac{1}{(K+1)^{1-\delta}}.
\end{align*}}

\noindent {\bf{(c)}}  Suppose $\gamma_0:= \mathcal{O}(\tfrac{1}{L_0})$, $a:=0.5$, $b:=0.5$, and $r:=0$. Let $\epsilon>0$ be an arbitrary scalar and $K_{\epsilon}$ be such that $\mathbb{E}\left[f^{\bf imp}(\bar \x_{K_\epsilon})\right]-f^*  \leq \epsilon$. Then,

\noindent {{(c-1)}} The total number of upper-level projection steps on $\Xscr$ is $K_{\epsilon}=\mathcal{O}\left(n^4L_0^2\tilde{L}_0^4\epsilon^{-2}  \right)$.

\noindent {{(c-2)}} \fyy{The total sample complexity of upper-level is} $\mathcal{O}\left(n^4L_0^2\tilde{L}_0^4\epsilon^{-2}  \right)$.
 
\noindent {{(c-3)}} The total number of lower-level projection steps on $\Yscr$ is $ 
\mathcal{O}\left(n^4L_0^2\tilde{L}_0^4\epsilon^{-2}  \ln\left( n^2L_0\tilde{L}_0^2\epsilon^{-1}  \right)\right).$

 \noindent {{(c-4)}} The total sample complexity of lower-level is $ \mathcal{O}\left( n^{4\bar \tau}L_0^{2\bar \tau}\tilde{L}_0^{4\bar \tau}\epsilon^{-2\bar \tau}\right)$ where $\bar \tau \geq 1-\tau \ln(\rho)$.\\

\noindent {\bf{(d)}} For any $a \in (0.5,1]$ and $b>1-a$, there exists $\x^* \in \Xscr^*$ such that $\lim_{k\to \infty} \|\bar \x_k-\x^*\|^2 = 0$ almost surely.
\end{theorem}}
\begin{proof}
{\noindent {\bf (a)} 
We define the errors $\Delta_t \triangleq \bar{F}(\hat \x_k,\y_t) - F(\hat{\x}_k,\y_t)$ for $t\geq 0$.  
Next, we estimate a bound on the term $\mathbb{E}[\|\Delta_t\|^2\mid \fyy{\hat \x_k,\y_t} ]$. \fy{From Assumption \ref{assum:lower_level_stoch} we} have
\begin{align}\label{eqn:bound_on_Delta_norm_sq}
\mathbb{E}[\|\Delta_t\|^2\mid\fyy{\hat \x_k,\y_t} ] &= \mathbb{E}\left[\left\|\tfrac{\sum_{\ell = 1}^{M_t} (G(\hat{\x}_k,\y_t,\omega_{\ell, t})-F(\hat \x_k,\fyyy{\y_t}))}{M_t}\right\|^2\mid \fyy{\hat \x_k,\y_t} \right]\notag \\
& = \tfrac{1}{M_t^2}\mathbb{E}\left[\sum_{\ell = 1}^{M_t} \left\|G(\hat{\x}_k,\y_t,\omega_{\ell, t})-F(\hat \x_k,\fyyy{\y_t})\right\|^2\mid \fyy{\hat \x_k,\y_t} \right]\leq \tfrac{\nu_\y^2\|\y_t\|^2+\nu_G^2}{M_t}.
\end{align}
From $ \fy{\y(\hat\x_k)} \in \mbox{SOL}(\Yscr, F(\fy{\hat \x_k,\bullet}))$, we have $\y(\hat \x_k)=\Pi_{\Yscr}\left[\y(\hat \x_k) -\alpha{F}(\hat \x_k,\y(\hat \x_k))\right]$ {for any $\alpha > 0$}. We have
\begin{align*}
&\quad  \|\y_{t+1}- \y(\hat \x_k)\|^2 = \|\Pi_{\Yscr}\left[\y_t - \alpha\bar{F}(\hat \x_k,\y_t)\right] - \Pi_{\Yscr}\left[\y(\hat \x_k) -\alpha{F}(\hat \x_k,\y(\hat \x_k))\right] \|^2 \\ 
&\leq  \|\y_t - \alpha\bar{F}(\hat \x_k,\y_t)-\y(\hat \x_k) +\alpha{F}(\hat \x_k,\y(\hat \x_k))\|^2\\
& =  \|\y_t - \alpha{F}(\hat \x_k,\y_t)-\alpha\Delta_t-\y(\hat \x_k)+\alpha{F}(\hat \x_k,\y(\hat \x_k)) \|^2\\
& = \|\y_t - \y(\hat \x_k)\|^2 + \alpha^2\|{F}(\hat \x_k,\y_t)-{F}(\hat \x_k,\y(\hat \x_k))\|^2 +\alpha^2\|\Delta_t\|^2 \\
&- 2\alpha(\y_t - \y(\hat \x_k))^T({F}(\hat \x_k,\y_t)-{F}(\hat \x_k,\y(\hat \x_k))) \\
&-2\alpha(\y_t - \y(\hat \x_k)-\alpha{F}(\hat \x_k,\y_t)+\alpha{F}(\hat \x_k,\y(\hat \x_k)) )^T\Delta_t.
\end{align*}
Taking conditional expectations in the preceding relation,  using \eqref{eqn:bound_on_Delta_norm_sq}, and invoking the strong monotonicity and Lipschitzian property of the mapping $F$ in Assumption \ref{ass-1}, we obtain
\begin{align*}
   \mathbb{E}[\|\y_{t+1}- \y(\hat \x_k)\|^2\mid \fyy{\hat \x_k,\y_t}] & \leq \left(1-2\mu_F\alpha+\alpha^2L_F^2\right)\|\y_t - \y(\hat \x_k)\|^2 +\tfrac{\nu_\y^2\|\y_t\|^2+\nu_G^2}{M_t}\alpha^2.
\end{align*}
Taking expectations on both sides, we obtain
\begin{align*}
   \mathbb{E}[\|\y_{t+1}- \y(\hat \x_k)\|^2] & \leq\left(1-2\mu_F\alpha+\alpha^2L_F^2\right) \mathbb{E}[\|\y_t - \y(\hat \x_k)\|^2] +\tfrac{\nu_\y^2   \mathbb{E}[\|\y_t- \y(\hat \x_k)+ \y(\hat \x_k)\|^2]+\nu_G^2}{M_t}\alpha^2\\
   & \leq\left(1-2\mu_F\alpha+\alpha^2L_F^2+\tfrac{2\nu_\y^2}{M_0}\alpha^2\right) \mathbb{E}[\|\y_t - \y(\hat \x_k)\|^2] +\tfrac{2\nu_\y^2 \|\y(\hat \x_k)\|^2+\nu_G^2}{M_t}\alpha^2.
\end{align*}
\fy{Let} $\lambda \triangleq 1-2\mu_F\alpha+\alpha^2L_F^2+\tfrac{2\nu_\y^2}{M_0}\alpha^2 $ \fyyy{and} $\Lambda_t\fyyy{(\hat \x_k)} \triangleq \tfrac{2\nu_\y^2 \|\y(\hat \x_k)\|^2+\nu_G^2}{M_t}\alpha^2$ for $t\geq 0$. Note that since $M_0 \geq \tfrac{2\nu_\y^2}{L_F^2}$ and that $\alpha \leq \tfrac{\mu_F}{2L_F^2}$, we have $\lambda \leq 1-\mu_F\alpha <1$. We obtain for any $t \geq 0$
\begin{align*}
 \mathbb{E}[\|\y_{t+1}- \y(\hat \x_k)\|^2] & \leq\lambda^{t+1}\|\y_0 - \y(\hat \x_k)\|^2+\sum_{j=0}^t\lambda^{t-j}\Lambda_j\fyyy{(\hat \x_k)}\\
 & \leq\lambda^{t+1}\|\y_0 - \y(\hat \x_k)\|^2+\Lambda_0\fyyy{(\hat \x_k)}(\max\{\lambda,\rho\})^t\sum_{j=0}^t\left(\tfrac{\min\{\lambda,\rho\}}{\max\{\lambda,\rho\}}\right)^{t-j}\\
 & \leq\lambda^{t+1}\|\y_0 - \y(\hat \x_k)\|^2+\tfrac{\Lambda_0\fyyy{(\hat \x_k)}(\max\{\lambda,\rho\})^t}{1-(\min\{\lambda,\rho\}/\max\{\lambda,\rho\})}\leq B\fyyy{(\hat \x_k)} \tilde{d}^{t+1}.
\end{align*}
{where $\tilde{d} \triangleq \max\{\lambda, \rho\}$ and $B\fyyy{(\hat \x_k)} \triangleq \sup_{\y \in \Yscr}\|\y -\y_0\|^2+\tfrac{\Lambda_0\fyyy{(\hat \x_k)}}{\max\{\lambda,\rho\}-\min\{\lambda,\rho\}}$. \fyyy{Note that in view of compactness of $\Yscr$, $B\fyyy{(\hat \x_k)}<\infty$. Also,} without loss of generality, we assume {that} $\rho\neq \lambda$.}

\noindent {\bf (b)} 
Let us define $\bar{F}(\hat \x_k,\y_t)\triangleq \tfrac{\sum_{\ell = 1}^{M_t} G(\hat{\x}_k,\y_t,\omega_{\ell, t})}{M_t}$ for $t\geq 0$ and $k\geq 0$. Note that from the compactness of the set $\Xscr$ and the continuity of the implicit function, the set $\Xscr^*$ is nonempty. Let $\x^* \in \Xscr$ be an arbitrary optimal solution. We have that}
\begin{align*}
\|\x_{k+1} - \x^*\|^2 &= \left\|\Pi_{\Xscr} \left[ \x_k - \gamma_k  {g}_{\eta_k,\fyyy{\tilde{\epsilon}_k}}(\x_k,v_k,\omega_k) \right] -\Pi_{\Xscr}\left[\x^* \right]\right\|^2 \leq  \left\|  \x_k - \gamma_k   {g}_{\eta_k,\fyyy{\tilde{\epsilon}_k}}(\x_k,v_k,\omega_k)-{\x^*}\right\|^2 \\
&= \|\x_k-\x^*\|^2 -2\gamma_k(\x_k-\x^*)^T  {g}_{\eta_k,\fyyy{\tilde{\epsilon}_k}}(\x_k,v_k,\omega_k)+\gamma_k^2\| {g}_{\eta_k,\fyyy{\tilde{\epsilon}_k}}(\x_k,v_k,\omega_k)\|^2\\
&= \|\x_k-\x^*\|^2 -2\gamma_k(\x_k-\x^*)^T  ({g}_{\eta_k}(\x_k,v_k,\omega_k)+w_k)+\gamma_k^2\| {g}_{\eta_k,\fyyy{\tilde{\epsilon}_k}}(\x_k,v_k,\omega_k)\|^2,
\end{align*}
where we define $w_k \triangleq {g}_{\eta_k,\fyyy{\tilde{\epsilon}_k}}(\x_k,v_k,\omega_k)-g_{\eta_k}(\x_k,v_k,\omega_k)$. Taking conditional expectations on the both sides, and invoking Lemma \ref{lem:smooth_grad_properties} and Lemma \ref{lem:inexact_error_bounds} (a), we obtain
\begin{align*}
    \mathbb{E}\left[\|\x_{k+1} - \x^*\|^2\mid \fyy{\x_k}\right] &\leq  \|\x_k-\x^*\|^2 -2\gamma_k(\x_k-\x^*)^T \nabla {f^{\bf imp}_{\eta_k}(\x_k)}\\
                                                           &-2\gamma_k    \mathbb{E}\left[(\x_k-\x^*)^Tw_k \mid \fyy{\x_k}\right] +3n^2\gamma_k^2\left(\tfrac{2{\tilde{L}_0^2} \fyyy{\tilde{\epsilon}_k}   }{\eta_k^2} + {L_0^2}\right).
\end{align*}
Invoking the convexity of ${f^{\bf imp}_{\eta_k}}$, bounding $-2\gamma_k(\x_k-\x^*)^Tw_k$,  and rearranging the terms, we obtain
\begin{align*}
2\gamma_k\left({f^{\bf imp}_{\eta_k}(\x_k )-f^{\bf imp}_{\eta_k}(\x^*)}\right) &\leq \|\x_k-\x^*\|^2 -\mathbb{E}\left[\|\x_{k+1} - \x^*\|^2\mid \fyy{\x_k}\right]\\
                                                                          & + \gamma_k^2 \|\x_k-\x^*\|^2 +\mathbb{E}\left[\|w_k\|^2\mid \fyy{\x_k}\right] +  3n^2\gamma_k^2\left(\tfrac{2{\tilde{L}_0^2} \fyyy{\tilde{\epsilon}_k}   }{\eta_k^2} + {L_0^2}\right).
\end{align*}
From Lemma \ref{lem:inexact_error_bounds} (b) we obtain 
\begin{align*}
2\gamma_k\left({f^{\bf imp}_{\eta_k}(\x_k )-f^{\bf imp}_{\eta_k}(\x^*)}\right) &\leq \|\x_k-\x^*\|^2 -\mathbb{E}\left[\|\x_{k+1} - \x^*\|^2\mid \fyy{\x_k}\right]\\
                                                                          & + \gamma_k^2 \|\x_k-\x^*\|^2 + \tfrac{{4\tilde{L}^2_0n^2}\fyyy{\tilde{\epsilon}_k}}{\eta_k^2} +  3n^2\gamma_k^2\left(\tfrac{2{\tilde{L}_0^2} \fyyy{\tilde{\epsilon}_k}   }{\eta_k^2} + {L_0^2}\right).
\end{align*}
From Lemma \ref{lemma:props_local_smoothing} (v) we have that {$f^{\bf imp}(\x_k) \leq f^{\bf imp}_{\eta_k}(\x_k)$ and $ f^{\bf imp}_{\eta_k}(\x^*) \leq {f^*} +\eta_k{L_0}$.  From the preceding inequalities} we obtain
\begin{align*}
 \quad {2\gamma_k\left(f^{\bf imp}(\x_k)-{f^*}\right)}&\leq\|\x_k-\x^*\|^2 -\mathbb{E}\left[\|\x_{k+1} - \x^*\|^2\mid \fyy{\x_k}\right] 
+ \gamma_k^2  \|\x_k-\x^*\|^2 \\
&+  (4+6\gamma_0^2)\tfrac{{\tilde{L}^2_0n^2}\fyyy{\tilde{\epsilon}_k}}{\eta_k^2}+2\gamma_k \eta_k L_0+ 3n^2L_0^2\gamma_k^2.
\end{align*}
Next, we derive a bound on \fy{$\tfrac{\fyyy{\tilde{\epsilon}_k}}{\eta_k^2}$}. From part (a) and the update rule of $\eta_k$ we have
\begin{align}\label{eqn:bound_on_ek_etak2}
\tfrac{\fyyy{\tilde{\epsilon}_k}}{\eta_k^2} =\left(\tfrac{\fyyy{\tilde{\epsilon}_k}}{\eta_k^2\gamma_k^2}\right) \gamma_k^2= \left(\tfrac{\left(\max\{\lambda,\rho\}\right)^{t_k}B\fyyy{(\hat \x_k)}(k+1)^{2(a+b)}}{\eta_0^2\gamma_0^2}\right) \gamma_k^2.
\end{align}
Note that from $\alpha \leq \tfrac{\mu_F}{2L_F^2}$ and $M_0 \geq \tfrac{2\nu_\y^2}{L_F^2}$, we have $\lambda \leq 1-\mu_F\alpha$. Thus, we have $\tau \geq \frac{-{2(a+b)}}{\ln(\max\{1-\mu_F\alpha,\rho\})} \geq \frac{-{2(a+b)}}{\ln(\max\{\lambda,\rho\})}$. From $t_k:=\lceil\tau \ln(k+1) \rceil \geq \tau \ln(k+1)$ and $\tau \geq  \frac{-{2(a+b)}}{\ln(\max\{\lambda,\rho\})}$ we have that 
\begin{align*}
\left(\max\{\lambda,\rho\}\right)^{t_k}(k+1)^{2(a+b)}\leq \left(\left(\max\{\lambda,\rho\}\right)^{\tau}\mathrm{e}^{2(a+b)}\right)^{\ln(k+1)}\leq \left(\max\{\lambda,\rho\}\right)^{\tau}\mathrm{e}^{2(a+b)} \leq 1.
\end{align*}
This relation and \eqref{eqn:bound_on_ek_etak2} imply that $\tfrac{\fyyy{\tilde{\epsilon}_k}}{\eta_k^2}\leq \left(\tfrac{B\fyyy{(\hat \x_k)}}{\eta_0^2\gamma_0^2}\right) \gamma_k^2$. Also, note that since $\Xscr$ is bounded, there exists a scalar $D_\Xscr \triangleq \frac{1}{2}\sup_{\x \in \Xscr} \|\x-\x^*\|^2$ \fyyy{such that $D_{\Xscr}<\infty$}.  Therefore, we obtain
\begin{align}{\label{eqn:recurs_ineq_as_conv}}
& {2\gamma_k\left(f^{\bf imp}(\x_k)-{f^*}\right)\leq  \|\x_k-\x^*\|^2  -\mathbb{E}\left[\|\x_{k+1} - \x^*\|^2\mid \fyy{\x_k} \right] 
+ 2\gamma_k^2\theta_0\fyyy{(\hat \x_k)}+ 2\gamma_k\eta_k L_0,}
\end{align}
where {$\theta_0\fyyy{(\hat \x_k)} \triangleq D_\Xscr  +\tfrac{\left(2+3\gamma_0^2\right)n^2\tilde{L}^2_0B\fyyy{(\hat \x_k)}}{\eta_0^2\gamma_0^2}+ 1.5n^2L_0^2\ \fyyy{<\infty}$}. {Taking expectations on the both sides and multiplying} both sides by $\frac{\gamma_k^{r-1}}{2}$, we have that 
\begin{align}\label{eqn:inexact_proof_a_ineq1}
& \gamma_k^r\left(\mathbb{E}\left[{f^{\bf imp}(\x_k)}\right]-{f^*}\right)\leq \frac{\gamma_k^{r-1}}{2}\left(\mathbb{E}\left[ \|\x_k-\x^*\|^2\right] -\mathbb{E}\left[\|\x_{k+1} - \x^*\|^2 \right] \right)
+ \gamma_k^{1+r}\theta_0\fyyy{(\hat \x_k)}+\gamma_k^r\eta_kL_0.
\end{align}
{Adding and subtracting the term $\frac{\gamma_{k-1}^{r-1}}{2}\mathbb{E}\left[ \|\x_k-\x^*\|^2\right]$, we obtain 
\begin{align*} 
 & \quad \gamma_k^r\left(\mathbb{E}\left[{f^{\bf imp}(\x_k)}\right]-{f^*}\right)\\ 
& \leq \frac{\gamma_{k-1}^{r-1}}{2}\mathbb{E}\left[ \|\x_k-\x^*\|^2\right] - \frac{\gamma_{k}^{r-1}}{2}\mathbb{E}\left[\|\x_{k+1} - \x^*\|^2 \right] + \left(\gamma_k^{r-1}-\gamma_{k-1}^{r-1}\right)D_\Xscr + \theta_0\fyyy{(\hat \x_k)} \gamma_k^{1+r}+\gamma_k^r\eta_kL_0.
\end{align*}
Summing both sides from $k = 1,\ldots, K$ we obtain
\begin{align*} 
\sum_{k=1}^K\gamma_k^r\left(\mathbb{E}\left[{f^{\bf imp}(\x_k)}\right]-{f^*}\right) &\leq \frac{\gamma_{0}^{r-1}}{2}\mathbb{E}\left[ \|\x_1-\x^*\|^2\right]  + \left(\gamma_K^{r-1}-\gamma_{0}^{r-1}\right)D_\Xscr  \\ &+\theta_0\fyyy{(\hat \x_k)}\sum_{k=1}^K\gamma_k^{1+r}+L_0\sum_{k=1}^K\gamma_k^r\eta_k.
\end{align*}
Writing \eqref{eqn:inexact_proof_a_ineq1} for $k:=0$ we have
\begin{align*}
& \gamma_0^r\left(\mathbb{E}\left[{f^{\bf imp}(\x_0)}\right]-{f^*}\right)\leq \frac{\gamma_0^{r-1}}{2}\left(\mathbb{E}\left[ \|\x_0-\x^*\|^2\right] -\mathbb{E}\left[\|\x_{1} - \x^*\|^2 \right] \right)
+ \theta_0\fyyy{(\hat \x_k)}\gamma_0^{1+r}+\gamma_0^r\eta_0L_0.
\end{align*}
Adding the preceding two relations together and using the definition of $D_\Xscr$, we obtain
\begin{align*} 
\sum_{k=0}^K\gamma_k^r\left(\mathbb{E}\left[{f^{\bf imp}(\x_k)}\right]-{f^*}\right) \leq D_\Xscr \gamma_K^{r-1} +\theta_0\fyyy{(\hat \x_k)} \sum_{k=0}^K\gamma_k^{1+r}+L_0\sum_{k=0}^K\gamma_k^r\eta_k.
\end{align*}
From the definition $\bar \x_K \triangleq \sum_{k=0}^K\alpha_{k,K}\x_k$ in Lemma \ref{Lem:averaging_seq} and applying the convexity of the implicit function, for all $K \geq 2^{\frac{1}{1-r}}-1$ we have
\begin{align*} 
\mathbb{E}\left[{f^{\bf imp}(\bar \x_K)}\right]-f^* \leq \frac{D_\Xscr\gamma_K^{r-1} +\theta_0\fyyy{(\hat \x_k)} \sum_{k=0}^K\gamma_k^{1+r}+L_0\sum_{k=0}^K\gamma_k^r\eta_k}{\sum_{k=0}^K\gamma_k^r}.
\end{align*}
Substituting $\gamma_k :=\frac{\gamma_0}{\sqrt{k+1}}$ and $\eta_k := \frac{\eta_0}{(k+1)^b}$, and invoking Lemma \ref{lem:harmonic_bounds}, we obtain
\begin{align*} 
\mathbb{E}\left[{f^{\bf imp}(\bar \x_K)}\right]-f^* & \leq \frac{D_\Xscr \gamma_0^{r-1}(K+1)^{0.5(1-r)} +\theta_0\fyyy{(\hat \x_k)}\gamma_0^{1+r}\frac{(K+1)^{1-0.5(1+r)}}{1-0.5(1+r)}+\gamma_0^r\eta_0L_0\frac{(K+1)^{1-0.5r-b}}{1-0.5r-b}}{\gamma_0^r\frac{(K+1)^{1-0.5r}}{2-r}} \\
& \leq (2-r)\left(\tfrac{D_\Xscr }{\gamma_0}+\tfrac{2\theta_0\fyyy{(\hat \x_k)}\gamma_0}{1-r}\right)\tfrac{1}{\sqrt{K+1}}+(2-r)\left(\tfrac{\eta_0L_0}{1-0.5r-b}\right)\tfrac{1}{(K+1)^{b}}.
\end{align*}}

\noindent {\bf(c)} {The results in (c-1) and (c-2) follow directly from part (b) by substituting $\gamma_0$ and $r$. To show part (c-3), note that in Algorithm \ref{algorithm:inexact_zeroth_order_SVIs},  we have $t_k : = \lceil \tau\ln(k+1)\rceil$. From part (b),  we require the following total number of iterations of the SA scheme.
\begin{align*}
2\sum_{k=0}^{K_\epsilon} t_k &= 2\sum_{k=0}^{K_\epsilon}   \lceil \tau\ln(k+1)\rceil\leq 2\left(K_\epsilon+1\right)+2 \tau\sum_{k=2}^{K_\epsilon+1}\ln(k)\\
                             &\leq 2\left(K_\epsilon+1\right)+2 \tau\int_{2}^{K_\epsilon+1}\ln({u})d{u}   \leq 2\left(K_\epsilon+1\right)+ 2\tau\left(K_\epsilon+2\right)\ln\left(K_\epsilon+2\right)\\
&\leq  4\max\{\tau,1\}\left(K_\epsilon+2\right)\ln\left(K_\epsilon+2\right).
\end{align*}
The bound in (c-3) follows from the preceding inequality and the bound on $K_\epsilon$ in (c-1). To show (c-4), note that the total samples used in the lower-level is as follows.
\begin{align*}
2\sum_{k=0}^{K_\epsilon}\sum_{t=0}^{ t_k} M_t &=2\sum_{k=0}^{K_\epsilon}\sum_{t=0}^{ t_k}   \lceil M_0\rho^{-t}\rceil \leq 4M_0 \sum_{k=0}^{K_\epsilon}\sum_{t=0}^{t_k}   \rho^{-t} =\mathcal{O}\left( \sum_{k=0}^{K_\epsilon}  \frac{\rho^{-t_k}}{\ln(\tfrac{1}{\rho})}\right)=\mathcal{O}\left( \sum_{k=0}^{K_\epsilon}  \frac{\rho^{-\tau\ln(k+1)}}{\ln(\tfrac{1}{\rho})}\right)\\
& \leq \mathcal{O}\left( \sum_{k=0}^{K_\epsilon}  \frac{e^{\fy{(\bar \tau-1)}\ln(k+1)}}{\ln(\tfrac{1}{\rho})}\right)=\mathcal{O}\left( \sum_{k=0}^{K_\epsilon}  \frac{(k+1)^{\fy{\bar \tau-1}}}{\ln(\tfrac{1}{\rho})}\right)=\mathcal{O}\left(  \frac{K_\epsilon^{\bar \tau}}{\ln(\tfrac{1}{\rho})}\right),
\end{align*}
where \fy{$\bar \tau \geq 1+\tau \ln(\tfrac{1}{\rho})$}. The bound in (c-4) follows from the preceding inequality and the bound on $K_\epsilon$ in (c-1).
}

\noindent {{\bf{(d)}} Consider relation \eqref{eqn:recurs_ineq_as_conv}. Rearranging the terms, for all $k\geq 0$ we have 
\begin{align*}
& \mathbb{E}\left[\|\x_{k+1} - \x^*\|^2\mid \fyy{\x_k} \right]  \leq   \|\x_k-\x^*\|^2  -2\gamma_k\left(f^{\bf imp}(\x_k)-{f^*}\right)
+ 2\gamma_k^2\theta_0\fyyy{(\hat \x_k)}+2\gamma_k\eta_k L_0.
\end{align*}
Note that we have $\sum_{k=0}^{\infty} \gamma_k^2 < \infty$ and
$\sum_{k=0}^{\infty} \gamma_k \eta_k < \infty$ since $b>0.5$. Thus, in view of
Lemma \ref{lemma:supermartingale}, we have that $\{\|\x_k-\x^*\|^2\}$ is a
convergent sequence in an almost sure sense and $\sum_{k=0}^{\infty}  \gamma_k
(f^{\bf imp}(\x_k) - f^*) < \infty$ almost surely. The former statement implies
that $\{\x_k\}$ is a bounded sequence {in an a.s. sense}. Further, the latter statement and
$\sum_{k=0}^{\infty} \gamma_k = \infty$ imply that $\liminf_{k \to \infty}
f^{\bf imp}(\x_k)= f^*$ {in an a.s. sense}. Thus, from continuity of the implicit function, there
is a subsequence of $\{\x_k\}_{k \in \mathcal{K}}$ with {limit point
denoted by $\hat \x$} such that $\hat \x \in \Xscr^*$. Since
$\{\|\x_k-\x^*\|^2\}$ is a convergent sequence for all $\x^* \in \Xscr^*$, we
have  $\{\|\x_k-\hat \x\|^2\}$ is a convergent sequence. But {we have shown} that
$\lim_{k\to \infty, \ k \in \mathcal{K}}\|\x_k-\hat \x\|^2 =0$ almost surely.
Hence $\lim_{k\to \infty}\|\x_k-\hat \x\|^2 =0$ almost surely where $\hat \x
\in \Xscr^*$. Next, we show that $\lim_{k\to \infty}\|\bar \x_k-\hat \x\|^2
=0$. In view of Lemmas \ref{Lem:averaging_seq} and \ref{lem:convergence_sum},
it suffices to have $\sum_{k=0}^\infty \gamma_k^r = \infty$ or equivalently, we
must have $ar\leq 1$. This is already satisfied as a consequence of $a \in
(0.5,1]$ and $r \in [0,1)$.} 
\end{proof}
{\begin{remark}[{\bf Variance-reduction schemes}]\em
\begin{enumerate}
\item[]
\item[(i)] In Algorithm~\ref{algorithm:inexact_lower_level_SA} we employ a variance-reduced (VR) scheme in computing an $\epsilon$-solution of the parametrized VI at the lower-level. This is crucial since it allows for computing an $\epsilon$-solution in $\log(1/\epsilon)$ steps while in a non-VR regime, it would have taken $\mathcal{O}(1/\epsilon)$ steps. Variance-reduction on strongly monotone VIs has been studied in~\cite{iusem19variance,cui2021analysis,jalilzadeh19proximal}, amongst others. 
\item[(ii)] In addressing single-stage SMPECs, while employing a VR scheme in either lower-level or upper-level is possible, but sometimes this approach may not be advisable to be adopted at the both levels simultaneously. For instance, in (\texttt{ZSOL$^{\bf 1s}_{\rm cnvx}$}), employing a VR scheme in the upper-level would lead to requiring an increasing number of inexact solutions of a lower-level stochastic VI at each iteration, where each of these solutions would require a VR scheme to be employed in the lower-level. Consequently, this may render the scheme impractical. 
\end{enumerate}
\end{remark}}

{\begin{remark}[{\bf Definition of history}]\em 
We conclude this subsection with a brief remark regarding the formal definition of the $\sigma-$algebra for Algorithms~\ref{algorithm:inexact_zeroth_order_SVIs}--\ref{algorithm:inexact_lower_level_SA}. First, $\mathcal{F}_{0,0} \triangleq \{\x_0\}$. In addition, $\mathcal{F}_{k,0}$ is defined as 
\begin{align*} 
	\mathcal{F}_{1,0}  = \mathcal{F}_{0,0} \cup \left\{\omega_0, v_0\right\} \cup \mathcal{F}_{0,t_0}^1 \cup \mathcal{F}_{0,t_0}^2,    \mbox{ where }
\end{align*}
\begin{align*}
\mathcal{F}_{0,t}^1 & \triangleq \left\{ \left\{G(\x_0,\y_{0},\omega_{\ell,{0}}\fyyy{)}\right\}_{\ell=1}^{M_{0}}, \cdots, \left\{G({\x}_0,\y_{t-1},\omega_{\ell,t-1}\fyyy{)}\right\}_{\ell=1}^{M_0}\right\} \mbox{ and } \\
\mathcal{F}_{0,t}^2 & \triangleq \left\{ \left\{G(\x_0+v_0,\y_{0},\omega_{\ell,{0}}\fyyy{)}\right\}_{\ell=1}^{M_{0}}, \cdots, \left\{G({\x}_0+v_0,\y_{t-1},\omega_{\ell,t-1}\fyyy{)}\right\}_{\ell=1}^{M_0}\right\} \mbox{ for } t = 0, \cdots, t_0-1.
\end{align*}
At the $k$th iteration with $k > 0$, we have that  
\begin{align*} 
	\mathcal{F}_{k,0}  = \mathcal{F}_{k-1,0} \cup \left\{\omega_k, v_k\right\} \cup \mathcal{F}_{k,t_k}^1 \cup \mathcal{F}_{k,t_k}^2, \mbox{ where }   
\end{align*}
\begin{align*}
\mathcal{F}_{k,t}^1 & \triangleq \left\{ \left\{G(\x_k,\y_{0},\omega_{\ell,{0}}\fyyy{)}\right\}_{\ell=1}^{M_{t}}, \cdots, \left\{G({\x}_k,\y_{t-1},\omega_{\ell,t-1}\fyyy{)}\right\}_{\ell=1}^{M_t}\right\} \mbox{ and } \\
\mathcal{F}_{k,t}^2 & \triangleq \left\{ \left\{G(\x_k+v_k,\y_{0},\omega_{\ell,{t}}\fyyy{)}\right\}_{\ell=1}^{M_{t}}, \cdots, \left\{G({\x}_k+v_k,\y_{t-1},\omega_{\ell,t-1}\fyyy{)}\right\}_{\ell=1}^{M_t}\right\} \mbox{ for } t = 0, \cdots, t_k-1.
\end{align*}
 In particular, at the $t$th, iteration of the SA scheme at the $k$th upper-level step, we may define $\mathcal{F}_{k,t}$ as 
\begin{align*}
\mathcal{F}_{k,t} \triangleq \mathcal{F}_{k,0} \cup \left\{ \left\{G(\hat{\x}_k,\y_{0},\omega_{\ell,{0}}\fyyy{)}\right\}_{\ell=1}^{M_{0}}, \cdots, \left\{G(\hat{\x}_k,\y_{t-1},\omega_{\ell,t-1}\fyyy{)}\right\}_{\ell=1}^{M_{t-1}}\right\}, \mbox{ for } t = 0, \cdots, t_k - 1.
\end{align*} 
Furthermore, at the $t$th step of the lower-level SA scheme associated with the $k$th iteration, the history is denoted by $\mathcal{F}_{k-1,t}^1$ and $\mathcal{F}_{k-1,t}^2$, defined as 
\begin{align*}
{\mathcal{F}}_{k-1,t}^{1} \triangleq \mathcal{F}_{k-1,0} \cup \{v_k,\omega_k\} \cup \mathcal{F}_{k,t-1}^1  \mbox{ and }
{\mathcal{F}}_{k-1,t}^{2} \triangleq \mathcal{F}_{k-1,0} \cup \{v_k,\omega_k\} \cup \mathcal{F}_{k,t-1}^2.
\end{align*}
 Naturally, one can employ these histories in constructing the conditional
expectations; specifically, at the $k$th iteration, we may use
$\mathcal{F}_{k-1,0}$ while at the $t$th step of the lower-level SA scheme at
the $k$th iteration, we may use $\mathcal{F}_{k-1,t-1}$. For expository ease,
we use the iterate as a proxy in constructing the conditional
expectation, as the reader will observe. Note that for expository ease, we
employ $\y_t$ at iteration $k$ as a proxy for the history (rather than
$\y_{k,t}$).  \end{remark}
}

\subsubsection{An exact zeroth-order scheme}
In this subsection, we consider the {case} where an exact solution of the
lower-level problem is available. {This case is particularly relevant when the lower-level problem is a deterministic variational inequality problem and highly accurate solutions are available}.  We develop {a} zeroth-order method where
the gradient mapping is approximated using two evaluations of the implicit
function. Similar to the inexact setting, we allow for iterative smoothing and
provide the convergence analysis in addressing the original implicit problem.
In the following, we derive non-asymptotic convergence rate statements and also, show an almost sure convergence result for the proposed zeroth-order method in the exact regimes.
{\begin{corollary}[\bf{Rate and complexity statements and \uss{a.s.} convergence for exact ({\texttt{ZSOL$^{\bf 1s}_{\rm cnvx}$}})}]\label{thm:convex_exact}\em
        Consider the problem {\eqref{prob:mpec_exp_imp}}. Suppose {Assumptions~\ref{ass-1}--~\ref{assum:u_iter_smooth}} hold. \fyyy{Let $\{\bar \x_k\}$ denote} the sequence generated by {Algorithm \ref{algorithm:inexact_zeroth_order_SVIs} (exact variant)} in which the stepsize and smoothing sequences are defined as $\gamma_k := \frac{\gamma_0}{(k+1)^a}$ and $\eta_k := \frac{\eta_0}{(k+1)^b}$, respectively, for all $k\geq 0$ where $\gamma_0$ and $\eta_0$ are strictly positive. Then, the following statements hold.

\noindent {\bf{(a)}} Let $a=0.5$ and $b \in [0.5,1)$ and $0\leq r< 2(1-b)$. Then, for all $K \geq 2^{\frac{1}{1-r}}-1$ we have
\begin{align*} 
\mathbb{E}\left[{f^{\bf imp}(\bar \x_K)}\right]-f^* &\leq (2-r)\left(\tfrac{D_\Xscr }{\gamma_0}+\tfrac{L_0^2n^2\gamma_0}{1-r}\right)\tfrac{1}{\sqrt{K+1}}+(2-r)\left(\tfrac{\eta_0L_0}{1-0.5r-b}\right)\tfrac{1}{(K+1)^{b}}.
\end{align*}
In particular, when $b:=1-\delta$ and $r=0$, where $\delta>0$ is a small scalar,  we have for all $K\geq 1$
\begin{align*} 
\mathbb{E}\left[{f^{\bf imp}(\bar \x_K)}\right]-f^* &\leq 2\left(\tfrac{D_\Xscr }{\gamma_0}+L_0^2n^2\gamma_0\right)\tfrac{1}{\sqrt{K+1}}+\left(\tfrac{2\eta_0L_0}{\delta}\right)\tfrac{1}{(K+1)^{1-\delta}}.
\end{align*}

\noindent {\bf{(b)}} Let $a:=0.5$, $b=0.5$, $r=0$, $\gamma_0:=\tfrac{\sqrt{D_\Xscr}}{nL_0}$, and $\eta_0\leq \sqrt{D_\Xscr}n$. Then, the iteration complexity {in projection steps on $\Xscr$} {as well as} the total sample complexity of upper-level evaluations, for achieving {$\mathbb{E}\left[{f^{\bf imp}(\bar \x_{K_\epsilon})}\right]-f^* \leq \epsilon$} for some $\epsilon>0$ is {bounded} as follows.
\begin{align*}
{K_\epsilon} \geq \frac{64n^2L_0^2D_\Xscr}{\epsilon^2} .
\end{align*}

\noindent {\bf{(c)}} For any $a \in (0.5,1]$ and $b>1-a$, there exists $\x^* \in \Xscr^*$ such that $\lim_{k\to \infty} \|\bar \x_k-\x^*\|^2 = 0$ almost surely.
\end{corollary}}
\begin{proof}
\noindent {\bf{(a)}} Let $\x^* \in \Xscr^*$ be an arbitrary optimal solution.  We can write:
\begin{align*}
    \|\x_{k+1} - \x^*\|^2 &= \left\|\Pi_{\Xscr} \left[ \x_k - \gamma_k  {g}_{\eta_k}(\x_k,{v}_k,\omega_k) \right] -\Pi_{\Xscr}\left[\x^* \right]\right\|^2 \leq  \left\|  \x_k - \gamma_k   {g}_{\eta_k}(\x_k,{v_k},\omega_k)-\x^*\right\|^2 \\
                          &= \|\x_k-\x^*\|^2 -2\gamma_k(\x_k-\x^*)^T  {g}_{\eta_k}(\x_k,{v_k},\omega_k)+\gamma_k^2\| {g}_{\eta_k}(\x_k,{v_k},\omega_k)\|^2.
\end{align*}
Taking conditional expectations on the both sides and invoking Lemma \ref{lem:smooth_grad_properties}, we obtain
\begin{align*}
    \mathbb{E}\left[\|\x_{k+1} - \x^*\|^2\mid \fyy{\x_k}\right] &\leq  \|\x_k-\x^*\|^2 -2\gamma_k(\x_k-\x^*)^T {\nabla f^{\bf imp}_{\eta_k}(\x_k)}+\gamma_k^2L_0^2n^2.
\end{align*}
Invoking the convexity of $f_{\eta_k}$, we obtain
\begin{align}\label{ineq:convex_imp_case_ineq1}
    \mathbb{E}\left[\|\x_{k+1} - \x^*\|^2\mid \fyy{\x_k}\right] &\leq  \|\x_k-\x^*\|^2 -2\gamma_k{\left(f_{\eta_k}^{\bf imp}(\x_k)-f_{\eta_k}^{\bf imp}(x^*)\right)}+\gamma_k^2L_0^2n^2.
\end{align}
Taking expectations from both sides of the preceding relation and rearranging the terms, we obtain
\begin{align*}
    2\gamma_k\left(\mathbb{E}\left[f^{\bf imp}_{\eta_k}(\x_k)\right]-f^{\bf imp}_{\eta_k}(\x^*)\right) &\leq\mathbb{E}\left[ \|\x_k-\x^*\|^2\right] -\mathbb{E}\left[\|\x_{k+1} - \x^*\|^2 \right] +\gamma_k^2L_0^2n^2.
\end{align*}
From the Lipschitzian property of the implicit function and Lemma \ref{lemma:props_local_smoothing} (v), we have that 
\begin{align}\label{ineq:convex_imp_case_ineq2}
    f^{\bf imp}_{\eta_k}(\x^*) \leq f^* +\eta_k L_0.
\end{align}
From the preceding two inequalities and that $f^{\bf imp}(\x_k) \leq f^{\bf imp}_{\eta_k}(\x_k)$, we obtain
\begin{align*}
    2\gamma_k\left(\mathbb{E}\left[{f^{\bf imp}(\x_k)}\right]-f^*\right) &\leq\mathbb{E}\left[ \|\x_k-\x^*\|^2\right] -\mathbb{E}\left[\|\x_{k+1} - \x^*\|^2 \right] +\gamma_k^2L_0^2n^2+2\gamma_k\eta_kL_0.
\end{align*}
{The rest of the proof {follows} in a similar fashion to that of Theorem \ref{thm:ZSOL_convex} (b).}

\noindent {\bf{(b)}} Under the specified setting, from part (a) we have
\begin{align*} 
\mathbb{E}\left[{f^{\bf imp}(\bar \x_K)}\right]-f^* &\leq 2\left(\tfrac{D_\Xscr }{\gamma_0}+L_0^2n^2\gamma_0\right)\tfrac{1}{\sqrt{K+1}}+\left(\tfrac{2\eta_0L_0}{0.5}\right)\tfrac{1}{\sqrt{K+1}}\\
&=2(nL_0\sqrt{D_\Xscr}+nL_0\sqrt{D_\Xscr})\tfrac{1}{\sqrt{K+1}}+\left(4nL_0\sqrt{D_\Xscr}\right)\tfrac{1}{\sqrt{K+1}} \\
& = \tfrac{8nL_0\sqrt{D_\Xscr}}{\sqrt{K+1}} \leq \epsilon.
\end{align*}
This implies the desired bound. 

\noindent {\bf{(c)}}  {The proof {follows} in a similar vein to that of Theorem \ref{thm:ZSOL_convex} (d).}
\end{proof}


{\subsection{Nonconvex single-stage SMPEC} \label{sec:nonconvex_single_stage}}
{In this \fyy{subsection}, in addressing \eqref{prob:mpec_exp_imp} in the nonconvex case, we consider a smoothed implicit problem given by the following.}
\begin{align}\label{prob:smoothed_implicit}
\begin{aligned}
        \min & \quad \fyy{f^{\bf imp}_{\eta}(\x)} \\
        \st & \quad \x \in \Xscr,
\end{aligned}
\end{align}
{where {$f_\eta^{\bf imp}$} is defined by \eqref{eqn:G-Smooth} for a given $\eta>0$. \\

\subsubsection{An inexact zeroth-order scheme}
In this subsection, we consider the {case} where an exact solution of the
lower-level problem is unavailable.} The outline of the proposed zeroth-order scheme is given by Algorithms \ref{algorithm:ZSOL_nonconvex}--\ref{algorithm:inexact_lower_level_SA_nonconvex}. We make the following assumptions \fyy{in these algorithms}.
\begin{assumption}\label{assum:u_iter_smooth_VR}\em
    Given a mini-batch {size of $N_k$} and a smoothing parameter $\eta>0$, let $\{v_{j,k}\}_{j=1}^{N_k} \in \mathbb{R}^n$ be $N_k$ iid replicates generated at epoch $k$ from \fyyy{the uniform distribution on} $\eta\mathbb{S}$ for all $k\geq 0$. Also, let the random realizations $\{\omega_{j,k}\}_{j=1}^{N_k}$ be iid replicates.
\end{assumption}
\begin{assumption}\label{assum:lower_level_stoch_nonconvex}\em
    Let the following hold and for all $k\geq 0$, $t \geq 0$, \fyy{$\hat \x_k \in \Xscr+\eta_k\mathbb{S}$}, and $\y_t \in \Yscr$.

\noindent (a) \fyyy{The replicates $\{G(\bullet,\bullet, \omega_{t})\}_{t=0}^{\infty}$ are generated randomly and are iid.}

\noindent (b)  $\mathbb{E}[G(\hat \x_k,\y_t,{\omega_{t}})\mid \hat \x_k,\y_t]  = F(\hat \x_k,\y_t)$ {holds almost surely.}

\noindent (c)  $\mathbb{E}[\|G(\hat \x_k,\y_t,{\omega_{t}})-F(\fyy{\hat{\x}_k},\y_t)\|^2 \mid \fyy{\hat{\x}_k},\y_t]  \leq  \nu_G^2$ holds almost surely for some $\nu_G>0$.
\end{assumption}
\fyy{Assumption~\ref{assum:lower_level_stoch_nonconvex} provides standard conditions on the first and second moment of the stochastic oracle. Such conditions have been assumed in the literature of the SA schemes extensively (e.g., see~\cite{nemirovski_robust_2009,Farzad1}).}
We utilize the following definition and lemma in the analysis in this subsection.
\begin{definition}[The residual mappings]\label{def:res_maps}
    {Suppose Assumption~\ref{ass-1} holds.} Given a scalar $\beta>0$ and a smoothing parameter $\eta>0$, for any $\x \in \mathbb{R}^n$, let the residual mapping $G_{\eta,\beta}(\x)$ and {its error-afflicted counterpart} $\tilde{G}_{\eta,\beta}(\x)$ be defined as 
\begin{align}
   G_{\eta,\beta}(\x) &\triangleq \beta {\left(\x - \Pi_{\Xscr}\left[\x - \tfrac{1}{\beta} \nabla_x {f^{\bf imp}_{\eta}(\x)}\right]\right)}, \\
\tilde{G}_{\eta,\beta}(\x) &\triangleq \beta\left( \x - \Pi_{\Xscr}\left[\x - \tfrac{1}{\beta} (\nabla_x {f^{\bf imp}_{\eta}(\x)}+\tilde{e}) \right]\right),
\end{align}
where $\tilde e \in \mathbb{R}^n$ is an arbitrary given vector.
\end{definition}
 {{It may be observed that $G_{\eta,\beta}$ is a residual for stationarity for the minimization of smooth nonconvex objectives over convex sets (cf.~\cite{beck17fom}). In fact, the first part of \eqref{equiv-stat-res} is a consequence of the  well known result relating the residual function $G_{\eta,\beta}(\x)$ to the standard stationarity condition (cf.~\cite[Thm.~9.10]{beck14introduction}) while the second implication in \eqref{equiv-stat-res} is Prop.~\ref{prop_equiv}.}  
\begin{lemma}\label{equiv-stat} \em Consider the problem~\eqref{prob:smoothed_implicit}. Then the following holds for any $\eta, \beta > 0$. 
    \begin{align}\label{equiv-stat-res}
    \left[ G_{\eta, \beta}(\x) = 0 \right]  \iff   \left[0 \in \nabla_{\x} {f^{\bf imp}_{\eta}(\x)}+ \Nscr_{\Xscr}(\x) \right] \implies \left[0 \in \partial_{2\eta} {f^{\bf imp}(\x)} + \Nscr_{\Xscr}(\x)\right]. 
\end{align}
\end{lemma}
Consequently, a zero of the  residual of the $\eta$-smoothed problem satisfies an $\eta$-approximate stationarity property for the original problem.}  
{The residual $\tilde{G}_{\eta,\beta}$ represents the counterpart of $G_{\eta,\beta}$ when employing an error-afflicted estimate of the gradient. {In fact, since our framework relies on sampling, leading to error, we obtain bounds on $\tilde{G}_{\eta,\beta}$. But it is still necessary to derive bounds on the original residual $G_{\eta,\beta}$ but this can be provided}  in terms of $\tilde{G}_{\eta,\beta}$ and $\tilde{e}$, the error in the gradient.}  

\begin{lemma}\label{lem:inexact_proj_2}\em
    Let Assumption~\ref{ass-1} hold. Then the following holds for any $\beta>0$, $\eta>0$, and $\x\in\mathbb{R}^n$.
\begin{align*}
    \|G_{\eta,\beta}(\x) \|^2 &  \leq  {2} \| \tilde{G}_{\eta,\beta}(\x) \|^2 + 2 \|\tilde{e}\|^2  .
\end{align*}
\end{lemma}
\begin{proof}
{From Definition \ref{def:res_maps}}, we may bound $G_{\eta,\beta}(\x)$ as follows.  
\begin{align*}
    \|G_{\eta,\beta}(\x)\|^2 & = \left\| \beta \left(\x - \Pi_{\Xscr}\left[\x - \tfrac{1}{\beta} \nabla_x {f^{\bf imp}_{\eta}(\x)}\right]\right) \right\|^2 \\
                  & = \left\| \beta \left( \x - \Pi_{\Xscr}\left[\x - \tfrac{1}{\beta} (\nabla_x {f^{\bf imp}_{\eta}(\x)}+\tilde{e})\right]\right) \right. \\
                  & + \left.  \beta  \Pi_{\Xscr}\left[\x - \tfrac{1}{\beta} (\nabla_x {f^{\bf imp}_{\eta}(\x)}+\tilde{e})\right]- \beta \Pi_{\Xscr}\left[\x - \tfrac{1}{\beta} \nabla_x {f^{\bf imp}_{\eta}(\x)}\right] \right\|^2 \\
                  & \leq {2}\left\| \beta \left(\x - \Pi_{\Xscr}\left[\x - \tfrac{1}{\beta} (\nabla_x {f^{\bf imp}_{\eta}(\x)}+\tilde{e})\right]\right)\right\|^2  \\
                  & + {2}\left\|  \beta  \Pi_{\Xscr}\left[\x - \tfrac{1}{\beta} (\nabla_x {f^{\bf imp}_{\eta}(\x)}+\tilde{e})\right]- \beta\Pi_{\Xscr}\left[\x - \tfrac{1}{\beta} \nabla_x {f^{\bf imp}_{\eta}(\x)}\right] \right\|^2 \\
                  & \leq {2} \|{\tilde{G}_{\eta,\beta}(\x)}\|^2 +{2} \|\tilde{e}\|^2,  
\end{align*} 
 {where the last inequality is a consequence of the non-expansivity of the Euclidean projector.}
\end{proof}
The proposed scheme can be compactly represented as follows.
\begin{align}\label{eqn:inexact_proj_scheme}
    \x_{k+1} := \Pi_{\Xscr} \left[ \x_k - \gamma\left( \nabla_{\x} {f^{\bf imp}_{\eta}(\x_k)}+e_k\right)\right],
\end{align}
 where we define the stochastic errors $e_k \triangleq  g_{\eta,N_k,\fyyy{\tilde{\epsilon}_k}}(\x_k)  - \nabla_{\x} {f^{\bf imp}_{\eta}(\x_k)}$ for all $k \geq 0$. We make use of the following result in the convergence analysis. 
\begin{algorithm}[H]
\caption{\fyy{\texttt{ZSOL$^{\bf 1s}_{\rm ncvx}$}: Variance-reduced zeroth-order method for nonconvex  \fyy{\eqref{eqn:SMPECepx}}}}\label{algorithm:ZSOL_nonconvex}
    \begin{algorithmic}[1]
        \STATE\textbf{input:} Given $\x_0 \in \Xscr$, ${\bar \x}_0: = \x_0$,  stepsize $\gamma>0$, smoothing parameter $\eta>0$, mini-batch sequence $\{N_k\}$ such that $N_k:=k +1$, an integer $K$, {a scalar $\lambda \in (0,1)$, and an integer $R$ randomly selected from $\{\lceil\lambda K\rceil ,\ldots,K\}$ using a uniform distribution}
    \FOR {$k=0,1,\ldots,{K}-1$}
        		    \STATE  \fyy{Do one of the following, depending on the type of the scheme.
    		    
\vspace{-.1in}    	
	    
\begin{itemize}
\item   Inexact scheme: Call Algorithm \ref{algorithm:inexact_lower_level_SA_nonconvex} to obtain $\y_\fyyy{\tilde{\epsilon}_k}(\x_k)$ 
\vspace{-.15in}    	

\item   Exact scheme: Evaluate $\y(\x_k)$
\end{itemize}}

\vspace{-.1in}    	
            \FOR {$j=1,\ldots,N_k$}
                    \STATE Generate {$v_{j,k} \in \eta \mathbb{S}$} 
    		    
    		      \STATE  {Do one of the following.
    		    
\vspace{-.1in}    	
	    
\begin{itemize}
\item   Inexact scheme: Call Algorithm \ref{algorithm:inexact_lower_level_SA_nonconvex} to obtain $\y_\fyyy{\tilde{\epsilon}_k}(\x_k+v_{j,k})$ 
\vspace{-.15in}    	

\item   Exact scheme: Evaluate $\y(\x_k+v_{j,k})$
\end{itemize}}

\vspace{-.1in}    

    		 \STATE Evaluate the inexact \fyy{or exact zeroth-order gradient approximation as follows.
 \begin{align}
              g_{\eta,\fyyy{\tilde{\epsilon}_k}}(\x_k,v_{j,k},\omega_{j,k}) &:=\tfrac{n\left(\fyy{\tilde f}(\x_k+v_{j,k}, \y_\fyyy{\tilde{\epsilon}_k}(\x_k+ v_{j,k}),\omega_{j,k}) - \fyy{\tilde f} (\x_k, \y_\fyyy{\tilde{\epsilon}_k}(\x_k),\omega_{j,k})\right)v_{j,k}}{\|v_{j,k}\|\eta} \tag{Inexact} \\
              g_{\eta}(\x_k,v_{j,k},\omega_{j,k}) &:=\tfrac{n\left(\fyy{\tilde f}(\x_k+v_{j,k}, \y(\x_k+ v_{j,k}),\omega_{j,k}) - \fyy{\tilde f} (\x_k, \y(\x_k),\omega_{j,k})\right)v_{j,k}}{\|v_{j,k}\|\eta}\tag{Exact}
\end{align}   	

}
      \ENDFOR
\STATE Evaluate the mini-batch zeroth-order gradient. \fyy{
 \begin{align}
             g_{\eta,N_k,\fyyy{\tilde{\epsilon}_k}}(\x_k) &:=\tfrac{\sum_{j=1}^{N_k} g_{\eta,\fyyy{\tilde{\epsilon}_k}}(\x_k,v_{j,k}\fyy{,\omega_{j,k}})}{N_k} \tag{Inexact} \\
             g_{\eta,N_k}(\x_k)  & :=\tfrac{\sum_{j=1}^{N_k} g_{\eta}(\x_k,v_{j,k}\fyy{,\omega_{j,k}})}{N_k} \tag{Exact}
\end{align}   	
}

\STATE {Update $\x_k$ as follows.}
         \fyy{\begin{align*}
             \x_{k+1} := \begin{cases}
                 \Pi_{\Xscr} \left[ \x_k - \gamma g_{\eta,N_k,\fyyy{\tilde{\epsilon}_k}}(\x_k) \right] & \hspace{3.1in} \mbox{(Inexact)} \\
                 \Pi_{\Xscr} \left[ \x_k - \gamma g_{\eta,N_k}(\x_k) \right] &  \hspace{3.2in} \mbox{(Exact)} 
             \end{cases}
     \end{align*}}
%
 
    \ENDFOR
        \STATE Return $\x_R$ 
   \end{algorithmic}
\end{algorithm}
\begin{algorithm}[H]
    \caption{{SA method for lower-level of \fyy{nonconvex \eqref{eqn:SMPECepx}}}}\label{algorithm:inexact_lower_level_SA_nonconvex}
    \begin{algorithmic}[1]
        \STATE \textbf{input:} An arbitrary $\y_0 \in \Yscr$,  vector $\hat{\x}_k$, and initial stepsize $\alpha_0>\frac{1}{2\mu_F}$ 
        \STATE Set $t_k := k+1$
      \FOR {$t=0,1,\ldots,t_k-1$}
      \STATE Generate a random realization of the stochastic mapping $G(\hat{\x}_k,{\y_t}, \omega_t)$
      \STATE {Update $\y_t$ as follows.}  $
 \y_{t+1} := \Pi_{\Yscr}\left[\y_t - \alpha_tG({\hat{\x}_k,\y_t},\omega_t)\right]$

\STATE Update the stepsize using $\alpha_{t+1} := \frac{\alpha}{t+\Gamma}$
        \ENDFOR
        \STATE Return $\y_{t_k}$ 
 
   \end{algorithmic}
\end{algorithm} 
\begin{lemma}\label{lem:descent_lemma_inexact}\em
Let Assumption~\ref{ass-1} \fyy{hold}. {Suppose $\x_k$ is generated by Algorithm \ref{algorithm:ZSOL_nonconvex} in which $\gamma \in (0,\frac{\eta}{n L_0})$ for a given $\eta>0$.} Then, we have for any $k$,  
    \begin{align*}
      {f^{\bf imp}_{\eta}(\x_{k+1})}  & \leq   {f^{\bf imp}_{\eta}(\x_k)}+ \left( -1+\tfrac{nL_0\gamma}{\eta}\right) \tfrac{\gamma}{4} \|G_{\eta,1/\gamma} (\x_k)\|^2 +{\left( 1-\tfrac{nL_0\gamma}{2\eta}\right)} {\gamma} \|e_k\|^2.
    \end{align*}
\end{lemma}
\begin{proof}
Note that by Lemma \ref{lemma:props_local_smoothing} (iv), {$\nabla f_{\eta}^{\bf imp}(\bullet)$} is Lipschitz with parameter $L\triangleq \frac{nL_0}{\eta}$. By the descent lemma, we have that 
    \begin{align*}
        {f^{\bf imp}_{\eta}(\x_{k+1})} & \leq {f^{\bf imp}_{\eta}(\x_k)} + \nabla_{\x} {f^{\bf imp}_{\eta}(\x_k)}^T(\x_{k+1}-\x_k) + \tfrac{L}{2}\|\x_{k+1}-\x_k\|^2 \\
                                & = {f^{\bf imp}_{\eta}(\x_k)}  + \left( \nabla_{\x} {f^{\bf imp}_{\eta}(\x_k)}+ e_k\right)^T(\x_{k+1}-\x_k) \\
                                &-  e_k^T(\x_{k+1}-\x_k)+ \tfrac{L}{2}\|\x_{k+1}-\x_k\|^2.
    \end{align*} 
 From the properties of the Euclidean projection, we have that 
\begin{align*}
   & (\x_k- \gamma (\nabla_x {f^{\bf imp}_{\eta}(\x_k)}+{e}_k))-\x_{k+1})^T(\x_k-\x_{k+1}) \leq 0\\
     \implies &(\nabla_x {f^{\bf imp}_{\eta}(\x_k)}+{e}_k))^T(\x_{k+1}-\x_k) \leq -\tfrac{1}{\gamma}\|\x_{k+1}-\x_k\|^2. 
\end{align*}
In addition, {for any $u,v \in \mathbb{R}^n$ we can write $u^Tv \leq \tfrac{1}{2}\left(\gamma\|u\|^2+\frac{\|v\|^2}{\gamma} \right)$. Thus,} we have that 
\begin{align*}
- {e}_k^T(\x_{k+1}-\x_k) \leq {\tfrac{\gamma}{2}} \|{e}_k\|^2 +  \tfrac{1}{2\gamma} \|\x_{k+1}-\x_k\|^2. 
\end{align*}
Consequently, {from the preceding  three inequalities} we have that 
 \begin{align*}
        {f^{\bf imp}_{\eta}(\x_{k+1})} & \leq  {f^{\bf imp}_{\eta}(\x_k)}   -\tfrac{1}{\gamma}\|\x_{k+1}-\x_k\|^2 +{\tfrac{\gamma}{2}} \|{e}_k\|^2 +  \tfrac{1}{2\gamma} \|\x_{k+1}-\x_k\|^2+ \tfrac{L}{2}\|\x_{k+1}-\x_k\|^2\\
        & =  {f^{\bf imp}_{\eta}(\x_k)} + \left( -\tfrac{1}{2\gamma}+\tfrac{L}{2}\right)\|\x_{k+1}-\x_k\|^2+{\tfrac{\gamma}{2}} \|{e}_k\|^2.
    \end{align*} 
    From $\gamma <\frac{1}{L}$, we have
    \begin{align*}
        {f^{\bf imp}_{\eta}(\x_{k+1})} & \leq {f^{\bf imp}_{\eta}(\x_k)}+ \left( -\tfrac{1}{2\gamma}+\tfrac{L}{2}\right) \|\x_{k+1}-\x\|^2 +  {\tfrac{\gamma}{2}}  \|e_k\|^2 \\
                                & = {f^{\bf imp}_{\eta}(\x_k)}+ \left( -\tfrac{1}{2\gamma}+\tfrac{L}{2}\right) \gamma^2  \|{\tilde{G}_{\eta,1/\gamma}(\x_k) }\|^2 +  {\tfrac{\gamma}{2}}  \|e_k\|^2 \\
                                & = {f^{\bf imp}_{\eta}(\x_k)}+ \left( -1+{L\gamma }\right) \tfrac{\gamma}{2} \|{\tilde{G}_{\eta,1/\gamma}(\x_k) }\|^2 + {\tfrac{\gamma}{2}}  \|e_k\|^2 \\
                                & \overset{\tiny \mbox{Lemma \ref{lem:inexact_proj_2}}}{\leq} {f^{\bf imp}_{\eta}(\x_k)}+ \left( -1+{L\gamma }\right)\tfrac{\gamma}{4}\|G_{\eta,1/\gamma} (\x_k)\|^2 +\left( 1-{L\gamma }\right)\tfrac{\gamma}{2} \|e_k\|^2 +   {\tfrac{\gamma}{2}} \|e_k\|^2 \\
                                & = {f^{\bf imp}_{\eta}(\x_k)}+ \left( -1+{L\gamma }\right) \tfrac{\gamma}{4} \|G_{\eta,1/\gamma} (\x_k)\|^2 +{\left( 1-\tfrac{L\gamma}{2}\right)} {\gamma} \|e_k\|^2.
    \end{align*} 
Substituting $L:= \frac{nL_0}{\eta}$ we obtain the desired inequality. 
\end{proof}
We make use of the following result in the convergence analysis. 
\begin{lemma}\label{lem:sublinear}\em Let $\{e_k\}$ be a non-negative sequence such that for an arbitrary non-negative sequence $\{\gamma_k\}$, the following relation is satisfied.
\begin{align}
e_{k+1}\leq (1-\alpha \gamma_k)e_k+\beta \gamma_k^2, \quad \hbox{for all } k\geq 0.
\end{align}
where $\alpha$ and $\beta$ are positive scalars. Suppose $\gamma_k=\tfrac{\gamma}{k+\Gamma}$ for any $k\geq 0$, where $\gamma>\tfrac{1}{\alpha}$ and $\Gamma>0$.  Then, we have 
\begin{align} 
e_k\leq \tfrac{\max \left\{\tfrac{\beta\gamma^2}{\alpha \gamma-1},\Gamma e_0\right\}}{k+\Gamma}, \qquad \hbox{for all } k\geq 0.
\end{align}
\end{lemma}

Next, we present the rate and complexity result for the proposed inexact method for addressing the nonconvex case.

\begin{theorem}[{\bf Rate and complexity {statements} for \fyy{inexact (\texttt{ZSOL$^{\bf 1s}_{\rm ncvx}$})}}]\label{thm:inexact_nonconvex}\em
    Consider \fy{Algorithms~\ref{algorithm:ZSOL_nonconvex}--\ref{algorithm:inexact_lower_level_SA_nonconvex} for solving \eqref{prob:mpec_exp_imp}} and suppose \fy{Assumptions~\ref{ass-1},~\ref{assum:u_iter_smooth_VR}, and~\ref{assum:lower_level_stoch_nonconvex}} hold. 

    \noindent {\bf{(a)}} Given $\hat \x_k \in \Xscr$, let $\y(\hat \x_k)$ denote the unique solution of $\mbox{VI}(\Yscr, {F(\hat \x_k,\bullet)})$.  Let $\y_{t_k} $ be generated by Algorithm \ref{algorithm:inexact_lower_level_SA_nonconvex} \fyy{where $t_k:=k+1$}. Let us define $C_F \triangleq \max_{\x \in X,\ \y \in \Yscr} \|F(\x,\y)\|$. {Then for all  $t_k \geq 0$, we have}
\begin{align*}
 \mathbb{E}[\|\y_{t_k} - \y(\hat \x_k)\|^2] \leq \fyyy{\tilde{\epsilon}_k} \triangleq  \tfrac{\max \left\{\tfrac{(C_F^2+\nu_G^2)\alpha^2}{2\alpha\mu_F-1},\Gamma \sup_{\y \in \Yscr}\|\y-\y_0\|^2\right\}}{t_k+\Gamma}.
\end{align*}

\noindent {\bf{(b)}} The following holds for any $\gamma<\frac{\eta}{n L_0}$, \fy{$\ell\triangleq  \lceil \lambda K\rceil$}, and all $K> \fy{\tfrac{2}{1-\lambda}}$.
\begin{align*}
 \mathbb{E}\left[ \|G_{\eta,1/\gamma} (\x_R)\|^2\right] \leq \frac{n^2{\gamma}(\fy{1-2\ln(\lambda)}) {\left( 1-\tfrac{nL_0\gamma}{2\eta}\right)} \left(  \fyy{\tfrac{8\tilde{L}^2_0(C_F^2+\nu_G^2)}{\eta^2\mu_F^2}}+L_0^2\right) +  \fy{\mathbb{E}\left[\fyy{f^{\bf imp}(\x_{\ell})}\right] }{-f^*} +2L_0\eta}{\left( 1-\tfrac{nL_0\gamma}{\eta}\right) \tfrac{\gamma}{4}\fy{(1-\lambda)K}  }.
\end{align*}

\noindent {\bf{(c)}}  Suppose $\gamma= \tfrac{\eta}{2nL_0}$ and $\eta=\tfrac{1}{L_0}$. Let $\epsilon>0$ be an arbitrary scalar and $K_{\epsilon}$ be such that $ \mathbb{E}\left[ \|G_{\eta,1/\gamma} (\x_R)\|^2\right]   \leq \epsilon$. Then,

\noindent {{(c-1)}} The total number of upper-level projection steps on $\Xscr$ is \fy{$K_{\epsilon}=\mathcal{O}\left(n^2L_0^2\fyy{\tilde{L}^2_0}\epsilon^{-1}\right)$}.

 \noindent {{(c-2)}} The total sample complexity of upper-level is \fy{$\mathcal{O}\left(n^4L_0^4\fyy{\tilde{L}^4_0} \epsilon^{-2}\right)$}.
 
\noindent {{(c-3)}} The total number of lower-level projection steps on $\Yscr$ is \fy{$ 
\mathcal{O}\left(n^6L_0^6 \fyy{\tilde{L}^6_0}\epsilon^{-3}\right)$}.

 \noindent {{(c-4)}} The total sample complexity of lower-level is \fy{$\mathcal{O}\left(n^6L_0^6\fyy{\tilde{L}^6_0} \epsilon^{-3}\right)$}.

\end{theorem}
\begin{proof}
\noindent {\bf (a)} Let us define the errors $\Delta_t \triangleq G(\hat{\x}_k,\fyyy{\y_t},\omega_t) - F(\hat{\x}_k,\fyyy{\y_t})$ for $t\geq 0$. 
We have
\begin{align*}
& \quad \|\y_{t+1}- \y(\hat x_k)\|^2 = \|\Pi_{\Yscr}\left[\y_t - \alpha_tG(\hat{\x}_k,\us{\y_t},\omega_t)\right] - \Pi_{\Yscr}\left[\y(\hat \x_k)\right] \|^2 \leq  \|\y_t - \alpha_tG(\hat{\x}_k,\us{\y_t},\omega_t)-\y(\hat \x_k) \|^2\\
& =  \|\y_t - \alpha_tF(\hat{\x}_k,\us{\y_t})-\alpha_t\Delta_t-\y(\hat \x_k) \|^2\\
& = \|\y_t - \y(\hat \x_k)\|^2 + \alpha_t^2\|F(\hat{\x}_k,\us{\y_t})\|^2 +\alpha_t^2\|\Delta_t\|^2 - 2\alpha_t(\y_t - \y(\hat \x_k))^TF(\hat{\x}_k,\us{\y_t}) \\
&-2\alpha_t(\y_t - \y(\hat \x_k)-\alpha_tF(\hat{\x}_k,\us{\y_t}))^T\Delta_t.
\end{align*}
Taking conditional expectations from the preceding relation and invoking \fy{Assumption~\ref{assum:lower_level_stoch_nonconvex}}, 
we obtain
\begin{align*}
   \mathbb{E}[\|\y_{t+1}- \y(\hat \x_k)\|^2\mid \fyy{\hat{\x}_k,\y_t}] & \leq \|\y_t - \y(\hat \x_k)\|^2 +\alpha_t^2(C_F^2+\nu_G^2)- 2\alpha_t(\y_t - \y(\hat \x_k))^TF(\hat{\x}_k,\fyyy{\y_t}).
\end{align*}
From strong monotonicity of mapping \us{$F(\hat \x_k,\bullet)$ uniformly in $\hat \x_k$} and the definition of $\y(\hat x_k)$, we have
\begin{align*}
(\y_t - \y(\hat \x_k))^TF(\hat{\x}_k,\fyyy{\y_t}) \geq (\y_t - \y(\hat \x_k))^TF(\y(\hat x_k),\hat{\x}_k)+\mu_F\|\y_t - \y(\hat \x_k)\|^2 \geq \mu_F\|\y_t - \y(\hat \x_k)\|^2.
\end{align*}
From the preceding relations, we obtain 
\begin{align*}
   \mathbb{E}[\|\y_{t+1}- \y(\hat \x_k)\|^2\mid \fyy{\hat{\x}_k,\y_t}] & \leq (1-2\mu_F\alpha_t)\|\y_t - \y(\hat \x_k)\|^2 +\alpha_t^2(C_F^2+\nu_G^2).
\end{align*}
Taking expectations from both sides, we have
\begin{align*}
   \mathbb{E}[\|\y_{t+1}- \y(\hat \x_k)\|^2] & \leq (1-2\mu_F\alpha_t)   \mathbb{E}[\|\y_t - \y(\hat \x_k)\|^2] +\alpha_t^2(C_F^2+\nu_G^2).
\end{align*}
Noting that in Algorithm \ref{algorithm:inexact_lower_level_SA_nonconvex} we have $\alpha_0 >\frac{1}{2\mu_F}$, using Lemma \ref{lem:sublinear}, we obtain that 
\begin{align*}
\mathbb{E}[\|\y_t - \y(\hat \x_k)\|^2]& \leq \tfrac{\max \left\{\tfrac{(C_F^2+\nu_G^2)\alpha^2}{2\alpha\mu_F-1},\Gamma \sup_{\y \in \Yscr}\|\y-\y_0\|^2\right\}}{t+\Gamma}, \qquad \hbox{for all } t \geq 0.
\end{align*}

        \noindent {\bf (b)} 
We can write 
 \begin{align}\label{ineq:nonc_inexact_b}
& \quad \mathbb{E}\left[\| e_k\|^2\mid \fyy{\x_k}\right]  =\mathbb{E}\left[\left\|g_{\eta,N_k,\fyyy{\tilde{\epsilon}_k}}(\x_k)  - \nabla_{\x} {f^{\bf imp}_{\eta}(\x_k)} \right\|^2 \mid \fyy{\x_k}\right]\notag\\
&=  \mathbb{E}\left[\left\|\tfrac{\sum_{j=1}^{N_k} g_{\eta,\fyyy{\tilde{\epsilon}_k}}(\x_k,v_{j,k},\omega_{j,k})}{N_k} - \nabla_{\x} {f^{\bf imp}_{\eta}(\x_k)} \right\|^2\mid \fyy{\x_k}\right] \notag\\
& \leq 2\mathbb{E}\left[\left\|\tfrac{\sum_{j=1}^{N_k} g_{\eta,\fyyy{\tilde{\epsilon}_k}}(\x_k,v_{j,k},\omega_{j,k})}{N_k} -\tfrac{\sum_{j=1}^{N_k} g_{\eta}(\x_k,v_{j,k},\omega_{j,k})}{N_k}\right\|^2\mid \fyy{\x_k}\right] +2\mathbb{E}\left[\left\|\tfrac{\sum_{j=1}^{N_k} g_{\eta}(\x_k,v_{j,k},\omega_{j,k})}{N_k} - \nabla_{\x} {f^{\bf imp}_{\eta}(\x_k)} \right\|^2\mid \fyy{\x_k}\right]\notag\\
& \leq \tfrac{2\sum_{j=1}^{N_k} \mathbb{E}\left[\left\|g_{\eta,\fyyy{\tilde{\epsilon}_k}}(\x_k,v_{j,k},\omega_{j,k})-g_{\eta}(\x_k,v_{j,k},\omega_{j,k}) \right\|^2\mid \fyy{\x_k}\right]}{N_k}+ \tfrac{2\sum_{j=1}^{N_k} \mathbb{E}\left[\left\|{g_{\eta}}(\x_k,v_{j,k},\omega_{j,k})-\nabla_{\x} {f^{\bf imp}_{\eta}(\x_k)} \right\|^2\mid \fyy{\x_k}\right]}{N_k^2}\notag\\
&\leq  \tfrac{{8\tilde{L}^2_0n^2}\fyyy{\tilde{\epsilon}_k}}{\eta^2}+\tfrac{2\sum_{j=1}^{N_k}\left( \mathbb{E}\left[\left\|{g_{\eta}}(\x_k,v_{j,k},\omega_{j,k})\right\|^2\mid \fyy{\x_k}\right]-\left\|\nabla_{\x} {f^{\bf imp}_{\eta}(\x_k)} \right\|^2\right)}{N_k^2}\notag\\
&\leq  \tfrac{{8\tilde{L}^2_0n^2}\fyyy{\tilde{\epsilon}_k}}{\eta^2}+\tfrac{2n^2L_0^2}{N_k},
\end{align}  
where in the second inequality, the first term is implied by the relation $\left\|\sum_{i=1}^mu_i\right\|^2 \leq m\sum_{i=1}^m\left\|u_i\right\|^2$ for any $u_i\in \mathbb{R}^n$ for all $i=1,\ldots,m$. The second term in the second inequality is implied by noting that from Lemma \ref{lem:smooth_grad_properties}, $g_{\eta}(\x_k,v)$ is an unbiased estimator of $\nabla_{\x} {f^{\bf imp}_{\eta}(\x_k)}$. {The third inequality is obtained using \fyyy{Lemma \ref{lem:inexact_error_bounds}}.} From \fyy{Lemma \ref{lem:descent_lemma_inexact}} we have 
\begin{align*}
      \left( 1-\tfrac{nL_0\gamma}{\eta}\right) \tfrac{\gamma}{4} \|G_{\eta,1/\gamma} (\x_k)\|^2 & \leq   {f^{\bf imp}_{\eta}(\x_k)}-{f^{\bf imp}_{\eta}(\x_{k+1})}   +{\left( 1-\tfrac{nL_0\gamma}{2\eta}\right)} {\gamma} \|e_k\|^2.
\end{align*}
Let {$f^{\bf imp,*}_\eta\triangleq \inf_{x \in \Xscr} f_\eta^{\bf imp}(\x)$}. Summing {the preceding relation} from \fy{$k =\ell, \ldots, K-1$ where $\ell\triangleq  \lceil \lambda K\rceil$}, we have that 
\begin{align*}
 \left( 1-\tfrac{nL_0\gamma}{\eta}\right) \tfrac{\gamma}{4}\sum_{\fy{k=\ell}}^{K-1} \|G_{\eta,1/\gamma} (\x_k)\|^2 & \leq  { f^{\bf imp}_{\eta}(\x_\ell)}-{ f^{\bf imp}_{\eta}(\x_K)}   +{\left( 1-\tfrac{nL_0\gamma}{2\eta}\right)} {\gamma} \sum_{k=\fy{\ell}}^{K-1}\|e_k\|^2.
\end{align*}
Taking expectations from the both sides, it follows that 
\begin{align*}
    & \quad \left( 1-\tfrac{nL_0\gamma}{\eta}\right) \tfrac{\gamma}{4}\fy{(K-\ell)}\mathbb{E}\left[ \|G_{\eta,1/\gamma} (\x_R)\|^2\right]   \leq {\left( 1-\tfrac{nL_0\gamma}{2\eta}\right)} {\gamma}\sum_{k=\fy{\ell}}^{K-1} \mathbb{E}\left[\|{e}_k\|^2\right] + \fy{\mathbb{E}\left[{ f^{\bf imp}_{\eta}(\x_\ell)}\right] }-  {f^{\bf imp,*}_\eta}\\
    & \leq {\left( 1-\tfrac{nL_0\gamma}{2\eta}\right)} {\gamma} \sum_{k=\fy{\ell}}^{K-1} \mathbb{E}\left[\|{e}_k\|^2\right] + \fy{  \mathbb{E}\left[ {f^{\bf imp}(\x_\ell) }+ { f^{\bf imp}_{\eta}(\x_\ell)}-{f^{\bf imp}(\x_\ell) }\right]}- {{f^{\bf imp,*}_\eta} +f^*-f^*}\\
    & \leq \left( 1-\tfrac{nL_0\gamma}{2\eta}\right) \gamma\sum_{k=\fy{\ell}}^{K-1} \mathbb{E}\left[\|{e}_k\|^2\right]+  \mathbb{E}\left[{ f^{\bf imp}(\x_\ell)}\right]-f^*+ \mathbb{E}\left[\left |  f_\eta^{\bf imp}(\x_\ell)- f^{\bf imp}(\x_\ell)\right|\right] + \left| f^*- f^{\bf imp,*}_\eta \right| \\
                            & \leq{\left( 1-\tfrac{nL_0\gamma}{2\eta}\right)} {\gamma} \sum_{k=\fy{\ell}}^{K-1}\left( \tfrac{{8\tilde{L}^2_0n^2}\fyyy{\tilde{\epsilon}_k}}{\eta^2}+\tfrac{2n^2L_0^2}{N_k}\right) +  \fy{\mathbb{E}\left[{f^{\bf imp}(\x_{\ell})}\right] }{-f^*} +2L_0\eta,
\end{align*} 
 where the preceding relation is implied by invoking the bound on $\mathbb{E}\left[\| e_k\|^2\right]$ and Lemma \ref{lemma:props_local_smoothing} (iii). Note that from part (a), we have $\fyyy{\tilde{\epsilon}_k}=   \frac{2(C_F^2+\nu_G^2)}{\mu_F^2t_k}$ where $t_k:=k+1$. Also, $N_k:=k+1$. \fy{Note that $K> {\tfrac{2}{1-\lambda}}$ implies $\ell \leq K-1$.} From Lemma~\ref{lem:harmonic_bounds}, \fy{using $\ell \geq 1$ we have $ \sum_{k=\fy{\ell}}^{K-1}\frac{1}{k+1}\leq  \frac{1}{\ell+1}+\ln\left( \frac{K}{\ell+1}\right) \leq 0.5 +\ln\left(\tfrac{N}{\lambda N+1}\right)\leq 0.5-\ln(\lambda)$. Also, $K-\ell \geq K-\lambda K=(1-\lambda)K$.} Thus, we obtain 
\begin{align*}
 \mathbb{E}\left[ \|G_{\eta,1/\gamma} (\x_R)\|^2\right] \leq \frac{{\left( 1-\tfrac{nL_0\gamma}{2\eta}\right)} 2n^2{\gamma} \left( {\tfrac{8\tilde{L}^2_0(C_F^2+\nu_G^2)}{\eta^2\mu_F^2}}+L_0^2\right)(\fy{0.5-\ln(\lambda)}) +  \fy{\mathbb{E}\left[{f^{\bf imp}(\x_{\ell})}\right] }{-f^*} +2L_0\eta}{\left( 1-\tfrac{nL_0\gamma}{\eta}\right) \tfrac{\gamma}{4}\fy{(1-\lambda)K}  }.
\end{align*}

\noindent {{\bf (c)} To show (c-1), using the relation in part (b) and substituting $\gamma= \tfrac{\eta}{2nL_0}$ we obtain
\begin{align*}
 \mathbb{E}\left[ \|G_{\eta,1/\gamma} (\x_R)\|^2\right] \leq  \frac{6n^2({1-2\ln(\lambda)}) \left( \fyy{\tfrac{8\tilde{L}^2_0(C_F^2+\nu_G^2)}{\eta^2\mu_F^2}}+L_0^2\right)+ \tfrac{16nL_0}{\eta}(\sup_{\x \in \Xscr}\fyy{f^{\bf imp}(\x)}-f^*) +32nL_0^2}{(1-\lambda)K  }.
\end{align*}
Further, from $\eta=\tfrac{1}{L_0}$ we obtain 
\begin{align*}
 \mathbb{E}\left[ \|G_{\eta,1/\gamma} (\x_R)\|^2\right] 
 & \leq \frac{6n^2L_0^2({1-2\ln(\lambda)}) \left( \fyy{\tfrac{8\tilde{L}^2_0(C_F^2+\nu_G^2)}{\mu_F^2}}+1\right)+ 16nL_0^2(\sup_{\x \in \Xscr}\fyy{f^{\bf imp}(\x)}-f^*) +32nL_0^2}{(1-\lambda)K  }.
\end{align*}
This implies that $ \mathbb{E}\left[ \|G_{\eta,1/\gamma} (\x_R)\|^2\right] \leq \frac{\mathcal{O}\left(n^2L_0^2\fyy{\tilde{L}^2_0}\right)}{K}$ and thus, we obtain \fy{$K_{\epsilon}=\mathcal{O}\left(n^2L_0^2\fyy{\tilde{L}^2_0} \epsilon^{-1}\right)$}. Next, we show (c-2). The total sample complexity of upper-level is as follows.
\begin{align*}
\sum_{k=0}^{K_\epsilon} N_k = \sum_{k=0}^{K_\epsilon} (k+1) =\mathcal{O} \mathcal(K_\epsilon^2)= \fy{\mathcal{O}\left(n^4L_0^4 \epsilon^{-2}\right).}
\end{align*}
To show (c-3), note that the total number of lower-level projection steps is given by 
\begin{align*}
\sum_{k=0}^{K_\epsilon} (1+N_k)t_k = \sum_{k=0}^{K_\epsilon} (k+1)(k+2) =\fy{\mathcal{O} \mathcal(K_\epsilon^3)= \mathcal{O}\left(n^6L_0^6 \epsilon^{-3}\right).}
\end{align*}
Noting that at each iteration in Algorithm \ref{algorithm:inexact_lower_level_SA_nonconvex} a single sample is taken, we obtain the bound in (c-4).}
\end{proof}
\fyyy{\begin{remark}[{\bf Variance-reduction and smoothing schemes in the nonconvex case}]\em
\begin{enumerate}
\item[]
\item[(i)] Unlike in (\texttt{ZSOL$^{\bf 1s}_{\rm cvx}$}), in (\texttt{ZSOL$^{\bf 1s}_{\rm ncvx}$}) we employ a variance-reduction scheme in the upper-level. This is mainly because, in contrast with the convex case, the use of the Euclidean projection in (\texttt{ZSOL$^{\bf 1s}_{\rm ncvx}$}) leads to the presence of the persistent error term ${\left( 1-\tfrac{nL_0\gamma}{2\eta}\right)} {\gamma} \|e_k\|^2$ (see Lemma~\ref{lem:descent_lemma_inexact}). The use of variance-reduction helps with contending with this error in establishing the convergence and rate results.
\item[(ii)]  Unlike in (\texttt{ZSOL$^{\bf 1s}_{\rm cvx}$}), in (\texttt{ZSOL$^{\bf 1s}_{\rm ncvx}$}) we employ a constant smoothing parameter. This is because assuming an iteratively updating smoothing parameter $\eta_k$ in the nonconvex case does not seem to allow for constructing a recursive error bound. For this reason, in the nonconvex case we limit our study to the case when the smoothing parameter is constant.
\end{enumerate}
\end{remark}}
{\subsubsection{An exact zeroth-order scheme}
In this subsection, we present the rate and complexity results for the exact variant of Algorithm~\ref{algorithm:ZSOL_nonconvex}.
\begin{corollary}[{\bf Rate and complexity statements for exact \fyy{(\texttt{ZSOL$^{\bf 1s}_{\rm ncvx}$})}}]\label{cor:exact_nonconvex}\em
    Consider {Algorithms~\ref{algorithm:ZSOL_nonconvex} (exact variant) for solving \eqref{prob:mpec_exp_imp}} and suppose {Assumptions~\ref{ass-1} and~\ref{assum:u_iter_smooth_VR}} hold. 

\noindent {\bf{(a)}} The following holds for any $\gamma<\frac{\eta}{n L_0}$, {$\ell\triangleq  \lceil \lambda K\rceil$}, and all $K> {\tfrac{2}{1-\lambda}}$.
\begin{align*}
 \mathbb{E}\left[ \|G_{\eta,1/\gamma} (\x_R)\|^2\right] \leq \frac{n^2L_0^2{\gamma}({0.5-\ln(\lambda)}) {\left( 1-\tfrac{nL_0\gamma}{2\eta}\right)} +  {\mathbb{E}\left[\fyy{f^{\bf imp}(\x_{\ell})}\right] }{-f^*} +2L_0\eta}{\left( 1-\tfrac{nL_0\gamma}{\eta}\right) \tfrac{\gamma}{4}{(1-\lambda)K}  }.
\end{align*}

\noindent {\bf{(b)}}  Suppose $\gamma= \tfrac{\eta}{2nL_0}$ and $\eta=\tfrac{1}{L_0}$. Let $\epsilon>0$ be an arbitrary scalar and $K_{\epsilon}$ be such that $ \mathbb{E}\left[ \|G_{\eta,1/\gamma} (\x_R)\|^2\right]   \leq \epsilon$. Then {the following hold.}

\noindent {{(b-1)}} The total number of upper-level projection steps on $\Xscr$ is {$K_{\epsilon}=\mathcal{O}\left(n^2L_0^2 \epsilon^{-1}\right)$}.

 \noindent {{(b-2)}} The total sample complexity of upper-level is {$\mathcal{O}\left(n^4L_0^4 \epsilon^{-2}\right)$}.

\end{corollary}
\begin{proof}
The proof can be {carried out} in a similar vein to that of Theorem \ref{thm:inexact_nonconvex} by noting that $\fyyy{\tilde{\epsilon}_k}:=0$ in the exact variant. The main difference lies in establishing the upper bound on $\mathbb{E}\left[\| e_k\|^2\mid\fyy{\x_k}\right]$ in \eqref{ineq:nonc_inexact_b}. To be precise, we derive this bound in the following.  
 \begin{align*} 
& \quad  \mathbb{E}\left[\| e_k\|^2\mid \fyy{\x_k}\right] =\mathbb{E}\left[\left\|g_{\eta,N_k}(\x_k)  - \nabla_{\x} {f^{\bf imp}_{\eta}(\x_k)} \right\|^2 \mid \fyy{\x_k}\right]\\
 &=  \mathbb{E}\left[\left\|\tfrac{\sum_{j=1}^{N_k} g_{\eta}(\x_k,v_{j,k},\omega_{j,k})}{N_k} - \nabla_{\x} {f^{\bf imp}_{\eta}(\x_k)} \right\|^2\mid \fyy{\x_k}\right]  \leq \tfrac{\sum_{j=1}^{N_k} \mathbb{E}\left[\left\|{g_{\eta}}(\x_k,v_{j,k},\omega_{j,k})-\nabla_{\x} \fyy{f^{\bf imp}_{\eta}(\x_k)} \right\|^2\mid \fyy{\x_k}\right]}{N_k^2}\\
&\leq  \tfrac{\sum_{j=1}^{N_k}\left( \mathbb{E}\left[\left\|{g_{\eta}}(\x_k,v_{j,k},\omega_{j,k})\right\|^2\mid \fyy{\x_k}\right]-\left\|\nabla_{\x} {f^{\bf imp}_{\eta}(\x_k)} \right\|^2\right)}{N_k^2}\leq  \tfrac{n^2L_0^2}{N_k}. 
\end{align*} 
\end{proof}

}

\section{Zeroth-order methods for two-stage \uss{SMPECs}}\label{sec:2s}
{In this section, we extend the zeroth-order schemes \uss{from the previous section to allow for accommodating} two-stage model~\eqref{prob:mpec_as_imp}. In Section~\ref{sec:4.1}, we discuss an implicit framework for two-stage SMPECs \uss{and} present inexact and exact schemes and an accelerated counterpart in Sections~\ref{sec:4.2} and Section~\ref{sec:4.3}. \uss{We} conclude with a discussion of  \uss{addressing nonconvexity in the implicit problem} in Section~\ref{sec:4.4}.}

\subsection{An implicit framework}\label{sec:4.1}
{Consider the implicit problem \eqref{prob:mpec_as_imp}. Given the function $f^{\bf imp}(\x)$ and a scalar $\eta$, we consider a spherical smoothing of $f^{\bf imp}_{\eta}(\x)$ as follows: 
\begin{align}\label{eqn:G-Smooth-as}\tag{G-Smooth{$^{\bf 2s}$}}
   f^{\bf imp}_{\eta}(\x) \triangleq \mathbb{E}_{u\in \mathbb{B}}[f^{\bf imp}(\x+\eta u)]
    = \mathbb{E}_{u\in \mathbb{B}}[\mathbb{E}[\tilde{f}(\x+\eta u,\y(\x+\eta u,\omega),\omega)]].
\end{align}
Similar to the single-stage case discussed in subsection \ref{sec:3.1}, the zeroth-order approximation of the gradient is given by \eqref{eqn:g_eta}. An unbiased estimate of $g_{\eta}(\x)$ is defined as 
\begin{align}\label{eqn:g_v_eta_as}
 g_{\eta}(\x,v,\omega) \triangleq \left( \frac{n}{\eta} \right)\left[\frac{\left(\tilde f(\x+v, \y(\x+v,\omega),\omega) - \tilde f(\x, \y(\x,\omega),\omega)\right)v}{\|v\|} \right]. 
\end{align}
Given a vector $\x_0 \in \Xscr$, we may employ \eqref{eqn:g_v_eta_as} in 
constructing a sequence $\{\x_k\}$ where $\x_k$ satisfies the following projected stochastic gradient update. 
\begin{align}\label{eqn:fixed_smoothing_scheme_as}
    \x_{k+1} := \Pi_{\Xscr} \left[ \x_k - \gamma_k {g}_{\eta}(\x_k,v_k,\omega_k) \right].
\end{align}}

{\begin{lemma}[\bf {Properties of the two-stage exact zeroth-order gradient}]\label{lem:smooth_grad_properties_two_stage}\em
Suppose Assumption~\ref{ass-1} (b) holds.  Consider \eqref{prob:mpec_as_imp}. Given $\x \in \Xscr$ and $\eta>0$,  consider the stochastic zeroth-order mapping $g_\eta(\x,v,\omega)$ defined by \eqref{eqn:g_v_eta_as} for $v \in \eta\mathbb{S}$ and $k\geq 0$, where $v$ and $\omega$ are independent.
 Then,  ${\nabla f^{\bf imp}_{\eta}(\x)} =\mathbb{E}[g_{\eta}(\x,v,\omega)\mid \x]$ and $\mathbb{E}[\|g_{\eta}(\x,v,\omega)\|^2\mid \x] \leq {L_0^2n^2}$ {almost surely} for all $k\geq 0$. 
\end{lemma}
\begin{proof}
The proof is similar to the proof of Lemma \ref{lem:smooth_grad_properties}. We provide the details for the sake of completeness. From \eqref{eqn:g_v_eta_as} and that $f^{\bf imp}(\x)  \triangleq \mathbb{E}[\tilde{f}(\x,\y(\x,\omega),\omega)]$ we can write 
\begin{align*}
\mathbb{E}[g_{\eta}(\x,v,\omega)\mid \x]&
=\mathbb{E}_{v \in \eta\mathbb{S}}\left[\left(\tfrac{n}{\eta}\right) \frac{\left(f^{\bf imp}(\x+v) - f^{\bf imp}(\x)\right)v}{\|v\|}\mid \x\right]\\
&=\left(\tfrac{n}{\eta}\right)\mathbb{E}_{v \in \eta\mathbb{S}}\left[ f^{\bf imp}(\x+v) \frac{ v}{\|v\|}\mid \x\right]\overset{\tiny \mbox{Lemma }\ref{lemma:props_local_smoothing} (i)}{=}  \nabla f^{\bf imp}_{\eta}(\x) .
\end{align*}
We have
\begin{align*}
\mathbb{E}[\|g_{\eta}(\x,v,\omega)\|^2\mid \x,\omega]  
 &=\left(\tfrac{n}{\eta}\right)^2\mathbb{E}\left[\left\| \tfrac{\left(\fyy{\tilde f}(\x+v, \y(\x+v,\omega),\omega) - {\tilde f}(\x, \y(\x,\omega),\omega)\right)v}{\|v\|}\right\|^2\mid \x,\omega\right]  \\
 & =\left(\tfrac{n}{\eta}\right)^2 \int_{\eta \mathbb{S}} \tfrac{\left\| \left({\tilde f}(\x+v, \y(\x+v,\omega),\omega) - {\tilde f}(\x, \y(\x,\omega),\omega)\right)v\right\|^2}{\|v\|^2}  p_v(v) dv \\
            &\overset{\tiny \mbox{Assumption }\ref{ass-1} (b.i)}{\leq} \frac{n^2}{\eta^2} \int_{\eta \mathbb{S}} L_0^2(\omega) \| v\|^2  p_v(v) dv \leq n^2 L_0^2(\omega)\int_{\eta \mathbb{S}} p_v(v) dv = n^2 L_0^2(\omega).
\end{align*}
Taking the expectation with respect to $\omega$ from the both sides of the preceding inequality and invoking $L_0^2 \triangleq \mathbb{E}[L_0^2(\omega)]<\infty$, we obtain the desired bound. 
\end{proof}}

\subsection{Inexact and exact schemes for convex regime}\label{sec:4.2}
{Consider the implicit form of \eqref{prob:mpec_as_imp} where $\y(\x,\omega)$ solves $\mbox{VI}(\Yscr,G(\x,\bullet,\omega))$ for almost every $\omega \in \Omega$. Computing such $\y(\x,\omega)$ is often challenging, in particular, when {$\mathcal{Y}$ is high-dimensional}. To contend with this challenge, we employ gradient-like methods for computing inexact solutions to the lower-level {$\omega$-specific VI parametrized by $\x$, denoted by VI$(\mathcal{Y}, G(\x,\bullet,\bullet))$}. We consider the case where we have access to an approximate solution $\y_\fyyy{\tilde{\epsilon}_k}(\x_k,\omega)$ such that 
\begin{align}
     \|\y_\fyyy{\tilde{\epsilon}_k} (\x_k,\omega) - \y(\x_k,\omega)\|^2  \leq \fyyy{\tilde{\epsilon}_k}, \mbox{ where } \y(\x_k,\omega) \in \mbox{SOL}(\Yscr, G(\x_k,\bullet,\omega)).  
\end{align}
Similar to the single-stage case, we may define an inexact zeroth-order gradient mapping $g_{\eta,\fyyy{\tilde{\epsilon}}}(\x,{v},\omega)$ as follows.
\begin{align}\label{def-g-eta-eps_as}
{g_{\eta,\fyyy{\tilde{\epsilon}}}(\x,v,\omega)  \triangleq } \frac{n({\tilde f}(\x+ v, \y_\fyyy{\tilde{\epsilon}}(\x+ v,\omega),\omega) - {\tilde f}(\x, \y_\fyyy{\tilde{\epsilon}}(\x,\omega),\omega))v}{\|v\|\eta}, 
\end{align}
where $v \in \eta\mathbb{S}$ and $\y_\fyyy{\tilde{\epsilon}_k}(\x_k,\omega)$ is an output of a gradient-like scheme. The outline of the proposed zeroth-order solver is presented in Algorithm~\ref{algorithm:inexact_zeroth_order_SVIs_2S} while an inexact approximation of $\y(\x,\omega)$ is computed by Algorithm~\ref{algorithm:Inexact_lower_level_SA_twostage}. In the following, we extend Lemma \ref{lem:smooth_grad_properties} to the two-stage regime.}

\fyyy{\begin{remark}\em 
Throughout the algorithms in this section, in evaluation of the exact and inexact solution to the lower level problem, denoted by $\y(\bullet, \omega)$ and $\y_{\tilde{\epsilon}}(\bullet, \omega)$, respectively, we assume that we have access to an oracle that returns random replicates of $\omega$. 
\end{remark}}
{\begin{lemma}[\bf Properties of the two-stage inexact zeroth-order gradient]{\label{lem:inexact_error_bounds_2s}}\em Suppose {Assumption~\ref{ass-1}} (b) holds.  Consider \eqref{prob:mpec_as_imp}. {Let} $g_{\eta,\fyyy{\tilde{\epsilon}}}(\x,v,\omega)$ be defined as \eqref{def-g-eta-eps_as} for $\omega \in \Omega$ and $v \in \eta\mathbb{S}$ for $\eta, \fyyy{\tilde{\epsilon}} >0$. Suppose $\|\y_{\fyyy{\tilde{\epsilon}}}(\x,\omega)-\y(\x,\omega)\|^2   \leq \fyyy{\tilde{\epsilon}}$ almost surely for any $\omega \in \Omega$ and all $\x \in \Xscr$. Then, the following hold {for any $\x \in \Xscr$.}

\noindent {\bf (a)}  $\mathbb{E}[\|g_{\eta,\fyyy{\tilde{\epsilon}}}(\x,v,\omega)\|^2\mid\x]  \leq  3n^2\left(\tfrac{2{\tilde{L}_0^2} \fyyy{\tilde{\epsilon}}   }{\eta^2} + {L_0^2}\right),$ {almost surely}. 

 \noindent {\bf (b)}  $ \mathbb{E}\left[ \left\| g_{\eta,\fyyy{\tilde{\epsilon}}}(\x,v,\omega) - g_{\eta}(\x,v,\omega) \right\|^2 \mid \x\right] 
 \leq   \frac{{4\tilde{L}^2_0n^2}\fyyy{\tilde{\epsilon}}}{\eta^2} $,{ almost surely}.

 
\end{lemma}
\begin{proof}  
\noindent {\bf (a)} In a similar fashion to the proof of Lemma \ref{lem:inexact_error_bounds} (a), we can show that 
\begin{align*}
 \|  g_{\eta,\fyyy{\tilde{\epsilon}}}(\x,{v},\omega)\| \leq    \frac{{\tilde{L}_0(\omega)}\|\y_{\fyyy{\tilde{\epsilon}}}(\x+v,\omega) - \y(\x+v,\omega)\| n }{\eta}   + \left\|{g_{\eta}(\x,v,\omega)}  \right\|  + \frac{\tilde{L}_0(\omega)\|\y_{\fyyy{\tilde{\epsilon}}}(\x,\omega)-\y(\x,\omega)\| n}{\eta} . 
\end{align*} 
{Invoking Lemma \ref{lem:smooth_grad_properties}, we} may then bound the second moment of $\|g_{\eta,\fyyy{\tilde{\epsilon}}}(\x,v,\omega)\|$ as follows.
\begin{align*}
    \mathbb{E}[\|{g_{\eta,\fyyy{\tilde{\epsilon}}}(\x,{v},\omega)}\|^2] & \leq   3\mathbb{E}\left[\left(\frac{{\tilde{L}_0^2(\omega){n^2}}\|\y_{\fyyy{\tilde{\epsilon}}}(\x+v,\omega) - \y(\x+v,\omega)\|^2  }{\eta^2}\right) \mid \x\right]  + 3\mathbb{E}\left[\left\|g_{\eta}(\x,v,\omega)\right\|^2\mid \x\right] \notag\\
                                                     & + 3\mathbb{E}\left[\left(\frac{\tilde{L}_0^2(\omega){n^2}\|\y_{\fyyy{\tilde{\epsilon}}}(\x+v,\omega) - \y(\x+v,\omega)\|^2   }{\eta^2}\right) \mid \x\right] \notag\\
                                                     &\leq  3\mathbb{E}\left[\left(\frac{{\tilde{L}_0^2(\omega){n^2}}\fyyy{\tilde{\epsilon}}^2  }{\eta^2}\right) \mid \x\right]  + 3L_0^2n^2  + 3\mathbb{E}\left[\left(\frac{\tilde{L}_0^2(\omega){n^2}\fyyy{\tilde{\epsilon}}^2 }{\eta^2}\right) \mid \x\right] \leq     3n^2\left(\tfrac{2{\tilde{L}_0^2} \fyyy{\tilde{\epsilon}}   }{\eta^2} + {L_0^2}\right).  \notag\\
\end{align*}

{\noindent {\bf (b)}} 
 In a similar fashion to the proof of Lemma \ref{lem:inexact_error_bounds} (b), we can show that 
\begin{align*}
 & \quad  \left\| { g_{\eta,\fyyy{\tilde{\epsilon}}}(\x,{v},\omega) - g_{\eta}(\x,{v},\omega) } \right\|
 \leq   \frac{{\tilde{L}_0(\omega){n}}\| \y_{\fyyy{\tilde{\epsilon}}}(\x+v,\omega) - \y(\x+v,\omega)\|}{\eta}   +   \frac{{\tilde{L}_0(\omega){n}} \| \y_{\fyyy{\tilde{\epsilon}}}(\x,\omega) - \y(\x,\omega)\|}{\eta}.  
  \end{align*}
It follows that \begin{align*}
\mathbb{E}\left[  \left\| { g_{\eta,\fyyy{\tilde{\epsilon}}}(\x,{v},\omega) - g_{\eta}(\x,{v},\omega) } \right\|^2 \mid \x\right] 
 & \leq  \frac{2 \mathbb{E}[ {\tilde{L}_0^2(\omega)n^2}\| \y_{\fyyy{\tilde{\epsilon}}}(\x+v,\omega) - \y(\x+v,\omega)\|^2\mid \x]}{\eta^2}  \\
  &  +   \frac{2\mathbb{E}[{\tilde{L}_0^2(\omega)n^2} \| \y_{\fyyy{\tilde{\epsilon}}}(\x,\omega) - \y(\x,\omega)\|^2\mid \x]}{\eta^2}\\
 & \leq \frac{2 \mathbb{E}[ {\tilde{L}_0^2(\omega)n^2}\fyyy{\tilde{\epsilon}}^2\mid \x]}{\eta^2}    +   \frac{2\mathbb{E}[{\tilde{L}_0^2(\omega)n^2} \fyyy{\tilde{\epsilon}}^2\mid \x]}{\eta^2}  \leq  \frac{{4\tilde{L}^2_0n^2}\fyyy{\tilde{\epsilon}}}{\eta^2}.
   \end{align*}
\end{proof}
}

{ \begin{algorithm}[H]
        \caption{\texttt{ZSOL$^{\bf 2s}_{\rm cnvx}$}: Zeroth-order method for convex  \fyy{\eqref{eqn:a_s_prob}}}\label{algorithm:inexact_zeroth_order_SVIs_2S}
    \fyy{\begin{algorithmic}[1]
        \STATE\textbf{input:} {Given $\x_0 \in \Xscr$, ${\bar \x}_0: = \x_0$,  stepsize sequence $\{\gamma_k\}$,  smoothing parameter sequence  $\{\eta_k\}$}, {inexactness sequence $\{\fyyy{\tilde{\epsilon}_k}\}$}, $r \in [0,1)$,  and $S_0 : = \gamma_0^r$
    \FOR {$k=0,1,\ldots,{K}-1$}
        \STATE Generate $v_{k} \in \eta_k \mathbb{S}$
    		    \STATE  Do one of the following, depending on the type of the scheme.
    		    
\vspace{-.1in}    	
	    
\begin{itemize}
\item   Inexact scheme: Call Alg. \ref{algorithm:Inexact_lower_level_SA_twostage} twice to obtain $\y_\fyyy{\tilde{\epsilon}_k}(\x_k,\omega_k)$ and $\y_\fyyy{\tilde{\epsilon}_k}(\x_k+v_k,\omega_k)$ 
\vspace{-.15in}    	

\item   Exact scheme: Evaluate $\y(\x_k,\omega_k)$ and $\y(\x_k+v_k,\omega_k)$
\end{itemize}

\vspace{-.1in}    	

    		 \STATE Evaluate the inexact or exact zeroth-order gradient approximation as follows.
 \begin{align}
             & g_{\eta_k,\fyyy{\tilde{\epsilon}_k}}(\x_k,v_k,\omega_k):= \tfrac{n\left(\fyy{\tilde f}(\x_k+ v_k, \y_\fyyy{\tilde{\epsilon}_k}(\x_k+ v_k,\omega_k),\omega_k) - \fyy{\tilde f} (\x_k, \y_\fyyy{\tilde{\epsilon}_k}(\x_k,\omega_k),\omega_k)\right)v_k}{\|v_k\|\eta_k} \tag{Inexact} \\
             & g_{\eta_k}(\x_k,v_k,\omega_k) := \tfrac{n\left(\fyy{\tilde f}(\x_k+ v_k, \y(\x_k+ v_k,\omega_k),\omega_k) - \fyy{\tilde f} (\x_k, \y(\x_k,\omega_k),\omega_k)\right)v_k}{\|v_k\|\eta_k}\tag{Exact}
\end{align}  	

\STATE {Update $\x_k$ as follows.}
        \begin{align*}
            \x_{k+1} :=  \begin{cases}
                 \Pi_{\Xscr} \left[ \x_k - \gamma_k g_{\eta_k,\fyyy{\tilde{\epsilon}_k}}(\x_k,v_k,\omega_k) \right] & \hspace{2.75in} \mbox{ (Inxact)}  \\
                   \Pi_{\Xscr} \left[ \x_k - \gamma_k g_{\eta_k}(\x_k,v_k,\omega_k) \right] & \hspace{2.8in}\mbox{ (Exact)} 
            \end{cases}
        \end{align*}

  \STATE Update the averaged iterate as follows. $ S_{k+1} : = S_k+\gamma_{k+1}^r$ and  $\bar \x_{k+1}:=\tfrac{S_k\bar \x_k+\gamma_{k+1}^r\x_{k+1}}{S_{k+1}}$
    \ENDFOR
   \end{algorithmic}}
\end{algorithm}

\begin{algorithm}[H]
  \caption{\fyy{\usv{Projection} method for the VI in the lower-level of \fyy{\eqref{eqn:a_s_prob}}}}\label{algorithm:Inexact_lower_level_SA_twostage}
    \begin{algorithmic}[1]
        \fyy{\STATE \textbf{input:} An arbitrary $\y_0\in \Yscr$,  vectors $\hat{\x}_k$ and $\omega$, scalar $\rho \in (0,1)$, {stepsize $\alpha>0$, integer $k$, and scalar $\tau>0$} 
        \STATE Compute $t_k:=\lceil\tau \ln(k+1) \rceil $ 
      \FOR {$t=0,1,\ldots,t_k-1$}
      \STATE Evaluate the mapping $G(\hat{\x}_k,{\y_t}, \omega)$
      \STATE {Update $\y_t$ as follows.} $ \y_{t+1} := \Pi_{\Yscr}\left[\y_t- \alpha G(\hat{\x}_k,\y_t,\omega)\right]$
        \ENDFOR
        \STATE Return $\y_{t_k}$ }
   \end{algorithmic}
\end{algorithm}}

{Next we develop rate and complexity statements for Algorithm~\ref{algorithm:inexact_zeroth_order_SVIs_2S}. The algorithm parameters for both inexact and exact schemes are defined next. 

\begin{definition}[Parameters for Algorithms~\ref{algorithm:inexact_zeroth_order_SVIs_2S}--\ref{algorithm:Inexact_lower_level_SA_twostage}]\label{def:algo_6}\em Let the stepsize and smoothing sequence in Algorithm~\ref{algorithm:inexact_zeroth_order_SVIs_2S} be given by {$\gamma_k := \frac{\gamma_0}{(k+1)^a}$ and $\eta_k := \frac{\eta_0}{(k+1)^b}$}, respectively for all $k\geq 0$ where {$\gamma_0, \eta_0, a,$ and $b$} are strictly positive.  In Algorithm~\ref{algorithm:Inexact_lower_level_SA_twostage}, suppose  $\alpha \leq \tfrac{\mu_F}{L_F^2}$.  Let $t_k:=\lceil\tau \ln(k+1) \rceil $ where $\tau \geq \frac{-{2(a+b)}}{\ln(1-\mu_F\alpha)}$. Finally, suppose $r \in [0,1)$ is an arbitrary scalar.  \end{definition}
}

{\begin{theorem}[\bf Rate and complexity statements and \uss{a.s.} convergence for inexact \fyy{(\texttt{ZSOL$^{\bf 2s}_{\rm cnvx}$})}]\label{thm:ZSOL_convex_2s}\em
Consider {the} sequence {$\{\bar \x_k\}$} generated by applying Algorithm~\ref{algorithm:inexact_zeroth_order_SVIs_2S} on \eqref{prob:mpec_as_imp}. Suppose Assumptions~\ref{ass-1}--~\ref{assum:u_iter_smooth} hold and algorithm parameters are defined by Def.~\ref{def:algo_6}.  

        \noindent {\bf(a)} {Suppose \fyy{$\hat \x_k \in \Xscr+\eta_k\mathbb{S}$} and let $\{\y_{t_k}\}$ be the sequence generated by Algorithm \ref{algorithm:Inexact_lower_level_SA_twostage}. Then for suitably defined $\tilde{d} < 1$ and $B > 0$, the following holds for $t_k \geq 1$. }  
\begin{align*}
\|\y_{t_k} - \y(\hat \x_k,\omega_k)\|^2 \leq \fyyy{\tilde{\epsilon}_k} \triangleq  B \tilde{d}^{t_k}.
\end{align*}
\noindent {\bf{(b)}}  Let $a=0.5$ and $b \in [0.5,1)$ and $0\leq r< 2(1-b)$. Then, for all $K \geq 2^{\frac{1}{1-r}}-1$ we have
 
\begin{align*} \mathbb{E}\left[f^{\bf imp}(\bar \x_K)\right]-f^*  \leq (2-r)\left(\tfrac{D_\Xscr }{\gamma_0}+\tfrac{2\theta_0\gamma_0}{1-r}\right)\tfrac{1}{\sqrt{K+1}}+(2-r)\left(\tfrac{\eta_0L_0}{1-0.5r-b}\right)\tfrac{1}{(K+1)^{b}},\end{align*}

where $\theta_0 \triangleq D_\Xscr  +\tfrac{\left(2+3\gamma_0^2\right)n^2\tilde{L}_0^2B}{\eta_0^2\gamma_0^2}+ 1.5n^2L_0^2$.
In particular, when $b:=1-\delta$ and $r=0$, where $\delta>0$ is a small scalar,  we have for all $K\geq 1$
\begin{align*} 
\mathbb{E}\left[f^{\bf imp}(\bar \x_K)\right]-f^* &\leq 2\left(\tfrac{D_\Xscr }{\gamma_0}+2\theta_0\gamma_0\right)\tfrac{1}{\sqrt{K+1}}+\left(\tfrac{2\eta_0L_0}{\delta}\right)\tfrac{1}{(K+1)^{1-\delta}}.
\end{align*}

\noindent {\bf{(c)}}  Suppose $\gamma_0\fyy{:= \mathcal{O}(\tfrac{1}{L_0})}$, \fyy{$a:=0.5$, $b:=0.5$, and $r:=0$}. Let $\epsilon>0$ be an arbitrary scalar and $K_{\epsilon}$ be such that $\mathbb{E}\left[\fyy{f^{\bf imp}(\bar \x_{K_\epsilon})}\right]-f^*  \leq \epsilon$. Then,

\noindent {{(c-1)}} The total number of upper-level projection steps on $\Xscr$ is $K_{\epsilon}=\fyy{\mathcal{O}\left(n^4L_0^2\fyy{\tilde{L}_0^4}\epsilon^{-2}  \right)}$.

\noindent {{(c-2)}} \fyy{The total sample complexity of upper-level} is $\fyy{\mathcal{O}\left(n^4L_0^2\tilde{L}_0^4\epsilon^{-2}  \right)}$.
 
\noindent {{(c-3)}} The total number of lower-level projection steps on $\Yscr$ is $ 
\fyy{\mathcal{O}\left(n^4L_0^2\tilde{L}_0^4\epsilon^{-2}  \ln\left( n^2L_0\tilde{L}_0^2\epsilon^{-1}  \right)\right)}.$


\noindent {\bf{(d)}} For any $a \in (0.5,1]$ and $b>1-a$, there exists $\x^* \in \Xscr^*$ such that $\lim_{k\to \infty} \|\bar \x_k-\x^*\|^2 = 0$ almost surely.
\end{theorem}}
\begin{proof}{
\noindent {\bf (a)} 
From $ \y(\hat\x_k,\omega_k) \in \mbox{SOL}(\Yscr, G(\hat \x_k,\bullet,\omega_k))$, we have {that the following fixed-point relationship holds.}  $$\y(\hat\x_k,\omega_k)=\Pi_{\Yscr}\left[\y(\hat\x_k,\omega_k) -\alpha G(\hat \x_k,\y(\hat\x_k,\omega_k),\omega_k)\right],$$ 
{for any $\alpha > 0$}. Thus, we can write
\begin{align*}
&\quad  \|\y_{t+1}- \y(\hat \x_k,\omega_k)\|^2 = \|\Pi_{\Yscr}\left[\y_t- \alpha G(\hat{\x}_k,\y_t,\omega_{k})\right] -\Pi_{\Yscr}\left[\y(\hat\x_k,\omega_k) -\alpha G(\hat \x_k,\y(\hat\x_k,\omega_k),\omega_k)\right] \|^2\\
& \leq  \|\y_t- \alpha G(\hat{\x}_k,\y_t,\omega_{k}) -\y(\hat\x_k,\omega_k) +\alpha G(\hat \x_k,\y(\hat\x_k,\omega_k),\omega_k) \|^2\\
& = \|\y_{t}- \y(\hat \x_k,\omega_k)\|^2 +\|\alpha G(\hat{\x}_k,\y_t,\omega_{k}) -\alpha G(\hat \x_k,\y(\hat\x_k,\omega_k),\omega_k) \|^2 \\
&-2\alpha(\y_t -\y(\hat\x_k,\omega_k))^T( G(\hat{\x}_k,\y_t,\omega_{k}) - G(\hat \x_k,\y(\hat\x_k,\omega_k),\omega_k)).
\end{align*}
Invoking Assumption~\ref{ass-1} (b) we obtain 
\begin{align*}
\quad \|\y_{t+1}- \y(\hat \x_k,\omega_k)\|^2 &\leq  \|\y_{t}- \y(\hat \x_k,\omega_k)\|^2 +\alpha L_F(\omega)\|\y_{t}- \y(\hat \x_k,\omega_k)\|^2 -2\alpha\mu_F(\omega)\|\y_{t}- \y(\hat \x_k,\omega_k)\|^2\\
& \leq (1+\alpha^2L_F^2-2\alpha \mu_F) \|\y_{t}- \y(\hat \x_k,\omega_k)\|^2.
\end{align*}
This implies that 
$ \|\y_{t_k}- \y(\hat \x_k,\omega_k)\|^2 \leq (1+\alpha^2L_F^2-2\alpha \mu_F)^{t_k}(\sup_{\y \in \Yscr}\|\y -\y_0\|^2).$ Note that $\alpha \leq \tfrac{\mu_F}{L_F^2}$ implies that $1+\alpha^2L_F^2-2\alpha \mu_F\leq 1-\alpha\mu_F$. Defining $\tilde{d} \triangleq 1-\alpha \mu_F$ and $B \triangleq \sup_{\y \in \Yscr}\|\y -\y_0\|^2$, we obtain the bound.}

{\noindent {\bf (b, d)} Recall the properties of the exact and inexact zeroth-order gradient mappings in the two-stage model provided in Lemmas \ref{lem:smooth_grad_properties_two_stage} and \ref{lem:inexact_error_bounds_2s}, respectively. Note that these results are identical to those of the single-stage model provided in Lemmas \ref{lem:smooth_grad_properties} and \ref{lem:inexact_error_bounds}, respectively. For this reason, the proof of the remaining parts can be done in a very similar fashion to the proofs in Theorem \ref{thm:ZSOL_convex}. As such, the proofs for (b) and (d) are omitted.}

{\noindent {\bf (c)} Note that (c-1) and (c-2) follow directly from part (b) by substituting $\gamma_0$ and $r$. To show (c-3), note that the total projection steps in the lower-level is as follows.
\begin{align*}
2\sum_{k=0}^{K_\epsilon}\sum_{t=0}^{ t_k} 1 &= 2(K_\epsilon+1)(t_{K_\epsilon}+1) = 2(K_\epsilon+1)(\lceil \tau\ln(K_\epsilon+1)\rceil+1)= \mathcal{O}\left(n^4L_0^2\tilde{L}_0^4\epsilon^{-2}  \ln\left( n^2L_0\tilde{L}_0^2\epsilon^{-1}  \right)\right).
\end{align*}
}
\end{proof}
\fyyy{\begin{remark}
The convergence rate in expectation in Theorem~\ref{thm:ZSOL_convex}~(b) and Theorem~\ref{thm:ZSOL_convex_2s}~(b) can be extended to the case that $a \in [0.5,1)$. However, the rate of convergence would be worse when $a \in (0.5,1)$ compared to when $a=0.5$. This is because employing Lemma~\ref{lem:harmonic_bounds}, the rate of convergence is characterized as $\mathcal{O}\left(\tfrac{1}{k^{1-a}} +\tfrac{1}{k^{a}} +\tfrac{1}{k^b}\right)$. For this reason we only present the rate analysis in those theorems for $a=0.5$.
\end{remark}}
\noindent {\bf An exact zeroth-order scheme.}
{Next, we address} the two-stage model~\eqref{prob:mpec_as_imp} {where} we consider the {case} where an exact solution of the lower-level problem is available. In the following, we extend the convergence properties of the ZSOL scheme to the exact case.
\begin{corollary}[\bf Rate and complexity statements and almost sure convergence for exact \fyy{(\texttt{ZSOL$^{\bf 2s}_{\rm cnvx}$})}]\label{thm:convex_exact_2s}\em
        Consider the problem {\eqref{prob:mpec_exp_imp}}. Suppose {Assumptions~\ref{ass-1}--~\ref{assum:u_iter_smooth}} hold. Suppose {$\{\bar \x_k\}$} denotes the sequence generated by {Algorithm \ref{algorithm:inexact_zeroth_order_SVIs_2S} (exact variant)} in which the stepsize and smoothing sequences are defined as $\gamma_k := \frac{\gamma_0}{(k+1)^a}$ and $\eta_k := \frac{\eta_0}{(k+1)^b}$, respectively, for all $k\geq 0$ where $\gamma_0$ and $\eta_0$ are strictly positive. Then, the following statements hold.

\noindent {\bf{(a)}} Let $a=0.5$ and $b \in [0.5,1)$ and $0\leq r< 2(1-b)$. Then, for all $K \geq 2^{\frac{1}{1-r}}-1$ we have
\begin{align*} 
\mathbb{E}\left[\fyy{f^{\bf imp}(\bar \x_K)}\right]-f^* &\leq (2-r)\left(\tfrac{D_\Xscr }{\gamma_0}+\tfrac{L_0^2n^2\gamma_0}{1-r}\right)\tfrac{1}{\sqrt{K+1}}+(2-r)\left(\tfrac{\eta_0L_0}{1-0.5r-b}\right)\tfrac{1}{(K+1)^{b}}.
\end{align*}
In particular, when $b:=1-\delta$ and $r=0$, where $\delta>0$ is a small scalar,  we have for all $K\geq 1$
\begin{align*} 
\mathbb{E}\left[\fyy{f^{\bf imp}(\bar \x_K)}\right]-f^* &\leq 2\left(\tfrac{D_\Xscr }{\gamma_0}+L_0^2n^2\gamma_0\right)\tfrac{1}{\sqrt{K+1}}+\left(\tfrac{2\eta_0L_0}{\delta}\right)\tfrac{1}{(K+1)^{1-\delta}}.
\end{align*}

\noindent {\bf{(b)}} Let $a:=0.5$, $b=0.5$, $r=0$, $\gamma_0:=\tfrac{\sqrt{D_\Xscr}}{nL_0}$, and $\eta_0\leq \sqrt{D_\Xscr}n$. Then, the iteration complexity {in projection steps on $\Xscr$} for achieving {$\mathbb{E}\left[\fyy{f^{\bf imp}(\bar \x_{K_\epsilon})}\right]-f^* \leq \epsilon$} for some $\epsilon>0$ is {bounded} as follows.
\begin{align*}
{K_\epsilon} \geq \frac{64n^2L_0^2D_\Xscr}{\epsilon^2} .
\end{align*}

\noindent {\bf{(c)}} For any $a \in (0.5,1]$ and $b>1-a$, there exists $\x^* \in \Xscr^*$ such that $\lim_{k\to \infty} \|\bar \x_k-\x^*\|^2 = 0$ almost surely.
\end{corollary}
\begin{proof}
{In view of the similarity between the results of Lemmas \ref{lem:smooth_grad_properties_two_stage} and \ref{lem:inexact_error_bounds_2s} with those of Lemmas \ref{lem:smooth_grad_properties} and \ref{lem:inexact_error_bounds}, the proof can be done in a similar fashion to that of Corollary \ref{thm:convex_exact}.} 
\end{proof}

\subsection{Exact accelerated schemes for convex regime}\label{sec:4.3}
In this subsection, we consider an accelerated scheme for resolving the problem \fyy{\eqref{eqn:a_s_prob}}, whose implicit form is defined as \fyy{\eqref{prob:mpec_as_imp}} where $\y(\x,\omega)$ is the unique solution of an $\omega$-specific strongly monotone variational inequality problem parametrized by $\x$. 
The deterministic
counterpart of this problem is the standard MPEC in which the lower-level
problem is a parametrized strongly monotone variational inequality problem.  While the previous subsection has considered a standard gradient-based framework, we consider an accelerated counterpart motivated by Nesterov's celebrated accelerated gradient method~\cite{nesterov83} that produces a
non-asymptotic rate of $\mathcal{O}(1/k^2)$ in terms of suboptimality for
smooth convex optimization problems. In~\cite{nesterov17}, Nesterov and
Spokoiny develop an accelerated zeroth-order scheme for the unconstrained minimization of a
smooth function.  Instead, we present an 
accelerated gradient-free scheme for a nonsmooth function by leveraging the smoothing architecture. Notably, this scheme can contend with MPECs with convex
implicit functions. In this subsection, we assume that $\y(\x,\omega)$ can be generated by invoking a suitable variational inequality problem solver.    
\begin{algorithm}[H]
    \caption{\fyy{\texttt{ZSOL$^{\bf 2s}_{\rm cnvx, acc}$}: Variance-reduced accelerated exact zeroth-order method for convex {\eqref{eqn:a_s_prob}}}}\label{algorithm:acc-inexact_zeroth_order_SVIs}
    \label{alg:acc_zsol}
    \begin{algorithmic}[1]
        \STATE\textbf{input:} {Given $\x_0 \in \Xscr$, $\lambda_0 = 1$,  stepsize sequence $\{\gamma_k\}$,  smoothing parameter sequence  $\{\eta_k\}$},   sample-size $\{N_k\}$
    \FOR {$k=0,1,\ldots,{K}-1$}
       \FOR {$j=1,\ldots,N_k$}
\STATE \fyy{Generate $v_{\fy{j,k}} \in \eta_k \mathbb{S}$ 
\STATE Evaluate $\y(\x_k+v_{{j,k}},\omega_{{j,k}})$}

    \STATE  \fyy{Evaluate the exact zeroth-order gradient approximation as follows.
$$   g_{\eta_k}(\x_k,v_{j,k},\omega_{j,k}) :=\tfrac{n\left(\fyy{\tilde f}(\x_k+ v_{j,k}, \y(\x_k+ v_{j,k},\omega_{j,k}),\omega_{j,k}) - \fyy{\tilde f} (\x_k, \y(\x_k,\omega_{j,k}),\omega_{j,k})\right)v_{j,k}}{\|v_{j,k}\|\eta_k}$$}
\ENDFOR
\STATE \fyy{Evaluate the mini-batch exact zeroth-order gradient as $g_{\eta_k,N_k}(\x_k) = \frac{\sum_{j=1}^{N_k} g_{\eta_k}(\x_k,v_{j,k},\omega_{j,k})}{N_k}.$
}

\STATE {Update $\x_k$ as follows.}
{\begin{align}
\begin{aligned}
    \vz_{k+1} & := \Pi_{\Xscr} \left[ \x_k - \gamma_k g_{\eta_k,N_k}(\x_k,v_k) \right] \\
    \lambda_{k+1} & := \tfrac{1+\sqrt{1+4\lambda_k^2}}{2} \\
    \x_{k+1} & = \vz_{k+1} + \tfrac{(\lambda_k-1)}{\lambda_{k+1}}\left(\vz_{k+1}-\vz_k\right). 
\end{aligned}
\end{align}}
    \ENDFOR
   \end{algorithmic}
\end{algorithm}
We provide convergence theory for {Algorithm \ref{alg:acc_zsol}} by appealing to related work on smoothed accelerated schemes for nonsmooth stochastic convex optimization~\cite{jalilzadeh2018smoothed}. There are two key differences between the framework presented here and that of our prior work.

\medskip

\noindent (a) {\em Smoothing.} In~\cite{jalilzadeh2018smoothed}, we employ a
deterministic smoothing technique~\cite{beck17fom} while in this paper, we
consider a locally randomized smoothing technique in a zeroth-order regime.
Notably, the latter leads to similar (but not identical) smoothness properties
\fy{with} related relationships (but not identical) between the smoothed function
and its original counterpart.   

\smallskip

\noindent (b) {\em Zeroth-order gradient approximation.}
In~\cite{jalilzadeh2018smoothed}, a sampled gradient of the smoothed function is
available. However, faced by the need to resolve hierarchical problems, we do
not have such access in this paper. Instead, we utilize an increasingly accurate zeroth-order
approximation of the gradient by raising the sample-size $N_k$ in constructing this approximation. {We make the following assumption on the generated random samples in  the proposed accelerated scheme in the upper-level.
 \begin{assumption}\label{assum:upper_level_acc}\em
Given a mini-batch sequence $\{N_k\}$ and a smoothing sequence $\{\eta_k\}$, let $v_{j,k} \in \mathbb{R}^n$, for $j=1,\ldots, N_k$ and $k \geq 0$ be generated randomly and independently, from $\eta_k\mathbb{S}$ for all $k\geq 0$. Also, let the random realizations $\{\omega_{j,k}\}$ be \fyy{iid replicates}.
 \end{assumption}
 We may define $\bar{w}_{k,N_k}$ as $\bar{w}_{k,N_k} \triangleq  g_{\eta_k,N_k}(\x_k) - \nabla_{\x} {f^{\bf imp}_{\eta_k}(\x_k)}$. The following claims can be made.
 \begin{lemma} Consider $\bar{w}_{k,N_k}$ obtained by generating $N_K$ independent realizations given by $\{v_{j,k}\}_{j=1}^{N_k}$ and  {$\{\omega_{j,k}\}_{j=1}^{N_k}$}. Let Assumption~\ref{assum:upper_level_acc} hold. {Then the following hold almost surely {for any $\x_k \in \Xscr.$}}

\noindent (a)   $\mathbb{E}[\bar{w}_{k,N_k} \mid \x_k] = 0$. 

\noindent (b)  $\mathbb{E}[\|\bar{w}_{k,N_k}\|^2 \mid \x_k]   \leq \tfrac{n^2L_0^2}{N_k} $.
\end{lemma}
\begin{proof}
Note that (a) holds in view of Lemma \ref{lem:smooth_grad_properties_two_stage}. 
Invoking Lemma \ref{lem:smooth_grad_properties_two_stage} can write 
 \begin{align*}
 &\quad \mathbb{E}\left[\| \bar{w}_{k,N_k}\|^2\mid \fyy{\x_k}\right]  =\mathbb{E}\left[\left\|g_{\eta_k,N_k}(\x_k)  - \nabla_{\x} {f^{\bf imp}_{\eta_k}(\x_k)} \right\|^2 \mid\fyy{\x_k}\right]\notag\\
&=  \mathbb{E}\left[\left\|\tfrac{\sum_{j=1}^{N_k} g_{\eta_k}(\x_k,v_{j,k},\omega_{j,k})}{N_k} - \nabla_{\x} f^{\bf imp}_{\eta_k}(\x_k) \right\|^2\mid \fyy{\x_k}\right] \leq  \tfrac{\sum_{j=1}^{N_k} \mathbb{E}\left[\left\|g_{\eta_k}(\x_k,v_{j,k},\omega_{j,k})-\nabla_{\x} f^{\bf imp}_{\eta_k}(\x_k) \right\|^2\mid \fyy{\x_k}\right]}{N_k^2}\notag\\
&\leq  \tfrac{\sum_{j=1}^{N_k}\left( \mathbb{E}\left[\left\|g_{\eta_k}(\x_k,v_{j,k},\omega_{j,k})\right\|^2\mid \fyy{\x_k}\right]-\left\|\nabla_{\x} f^{\bf imp}_{\eta_k}(\x_k) \right\|^2\right)}{N_k^2}\leq  \tfrac{n^2L_0^2}{N_k}.
\end{align*}  
\end{proof}}
\begin{lemma}\em~\cite[Lemma 4]{jalilzadeh2018smoothed} \label{lem:acc-lemma} Consider the problem {\eqref{prob:mpec_as_imp}}. Suppose {Assumptions~\ref{ass-1}--~\ref{assum:u_iter_smooth}, ~\ref{assum:upper_level_acc}} hold. Suppose {$\{\x_k,\vz_k\}$} denote the sequence generated by Algorithm \ref{alg:acc_zsol} in which the stepsize and smoothing sequences are  defined as {$\eta_k = \tfrac{1}{k+1}$ and $\gamma_k = \tfrac{1}{2(k+1)}$, and $N_k = \lfloor (k+1)^a\rfloor$ for $k\geq 0$. Suppose $\|\vx_0-\vx^*\|\leq C$ for some $C>0$. Then the following holds.}
    \begin{align}
        \mathbb{E}\left[{f^{\bf imp}_{\eta_K}}(\vz_K) - {f^{\bf imp}_{\eta_K}}(\vx^*)\right] \leq \frac{2}{\gamma_{K-1} (K-1)^2} \sum_{k=1}^{K-1} \tfrac{\gamma_k^2k^2 {n^2L_0^2}}{N_{k-1}} + \frac{2C^2}{\gamma_{K-1} (K-1)^2}. 
        \end{align}
\end{lemma}

We may now provide the main rate statement for the smoothed accelerated scheme by adapting~\cite[Thm.~5]{jalilzadeh2018smoothed}. 

\begin{proposition}[{\bf Rate statement for Algorithm~\ref{alg:acc_zsol}}]\label{prop:acc_convex_exact}\em
Consider the problem {\eqref{prob:mpec_as_imp}}. Suppose {Assumptions~\ref{ass-1}--~\ref{assum:u_iter_smooth}, ~\ref{assum:upper_level_acc}} hold. Suppose {$\{\x_k,\vz_k\}$} denote the sequence generated by Algorithm~\ref{alg:acc_zsol} in which the stepsize and smoothing sequences are  defined as {$\eta_k = \tfrac{1}{k+1}$ and $\gamma_k = \tfrac{1}{2(k+1)}$, and $N_k = \lfloor (k+1)^a\rfloor$ for $k\geq 0$. Suppose $\|\vx_0-\vx^*\|\leq C$ for some $C>0$}. Then the following hold for $a = 1+\delta$ \fyy{where $\delta>0$}. {Suppose $\fyyy{K_\epsilon}$ is such that $\mathbb{E}[f^{\bf imp}(\vz_{K_\epsilon})] - f^*\leq \epsilon$. Then the following holds.}

\noindent {\bf(a)} {The iteration complexity in {terms of} zeroth-order gradient steps is $\mathcal{O}(1/\epsilon)$.} 

\noindent {\bf(b)}  {We have} $\sum_{k=1}^{K_{\epsilon}} N_k \leq \mathcal{O}(1/\epsilon^{2+\delta})$ implying that \fyy{the sample complexity as well as} the iteration complexity in terms of lower-level calls to the VI solver \fyy{are both} $\mathcal{O}(1/\epsilon^{2+\delta})$.
\end{proposition}
\begin{proof}
    \noindent {\bf(a)} From Lemma~\ref{lem:acc-lemma}, we have that  
    \begin{align}
        \mathbb{E}\left[{f^{\bf imp}_{\eta_K}}(\vz_K) -{f^{\bf imp}_{\eta_K}}(\vx^*)\right] \leq \frac{2}{\gamma_{K-1} (K-1)^2} \sum_{k=1}^{K-1} \tfrac{\gamma_k^2k^2 {n^2L_0^2}}{N_{k-1}} + \frac{2C^2}{\gamma_{k-1} (K-1)^2}. 
    \end{align}
    From Lemma~\ref{lemma:props_local_smoothing} (v), we have that {$f^{\bf imp}(\x) \leq f^{\bf imp}_{\eta_K}(\x) \leq f^{\bf imp}(\x)+\eta_K L_0$}. Consequently, we have
\begin{align*}
    \mathbb{E}\left[{f^{\bf imp}}(\vz_K) - {f^*}\right] & \leq \mathbb{E}\left[{f^{\bf imp}_{\eta_K}}(\vz_K) -{f^{\bf imp}_{\eta_K}}(\vx^*)\right] + \eta_K L_0 \\
                    & \leq \frac{2}{\gamma_{K-1} (K-1)^2} \sum_{k=1}^{K-1} \tfrac{\gamma_k^2k^2 {n^2L_0^2}}{N_{k-1}} + \frac{2C^2}{\fy{\gamma_{K-1}} (K-1)^2}  +\fy{\eta_K} L_0 \leq  \mathcal{O}(\tfrac{1}{K}), 
    \end{align*}
{where we used $\eta_k = \tfrac{1}{k+1}$ and $\gamma_k = \tfrac{1}{2(k+1)}$, and $N_k = \lfloor (k+1)^a\rfloor$} where $a = 1+\delta$.  

    \noindent {\bf(b)} {The proof can be done in a similar vein to that of~\cite[Thm.~5]{jalilzadeh2018smoothed} and thus, it is omitted.}

\end{proof}
\begin{remark} {Several points deserve emphasis. (i) The proposed scheme employs diminishing smoothing sequences rather than fixed, leading to asymptotic convergence guarantees, a key distinction from the scheme proposed in~\cite{nesterov17}. (ii) By adapting the framework employed for the inexact oracles, one may consider similar extensions to the accelerated framework. However, this would lead to bias in the gradient approximation and one would expect this to adversely affect the rate. This remains a goal of future study.}
\end{remark}

\subsection{Nonconvex two-stage SMPEC}\label{sec:4.4}
{In this {subsection}, we address the two-stage model~\eqref{prob:mpec_as_imp} \uss{when} the implicit function is nonconvex. The outline of the proposed zeroth-order scheme is given by Algorithm \ref{algorithm:ZSOL_nonconvex_2s} in both inexact and exact variants. In the following we present the results for each of the two variants. 

\subsubsection{An inexact zeroth-order scheme}
In the following, we present the rate and complexity result for the proposed inexact method for addressing the two-stage model in the nonconvex case.
\begin{algorithm}[H]
    \caption{\fyy{\texttt{ZSOL$^{\bf 2s}_{\rm ncnvx}$}: Variance-reduced zeroth-order method for nonconvex \eqref{eqn:a_s_prob}}}\label{algorithm:ZSOL_nonconvex_2s}
    \fyy{\begin{algorithmic}[1]
        \STATE\textbf{input:} Given $\x_0 \in \Xscr$, ${\bar \x}_0: = \x_0$,  stepsize $\gamma>0$, smoothing parameter $\eta>0$, mini-batch sequence $\{N_k\}$ such that $N_k:=k +1$, an integer $K$, {a scalar $\lambda \in (0,1)$, and an integer $R$ randomly selected from $\{\lceil\lambda K\rceil ,\ldots,K\}$ using a uniform distribution}
    \FOR {$k=0,1,\ldots,{K}-1$}

            \FOR {$j=1,\ldots,N_k$}
                    \STATE Generate \fyy{$v_{j,k} \in \eta \mathbb{S}$} 
    		    
    		      \STATE  \fyy{Do one of the following.
    		    
\vspace{-.1in}    	
	    
\begin{itemize}
\item   Inexact scheme: Call Alg. \ref{algorithm:Inexact_lower_level_SA_twostage} twice to obtain $\y_\fyyy{\tilde{\epsilon}_k}(\x_k,\omega_{j,k})$  and $\y_\fyyy{\tilde{\epsilon}_k}(\x_k+v_{j,k},\omega_{j,k})$ 
\vspace{-.15in}    	

\item   Exact scheme: Evaluate $\y(\x_k,\omega_{j,k})$ and $\y(\x_k+v_{j,k},\omega_{j,k})$
\end{itemize}}

\vspace{-.1in}    

\STATE Evaluate the inexact \fyy{or exact zeroth-order gradient approximation as follows.
 \begin{align}
              g_{\eta,\fyyy{\tilde{\epsilon}_k}}(\x_k,v_{j,k},\omega_{j,k}) &:=\tfrac{n\left(\fyy{\tilde f}(\x_k+v_{j,k}, \y_\fyyy{\tilde{\epsilon}_k}(\x_k+ v_{j,k},\omega_{j,k}),\omega_{j,k}) - \fyy{\tilde f} (\x_k, \y_\fyyy{\tilde{\epsilon}_k}(\x_k,\omega_{j,k}),\omega_{j,k})\right)v_{j,k}}{\|v_{j,k}\|\eta} \tag{Inexact} \\
              g_{\eta}(\x_k,v_{j,k},\omega_{j,k}) &:=\tfrac{n\left(\fyy{\tilde f}(\x_k+v_{j,k}, \y(\x_k+ v_{j,k},\omega_{j,k}),\omega_{j,k}) - \fyy{\tilde f} (\x_k, \y(\x_k,\omega_{j,k}),\omega_{j,k})\right)v_{j,k}}{\|v_{j,k}\|\eta}\tag{Exact}
\end{align}   }

      \ENDFOR

\STATE Evaluate the mini-batch zeroth-order gradient. \fyy{
 \begin{align}
             g_{\eta,N_k,\fyyy{\tilde{\epsilon}_k}}(\x_k) &:=\tfrac{\sum_{j=1}^{N_k} g_{\eta,\fyyy{\tilde{\epsilon}_k}}(\x_k,v_{j,k}\fyy{,\omega_{j,k}})}{N_k} \tag{Inexact} \\
             g_{\eta,N_k}(\x_k)  & :=\tfrac{\sum_{j=1}^{N_k} g_{\eta}(\x_k,v_{j,k}\fyy{,\omega_{j,k}})}{N_k} \tag{Exact}
\end{align}   	
}

\STATE {Update $\x_k$ as follows.}
         \fyy{\begin{align*}
             \x_{k+1} := \begin{cases}
                 \Pi_{\Xscr} \left[ \x_k - \gamma g_{\eta,N_k,\fyyy{\tilde{\epsilon}_k}}(\x_k) \right] & \hspace{3.1in} \mbox{(Inexact)} \\
                 \Pi_{\Xscr} \left[ \x_k - \gamma g_{\eta,N_k}(\x_k) \right] &  \hspace{3.2in} \mbox{(Exact)} 
             \end{cases}
     \end{align*}}

    \ENDFOR
        \STATE Return $\x_R$ 
   \end{algorithmic}}
\end{algorithm}

\begin{theorem}[{\bf Rate and complexity statements for \fyy{inexact (\texttt{ZSOL$^{\bf 2s}_{\rm ncnvx}$})}}]\label{thm:inexact_nonconvex_2s}\em
    Consider Algorithms~\ref{algorithm:ZSOL_nonconvex_2s}--\ref{algorithm:Inexact_lower_level_SA_twostage} for solving \eqref{prob:mpec_as_imp} and suppose Assumptions~\ref{ass-1} and~\ref{assum:u_iter_smooth_VR} hold. 

    \noindent {\bf{(a)}} Given $\hat \x_k \in \Xscr$, let $\y(\hat \x_k,\omega_{j,k})$ denote the unique solution of $\mbox{VI}(\Yscr, G(\hat \x_k,\bullet,\omega_{j,k}))$.  Let $\y_{t_k} $ be generated by Algorithm~\ref{algorithm:Inexact_lower_level_SA_twostage}. Then for suitably defined $\tilde{d} < 1$ and $B > 0$, the following holds for $t_k \geq 1$.   
\begin{align*}
\|\y_{t_k} - \y(\hat \x_k,\omega_{j,k})\|^2 \leq \fyyy{\tilde{\epsilon}_k} \triangleq  B \tilde{d}^{t_k}.
\end{align*}

\noindent {\bf{(b)}} The following holds for any $\gamma<\frac{\eta}{n L_0}$, {$\ell\triangleq  \lceil \lambda K\rceil$}, and all $K> {\tfrac{2}{1-\lambda}}$.
\begin{align*}
 \mathbb{E}\left[ \|G_{\eta,1/\gamma} (\x_R)\|^2\right] \leq \frac{n^2{\gamma}({1-2\ln(\lambda)}) {\left( 1-\tfrac{nL_0\gamma}{2\eta}\right)} \left( \tfrac{{4\tilde{L}^2_0B}}{\eta^2}+L_0^2\right) +  {\mathbb{E}\left[f^{\bf imp}(\x_{\ell})\right] }{-f^*} +2L_0\eta}{\left( 1-\tfrac{nL_0\gamma}{\eta}\right) \tfrac{\gamma}{4}{(1-\lambda)K}  }.
\end{align*}

\noindent {\bf{(c)}}  Suppose $\gamma= \tfrac{\eta}{2nL_0}$ and $\eta=\tfrac{1}{L_0}$. Let $\epsilon>0$ be an arbitrary scalar and $K_{\epsilon}$ be such that $ \mathbb{E}\left[ \|G_{\eta,1/\gamma} (\x_R)\|^2\right]   \leq \epsilon$. Then,

\noindent {{(c-1)}} The total number of upper-level projection steps on $\Xscr$ is {$K_{\epsilon}=\mathcal{O}\left(n^2L_0^2\fyy{\tilde{L}^2_0}\epsilon^{-1}\right)$}.

 \noindent {{(c-2)}} The total sample complexity of upper-level is {$\mathcal{O}\left(n^4L_0^4\fyy{\tilde{L}^4_0} \epsilon^{-2}\right)$}.
 
\noindent {{(c-3)}} The total number of lower-level projection steps on $\Yscr$ is {$ 
\mathcal{O}\left(\tau n^4L_0^4{\tilde{L}^4_0}\epsilon^{-2}\ln(n^2L_0^2{\tilde{L}^2_0}\epsilon^{-1})\right)$}.


\end{theorem}
\begin{proof}
\noindent {\bf (a)} The proof of (a) is analogous to that of Theorem \ref{thm:ZSOL_convex_2s} (a) and it is omitted. 

\noindent {\bf (b)} In view of the similarity between the results of Lemmas \ref{lem:smooth_grad_properties_two_stage} and \ref{lem:inexact_error_bounds_2s} with those of Lemmas \ref{lem:smooth_grad_properties} and \ref{lem:inexact_error_bounds}, respectively, in a similar fashion to the proof of Theorem \ref{algorithm:ZSOL_nonconvex} (b), we can obtain
\begin{align*}
    & \quad \left( 1-\tfrac{nL_0\gamma}{\eta}\right) \tfrac{\gamma}{4}{(K-\ell)}\mathbb{E}\left[ \|G_{\eta,1/\gamma} (\x_R)\|^2\right]  \\                                                                                                              & \leq{\left( 1-\tfrac{nL_0\gamma}{2\eta}\right)} {\gamma} \sum_{k={\ell}}^{K-1}\left( \tfrac{{8\tilde{L}^2_0n^2}\fyyy{\tilde{\epsilon}_k}}{\eta^2}+\tfrac{2n^2L_0^2}{N_k}\right) +  {\mathbb{E}\left[{f^{\bf imp}(\x_{\ell})}\right] }{-f^*} +2L_0\eta.
\end{align*}
 Next, we derive a bound on $\fyyy{\tilde{\epsilon}_k}$. Note that from part (a), we have $\fyyy{\tilde{\epsilon}_k}= B\tilde{d}^{t_k}$ where $t_k:=\lceil \tau \ln(k+1)\rceil \geq \tau\ln(k+1)$. We have
\begin{align*} 
(k+1) \fyyy{\tilde{\epsilon}_k} \leq  B\tilde{d}^{\tau\ln(k+1)}(k+1) =B\left(\tilde{d}^{\tau}\mathrm{e}\right)^{\ln(k+1)} \leq B,
\end{align*}
where the last inequality is implied from $\tau \geq \tfrac{-1}{\ln(\tilde{d})}$. Thus, we have that $\fyyy{\tilde{\epsilon}_k}\leq \tfrac{B}{k+1}$. {Note that $K> {\tfrac{2}{1-\lambda}}$ implies $\ell \leq K-1$.} From Lemma~\ref{lem:harmonic_bounds}, {using $\ell \geq 1$ we have $ \sum_{k={\ell}}^{K-1}\frac{1}{k+1}\leq  \frac{1}{\ell+1}+\ln\left( \frac{K}{\ell+1}\right) \leq 0.5 +\ln\left(\tfrac{N}{\lambda N+1}\right)\leq 0.5-\ln(\lambda)$. Also, $K-\ell \geq K-\lambda K=(1-\lambda)K$.} Thus, we obtain 
\begin{align*}
 \mathbb{E}\left[ \|G_{\eta,1/\gamma} (\x_R)\|^2\right] \leq \frac{{\left( 1-\tfrac{nL_0\gamma}{2\eta}\right)} 2n^2{\gamma} \left( \tfrac{{4\tilde{L}^2_0B}}{\eta^2}+L_0^2\right)({0.5-\ln(\lambda)}) +  {\mathbb{E}\left[{f^{\bf imp}(\x_{\ell})}\right] }{-f^*} +2L_0\eta}{\left( 1-\tfrac{nL_0\gamma}{\eta}\right) \tfrac{\gamma}{4}{(1-\lambda)K}  }.
\end{align*}

\noindent {\bf (c)} The proofs of (c-1) and (c-2) are analogous to those of Theorem \ref{thm:inexact_nonconvex} (c-1) and (c-2), respectively. To show (c-3), note that the total number of lower-level projection steps is given by 
\begin{align*}
\sum_{k=0}^{K_\epsilon} 2N_kt_k& = 2\sum_{k=0}^{K_\epsilon} (k+1)\lceil \tau\ln(k+1) \rceil \leq 2\tau\int_{1}^{K_\epsilon}(x+1)(\ln(x+1)+1)dx = \mathcal{O}\left(\tau K_\epsilon^2\ln(K_\epsilon)\right)\\
& = \mathcal{O}\left(\tau n^4L_0^4{\tilde{L}^4_0}\epsilon^{-2}\ln(n^2L_0^2{\tilde{L}^2_0}\epsilon^{-1})\right).
\end{align*}
\end{proof}}

{\subsubsection{An exact zeroth-order scheme}
Here we present the rate and complexity results for the exact variant of Algorithm~\ref{algorithm:ZSOL_nonconvex_2s}.
\begin{corollary}[{\bf Rate and complexity statements for exact \fyy{(\texttt{ZSOL$^{\bf 2s}_{\rm ncnvx}$})}}]\label{cor:exact_nonconvex_2s}\em
    Consider {Algorithms~\ref{algorithm:ZSOL_nonconvex_2s} (exact variant) for solving \eqref{prob:mpec_as_imp}} and suppose {Assumptions~\ref{ass-1} and~\ref{assum:u_iter_smooth_VR}} hold. 

\noindent {\bf{(a)}} The following holds for any $\gamma<\frac{\eta}{n L_0}$, {$\ell\triangleq  \lceil \lambda K\rceil$}, and all $K> {\tfrac{2}{1-\lambda}}$.
\begin{align*}
 \mathbb{E}\left[ \|G_{\eta,1/\gamma} (\x_R)\|^2\right] \leq \frac{n^2L_0^2{\gamma}({0.5-\ln(\lambda)}) {\left( 1-\tfrac{nL_0\gamma}{2\eta}\right)} +  {\mathbb{E}\left[\fyy{f^{\bf imp}(\x_{\ell})}\right] }{-f^*} +2L_0\eta}{\left( 1-\tfrac{nL_0\gamma}{\eta}\right) \tfrac{\gamma}{4}{(1-\lambda)K}  }.
\end{align*}

\noindent {\bf{(b)}}  Suppose $\gamma= \tfrac{\eta}{2nL_0}$ and $\eta=\tfrac{1}{L_0}$. Let $\epsilon>0$ be an arbitrary scalar and $K_{\epsilon}$ be such that $ \mathbb{E}\left[ \|G_{\eta,1/\gamma} (\x_R)\|^2\right]   \leq \epsilon$. Then,

\noindent {{(b-1)}} The total number of upper-level projection steps on $\Xscr$ is {$K_{\epsilon}=\mathcal{O}\left(n^2L_0^2 \epsilon^{-1}\right)$}.

 \noindent {{(b-2)}} The total sample complexity of upper-level is {$\mathcal{O}\left(n^4L_0^4 \epsilon^{-2}\right)$}.

\end{corollary}
\begin{proof}
The proof can be done in a similar vein to that of Theorem \ref{thm:inexact_nonconvex_2s} by noting that $\fyyy{\tilde{\epsilon}_k}:=0$ in the exact variant.
\end{proof}

} 

\section{Numerical results}\label{sec:5}
In this section, we demonstrate the proposed methodology by{comparing the
performance of the proposed scheme with sample-average approximation (SAA)
schemes on a breadth of  two-stage and single-stage SMPECs of varying structure
and scale \uss{in Sections~\ref{sec:5.1} and ~\ref{sec:5.2}, respectively.} \uss{We then provide confidence intervals in large-scale settings in Section~\ref{sec:5.3} and conclude with a study of how the schemes perform on a set of test problems from the literature (Section~\ref{sec:5.4}).}  Implementations were developed in \texttt{MATLAB} on a PC
with 16GB RAM and 6-Core Intel Core i7 processor (2.6GHz).
\subsection{Two-stage SMPECs}\label{sec:5.1}
In this section, we apply the schemes on a stochastic Stackelberg-Nash-Cournot
equilibrium problem \uss{which leads to a two-stage SMPEC}. The deterministic setting of the problem is derived from
\cite{sherali83stackelberg}. Consider a  market with $N$
profit-maximizing \uss{firms} by competing in Cournot (quantities)  under the (Cournot)
assumption that the remaining firms will hold their outputs at existing levels.
In addition, there \uss{exists} a \fyy{leader}, supplying the same product,
\uss{that} sets production levels by explicitly considering the reaction of the
other $N$ firms to its output variations. We assume that the $i$th Cournot
firm (follower)  supplies $q_i$ units of the product while $f_i(q_i)$ denotes
the \uss{cost of producing $q_i$ units}. In a similar fashion, suppose $x$ denotes the output  of the
\fyy{leader} and let $f(x)$ denote the total cost. Next, let
$p(\cdot,\omega)$ represent the random inverse demand curve. The $N$ Cournot firms
have sufficient capacity installed and can \uss{therefore} wait to observe the
quantities supplied by the \fyy{leader} as well as the realized demand
function before making a decision on their supply quantities. For a given $x\ge
0$, let $\{q_1(x,\omega),\dots,q_N(x,\omega)\}$ be \uss{the} set of quantities for every
$\omega\in\Omega$ \uss{where} each $q_i(x,\omega)$ solve the following profit
maximization problem assuming that $q_j(x,\omega)$, $j \ne i$ are fixed:
\begin{align}
{\max_{q_i \, \ge \, 0}} \quad  q_ip\left(q_i+x+\mbox{$\sum_{j=1,j\ne i}^N$}q_j(x,\omega),\omega \right)-f_i(q_i). \label{snc-1}
\end{align}
Accordingly, let $Q(x,\omega) \triangleq \sum_{i=1}^Nq_i(x,\omega)$. In addition, we assume there exists a capacity limit $x^u$ for $x$. Then $x^*$ is said to be a Stackelberg-Nash-Cournot equilibrium solution if $x^*$ solves
\begin{align}
{\max_{0\, \le \, x \, \le \, x^u}} \  \mathbb{E}[xp(x+Q(x,\omega),\omega)]-f(x). \label{snc-2}
\end{align}
We consider the case of a linear demand curve with convex quadratic cost functions. Specifically, let $p(\uss{u},\omega)=a(\omega)-b\uss{u}$ and let $f_i(q)=\tfrac{1}{2}cq^2$ for $i=1,\cdots,N$, and $f(x)=\tfrac{1}{2}dx^2$. 
Under this condition, the follower's objective can be shown to be strictly concave in $q^i$~\cite{xu06implicit}. Consequently, the concatenated necessary and sufficient equilibrium conditions of the follower-level game are given by the following conditions.
\begin{align} 
     \begin{aligned}
         0 & \leq q \perp \uss{F(q)} - p(x+Q(x,\omega),\omega) {\bf 1} - p'(x+Q(x,\omega)\uss{,\omega}) q   \geq 0, \label{snc-3}
\end{aligned} 
\end{align}
where \uss{$F(q) = \pmat{f'_1(q_1); \cdots;f'_N(q_N)}$}.
We observe that \eqref{snc-3} is a strongly monotone \uss{linear complementarity}  problem for $x \geq 0$ and for every $\omega \in \Omega$. Consequently, $q:
\mathbb{R}_+ \times \Omega \to \mathbb{R}_+^N$ is a single-valued map and is
convex in its first argument for every $\omega$ if $c_j$ is quadratic and
convex~\cite[Prop.~4.2]{demiguel09stochastic}. In fact, it can be claimed that $q(\cdot,\omega)$ is
a piecewise C$^2$ and non-increasing function with $\partial_{x} q(x,\omega)
\subset (-1,0]$ for $X \geq 0$. Consider the leader's problem~\eqref{snc-2}.
Consequently, we have that 
\begin{align*}
    \ic{\Real_+} \ni x & \perp \mathbb{E}\left[-p (x + Q(x,\omega),\omega) + (1 +  \partial_{x} Q(x,\omega)) bx- a(\omega)\right] + \nabla_{x} f(x) \in \ic{\Real_+}. 
\end{align*}
This may be viewed as the following inclusion which has been shown to be monotone~\cite[Thm.~4.4]{demiguel09stochastic}.
\begin{align*}
    0 & \in  \mathbb{E}[T(x,\omega)] + \mathcal{N}_{\Real_+}, \\ 
    \mbox{ where } T(x,\omega) & \triangleq \uss{\left[-p (x+ Q(x,\omega),\omega){\bf 1} - a(\omega){\bf 1}\right] 
    + \nabla_{x} f(x) + \{[(1 +  \partial_{x} Q(x,\omega)) b x]\}}.
\end{align*}
\begin{table}[htb]
\scriptsize
\caption{Errors and time comparison of the three schemes with different parameters}
\begin{center}
\begin{threeparttable} 
    \begin{tabular}{ c | c | c | c | c | c | c | c | c }
    \hline
    \multicolumn{3}{c|}{} & \multicolumn{2}{c|}{\uss{(\texttt{ZSOL}$^{\bf 2s}_{\rm cnvx}$)}} & \multicolumn{2}{c|}{\uss{(\texttt{ZSOL}$^{\bf 2s}_{\rm acc,cnvx}$)}}  & \multicolumn{2}{c}{SAA}  \\ \cline{4-9}
    \multicolumn{3}{c|}{}& & & & & \\[-0.9em]
     \multicolumn{3}{c|}{} & $f^*-f(\bar{x}_K)$ & Time & $f^*-f(x_K)$ & Time & $f^*-f(\hat{x})$ & Time \\ \hline
    \multirow{4}{*}{$N=10$} & \multirow{2}{*}{$b=1$} & $c=0.05$ &  1.2e-3 & 0.1 & 6.6e-5 & 1.4 & 5.4e-4  & 130.2 \\
     & & $c=0.1$ & 8.2e-4 & 0.1 & 4.8e-5 & 1.4 & 4.2e-4 & 109.2 \\
    & \multirow{2}{*}{$b=0.5$} & $c=0.05$ & 1.7e-3 & 0.1 & 7.0e-5 & 1.3 & 3.8e-4 & 122.5 \\
    & & $c=0.1$ & 1.2e-3 & 0.1 & 6.3e-5 & 1.4 & 2.2e-4 & 116.8 \\ \hline
    \multirow{4}{*}{$N=20$} & \multirow{2}{*}{$b=1$} & $c=0.05$ & 4.5e-4 & 0.1 & 2.6e-5 & 1.5 & 2.6e-4 & 426.7 \\
     & & $c=0.1$ & 4.0e-4 & 0.1 & 1.3e-5 & 1.4 & 5.7e-4 & 443.1 \\
     & \multirow{2}{*}{$b=0.5$} & $c=0.05$ & 6.3e-4 & 0.1 & 2.3e-5 & 1.4 & 4.8e-4 & 419.1 \\ 
     & & $c=0.1$ & 4.2e-4 & 0.1 & 2.9e-5 & 1.5 & 3.1e-4 & 450.0\\ \hline
         \multirow{4}{*}{$N=100$} & \multirow{2}{*}{$b=1$} & $c=0.05$ & 9.9e-5 & 0.2 & 3.2e-6 & 4.3 & -- & -- \\
     & & $c=0.1$ & 2.3e-5 & 0.2 & 1.3e-6 & 4.4 & -- & -- \\
     & \multirow{2}{*}{$b=0.5$} & $c=0.05$ & 2.6e-4 & 0.2 & 4.7e-6 & 4.2 & -- & -- \\ 
     & & $c=0.1$ & 2.5e-5 & 0.2 & 1.4e-6 & 4.5 & -- & -- \\ \hline
     \multirow{4}{*}{$N=1000$} & \multirow{2}{*}{$b=1$} & $c=0.05$ & 2.2e-5 & 0.6 & 3.6e-7 & 27.9 & -- & -- \\
     & & $c=0.1$ & 1.7e-6 & 0.6 & 8.3e-8 & 28.8 & -- & -- \\
     & \multirow{2}{*}{$b=0.5$} & $c=0.05$ & 2.5e-5 & 0.6 & 3.1e-7 & 29.1 & -- & -- \\ 
     & & $c=0.1$ & 1.4e-6 & 0.6 & 8.9e-8 & 28.4 & -- & -- \\ \hline
     \multirow{4}{*}{$N=10000$} & \multirow{2}{*}{$b=1$} & $c=0.05$ & 1.0e-5 & 4.6 & 5.2e-7 & 403.5 & -- & -- \\
     & & $c=0.1$ & 6.0e-6 & 4.5 & 3.8e-8 & 392.4 & -- & -- \\
     & \multirow{2}{*}{$b=0.5$} & $c=0.05$ & 1.1e-5 & 4.7 & 5.6e-8 & 334.2 & -- & -- \\ 
     & & $c=0.1$ & 7.1e-6 & 4.6 & 2.7e-8 & 399.7 & -- & -- \\ \hline
    \end{tabular}
\begin{tablenotes}
\small
\item The errors and time in the table are \uss{based on averaging over} 20 runs (`--' implies runtime $ > $ 3600s)
    \end{tablenotes}
  \end{threeparttable}
\end{center}
\label{time}
\end{table}
\noindent {\bf Problem and algorithm parameters.} Suppose there are 
$N=10$ Cournot firms and $c=d=0.1$. Furthermore, $b=1$ and $a(\omega) \sim
\mathcal{U}(7.5,12.5)$ where $\mathcal{U}(l,u)$ denotes the uniform
distribution on $[l,u]$. We choose $\gamma_k =  \tfrac{1}{\sqrt{k+1}}$ and $\eta_k = \tfrac{1}{\sqrt{k+1}}$, $\forall k \ge 1$ in (\texttt{ZSOL}$^{\bf 2s}_{\rm cnvx}$) and $\gamma_k =  \tfrac{1}{2(k+1)}$ and $\eta_k = \tfrac{1}{k+1}$, $\forall k \ge 1$ in (\texttt{ZSOL}$^{\bf 2s}_{\rm acc,cnvx}$). In addition, we choose sample size $N_k = \lfloor k^{1.01} \rfloor$. \\

\noindent {\bf Description of testing.} We compare the performance of (ZSOL) and (acc-ZSOL) with Nesterov's fixed
smoothing scheme  under the same number of iterations in Fig. \ref{tra}. Next we
change the size and parameters of the original game to ascertain parametric
sensitivity. In Table~\ref{time}, we consider a set of 12 problems where the
settings, the empirical errors, and elapsed time are shown in Table~\ref{time}. Note that we have access to the true solution from~\cite{sherali83stackelberg} and this is employed for computing the sub-optimality metrics.  In addition, to show the performance of our proposed schemes, we consider the
(SAA) scheme (utilizing the average of 1000 samples) used in
\cite{demiguel09stochastic}. Let $(\omega_k)_{k=1}^K$ denote independent
identically distributed (i.i.d.) samples. Then, with (SAA) we solve the
following formulation of problem: \begin{align*}
\max_{0\le x\le x^u} \  &\tfrac{1}{K}\sum_{k=1}^K\left[x\cdot(a(\omega_k)-b\cdot(x+Q(x,\omega_k)))\right]-\tfrac{1}{2}dx^2 \\
\st \ & 0 \le q_{i,k} \perp (c+2b)q_{i,k}-a(w_k)+b\cdot\left(x+\mbox{$\sum_{j=1,j\ne i}^N$}q_{j,k}(x,\omega_k)\right) \ge 0, \ \forall i,k.
\end{align*}
This problem allows for utilizing \texttt{NLPEC}~\cite{ferris2002mathematical} in GAMS to compute a solution. For comparison, we employ an alternative method to solve (SAA). (SAA) can be equivalently formulated as \begin{align*}
\max_{0\le x\le x^u} \  &\tfrac{1}{K}\sum_{k=1}^K\left[x\cdot(a(\omega_k)-b\cdot(x+Q(x,\omega_k)))\right]-\tfrac{1}{2}dx^2,
\end{align*}
\ic{where $Q(x,\omega_k) \triangleq \sum_{i=1}^Nq_i(x,\omega_k)$ and $q_i(x,\omega_k)$ is the solution to the following optimization problem:
\begin{align*}
{\max_{q_i \, \ge \, 0}} \quad  q_ip\left(q_i+x+\mbox{$\sum_{j=1,j\ne i}^N$}q_j(x,\omega_k),\omega_k \right)-f_i(q_i).
\end{align*}}
This problem allows for utilizing \ic{gradient based methods} to compute a solution. The results are shown in~\ref{saa}. Next, we provide some key insights from our testing. 

\begin{table}[htb]
\scriptsize
\caption{Errors and time comparison of (SAA) with different solution methods}
\begin{center}
\begin{threeparttable} 
   \ic{ \begin{tabular}{ c | c | c | c | c | c | c  }
    \hline
    \multicolumn{3}{c|}{} & \multicolumn{2}{c|}{SAA(NLPEC)} & \multicolumn{2}{c}{SAA(Gradient)}  \\ \cline{4-7}
    \multicolumn{3}{c|}{}& & & \\[-0.9em]
     \multicolumn{3}{c|}{} & $f^*-f(\hat{x})$ & Time & $f^*-f(\hat{x})$ & Time \\ \hline
    \multirow{4}{*}{$N=10$} & \multirow{2}{*}{$b=1$} & $c=0.05$ & 5.4e-4  & 130.2 & 4.6e-4 & 1.0  \\
     & & $c=0.1$ & 4.2e-4 & 109.2 & 4.5e-4 & 1.0 \\
    & \multirow{2}{*}{$b=0.5$} & $c=0.05$ & 3.8e-4 & 122.5 & 3.3e-4 & 1.0 \\
    & & $c=0.1$ & 2.2e-4 & 116.8 & 2.4e-4 & 1.0 \\ \hline
    \multirow{4}{*}{$N=20$} & \multirow{2}{*}{$b=1$} & $c=0.05$ & 2.6e-4 & 426.7 & 3.1e-4 & 1.1 \\
     & & $c=0.1$ & 5.7e-4 & 443.1 & 4.2e-4 & 1.1 \\
     & \multirow{2}{*}{$b=0.5$} & $c=0.05$ & 4.8e-4 & 419.1 & 5.6e-4 & 1.1 \\ 
     & & $c=0.1$ & 3.1e-4 & 450.0 & 3.8e-4 & 1.1 \\ \hline
         \multirow{4}{*}{$N=100$} & \multirow{2}{*}{$b=1$} & $c=0.05$ & -- & -- & 1.1e-4 & 5.5 \\
     & & $c=0.1$ & -- & -- & 2.8e-5 & 5.5  \\
     & \multirow{2}{*}{$b=0.5$} & $c=0.05$ & -- & -- & 3.0e-4 & 5.5 \\ 
     & & $c=0.1$  & -- & -- & 3.2e-5 & 5.6 \\ \hline
     \multirow{4}{*}{$N=1000$} & \multirow{2}{*}{$b=1$} & $c=0.05$  & -- & -- & 2.3e-5 & 324.7 \\
     & & $c=0.1$  & -- & -- & 1.9e-6 & 312.8 \\
     & \multirow{2}{*}{$b=0.5$} & $c=0.05$  & -- & -- & 2.6e-5 & 306.2 \\ 
     & & $c=0.1$  & -- & -- &  2.1e-6  &  316.5  \\ \hline
    \end{tabular}}
\begin{tablenotes}
\small
\item 
    \end{tablenotes}
  \end{threeparttable}
\end{center}
\label{saa}
\end{table}

\begin{figure}[htbp]
\centering
\includegraphics[width=.4\textwidth]{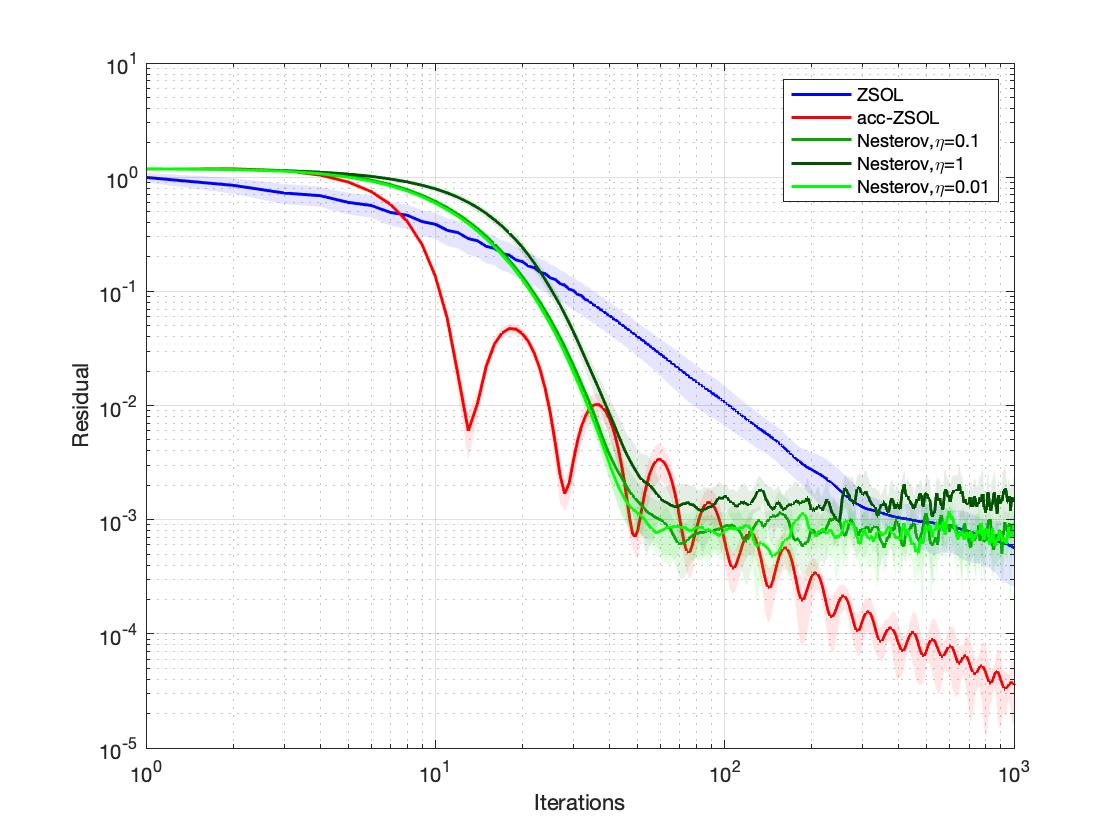}
\caption{Comparision of \uss{(\texttt{ZSOL}$^{\bf 2s}_{\rm cnvx}$)}  and \uss{(\texttt{ZSOL}$^{\bf 2s}_{\rm acc,cnvx}$)} with acceleration with fixed smoothing (Nesterov) on  convex (\uss{SMPEC$^{\bf 2s}$})}
\label{tra}
\end{figure}

\noindent \uss{\bf Insights.} 

\noindent { (i) \em Scalability.} Both \uss{(\texttt{ZSOL}$^{\bf 2s}_{\rm cnvx}$)}  and \uss{(\texttt{ZSOL}$^{\bf 2s}_{\rm acc,cnvx}$)}  show far better scalability in terms of $N$ with modest impact on accuracy and run-time. (SAA) schemes on the other hand grow by a factor of $10$ when number of firms double. In fact, for $N=20$, the (SAA) framework requires CPU time which is between 50 and 100 times greater than that required by the zeroth-order schemes. \uss{(SAA) schemes could not produce solutions for $N \geq 100$ in our tests \uss{while our proposed schemes can contend with problems with $N = 10,000$ within 5s in the unaccelerated regime}. \uss{The lack of scalability tends to be less surprising since the sample-average subproblems require solving MPECs with $\mathcal{O}(N)$  constraints and as $N$ becomes large, direct solutions become challenging, as reflected by the computational times.}} \ic{We observe that the gradient based approach that uses sample-averages appears to scale better than \texttt{NLPEC}. However, we still see a difference in performance and quality between the gradient-enabled SAA scheme and the proposed implicit SA framework.}

\noindent {(ii) \em Accuracy.} The accelerated scheme provides nearly $10$ times more accurate solutions than the unaccelerated scheme at a modest computational cost. \uss{This is aligned with the superior error bounds of such schemes compared to their unaccelerated counterparts.}  

\noindent {(iii) \em Comparison of accelerated schemes.} Figure~\ref{tra}
demonstrates the benefits of diminishing smoothing sequences as the scheme
suggested in~\cite{nesterov17} degenerates for different values of the fixed
smoothing parameter. Notably, \uss{(\texttt{ZSOL}$^{\bf 2s}_{\rm acc,cnvx}$)}
shows no such degeneration and progressively improves in function value. We
    notice in Table~\ref{time}, \uss{(\texttt{ZSOL}$^{\bf 2s}_{\rm acc,cnvx}$)}
    takes longer than \uss{(\texttt{ZSOL}$^{\bf 2s}_{\rm cnvx}$)} with the same
    number iterations, \uss{arising from the fact that \uss{(\texttt{ZSOL}$^{\bf 2s}_{\rm acc,cnvx}$)}} utilizes an increasing sample size and solves more
lower-level problems than \uss{(\texttt{ZSOL}$^{\bf 2s}_{\rm cnvx}$)}.

\noindent {(iv) \em Performance of {\em (\texttt{ZSOL}$^{\bf 2s}_{\rm cnvx}$)} with various $\gamma_k$ and $\eta_k$.} As shown in Table~\ref{ab}, we compare the results generated by (\texttt{ZSOL}$^{\bf 2s}_{\rm cnvx}$) with various values of $(a,b)$ used in $\gamma_k \coloneqq \frac{\gamma_0}{(k+1)^a}$ and $\eta_k \coloneqq \frac{\eta_0}{(k+1)^b}$. As shown in the table, for this particular problem, we find that smaller $a$ ($a=0.5$) generates better results in (\texttt{ZSOL}$^{\bf 2s}_{\rm cnvx}$). When the size of problem is large ($N=1000$), fixing $a=0.5$, larger values of $b$ lead to smaller residuals.

\begin{table}[htbp]
\scriptsize
\caption{Errors of (\texttt{ZSOL}$^{\bf 2s}_{\rm cnvx}$) with various $\gamma_k$ and $\eta_k$}
\begin{center}
\begin{threeparttable} 
\begin{tabular}{c}
    \begin{tabular}[t]{ c | c | c | c |  c  | c | c}
    \hline
    & $(a,b)$ & $(0.5,0.5)$ & $(0.5,0.7)$ & $(0.5,0.9)$ & $(0.7,0.4)$ & $(0.9,0.2)$ \\ \hline
    \multirow{3}{*}{$f^*-f(\bar{x}_K)$} & $N = 10$ & 1.2e-3 & 1.7e-3 & 1.5e-3 & 1.9e-3 & 7.7e-2 \\ \cline{2-7}
     & $N = 100$ & 2.5e-5 & 3.0e-5 & 2.6e-5 & 1.1e-3 & 1.6e-2 \\ \cline{2-7}
     & $N = 1000$ & 1.4e-6 & 4.8e-7 & 4.4e-7 & 2.9e-4 & 7.1e-4 \\ \hline
   \end{tabular}
    \end{tabular}
    
\begin{tablenotes}
\small
    \end{tablenotes}
  \end{threeparttable}
\end{center}
\label{ab}
\end{table}

\subsection{Single-stage SMPECs} \label{sec:5.2}
\uss{We consider both the convex and the nonconvex regimes next.} 

\noindent {\bf I. A convex implicit function.} {First, we consider a single-stage SMPEC where the the lower level is a parametrized stochastic variational inequality, i.e. given $x$, the  lower-level problem is a noncooperative game in which the $i$th player solves the following problem.} \[
{\max_{q_i\ge 0}} \  \mathbb{E}[q_i(a(\omega)-b(q_i+x+\mbox{$\sum_{j\ne i}q_j(x)$})]-\tfrac{1}{2}cq_i^2,
\]
Accordingly, the upper-level problem in $x$ is \uss{defined} as follows
\begin{align*}
{\max_{0\le x\le x^u}} \  \mathbb{E}\left[x(a(\uss{\xi})-b(x+\mbox{$\sum_{i=1}^Nq_i(x)$}))\right]-\tfrac{1}{2}dx^2.
\end{align*}
\begin{table}[h]
\scriptsize
\caption{Comparison of \uss{(\texttt{ZSOL}$^{\bf 1s}_{\rm cnvx}$)} and (SAA) (\uss{Convex implicit function})}
\begin{center}
\begin{threeparttable} 
    \begin{tabular}{ c | c | c | c | c | c | c  }
    \hline
    \multicolumn{3}{c|}{} & \multicolumn{2}{c|}{(\texttt{ZSOL}$^{\bf 1s}_{\rm cnvx}$)}  & \multicolumn{2}{c}{SAA}  \\ \cline{4-7}
       \multicolumn{3}{c|}{}& & & & \\[-0.9em]
     \multicolumn{3}{c|}{} & $f^*-f(\bar{x}_K)$ & Time & $f^*-f(\hat{x})$  & Time \\ \hline
    \multirow{4}{*}{$N=10^2$} & \multirow{2}{*}{$b=0.01$} & $c=3$ &  6.9e-4 & 0.1 & 2.2e-4 & 0.05 \\
     & & $c=5$ & 3.7e-4 & 0.1 &  2.4e-4 & 0.05 \\
    & \multirow{2}{*}{$b=0.02$} & $c=3$ & 8.1e-4 & 0.1 &  7.3e-4 & 0.05 \\
    & & $c=5$ & 3.5e-4 & 0.1 & 4.0e-4 & 0.05 \\ \hline
    \multirow{4}{*}{$N=10^3$} & \multirow{2}{*}{$b=0.01$} & $c=3$ & 7.0e-4 & 0.4 & 7.0e-4 & 1.2 \\
     & & $c=5$ & 4.3e-4 & 0.4 & 5.0e-4 & 1.1 \\
     & \multirow{2}{*}{$b=0.02$} & $c=3$ & 8.0e-4 & 0.4 & 6.8e-4 & 1.2 \\ 
     & & $c=5$ & 4.7e-4 & 0.4 & 4.2e-4 & 1.2 \\ \hline
         \multirow{4}{*}{$N=10^4$} & \multirow{2}{*}{$b=0.01$} & $c=3$ & 5.1e-4 & 5.8 & 7.3e-4 & 88.6 \\
     & & $c=5$ & 2.5e-4 & 5.2 & 5.4e-4 & 85.7 \\
     & \multirow{2}{*}{$b=0.02$} & $c=3$ & 6.4e-4 & 5.6 & 4.3e-4 & 93.5 \\ 
     & & $c=5$ & 3.1e-4 & 5.3 & 4.7e-4 & 87.3 \\ \hline
              \multirow{4}{*}{$N=10^5$} & \multirow{2}{*}{$b=0.01$} & $c=3$ & 8.7e-4 & 45.6 & -- & -- \\
     & & $c=5$ & 6.5e-4 & 47.1 & -- & -- \\
     & \multirow{2}{*}{$b=0.02$} & $c=3$ & 9.7e-4 & 46.3 & -- & -- \\ 
     & & $c=5$ & 7.5e-4 & 46.7 & -- & -- \\ \hline
    \end{tabular}
\begin{tablenotes}
\small
\item The errors and time in the table are \uss{based on averaging over} 20 runs (`--' implies runtime $ > $ 3600s)
    \end{tablenotes}
  \end{threeparttable}
\end{center}
\label{time2}
\end{table}
\uss{Since the lower-level equilibrium problem has a unique solution (since it is characterized by a strongly monotone map), the resulting implicit function can be shown to be  convex. }

\noindent {\bf Algorithm and Problem parameters.} We assume $b=0.01$ and $c=3$ here, other parameters are the same as in the
previous section. It can be shown that $\mu_F=3.01$ and $L_F=3.11$. We assume that $\gamma_k =  \tfrac{1}{\sqrt{k+1}}$ and $\eta_k
= \tfrac{1}{\sqrt{k+1}}$ for \uss{(\texttt{ZSOL}$^{\bf 1s}_{\rm cnvx}$)}. In \uss{(\texttt{ZSOL}$^{\bf 1s}_{\rm cnvx}$)}, we run $10^3$ iterations. In the
lower-level's variance-reduced stochastic approximation scheme, we choose
steplength $\alpha = 0.15$, sampling rate $\rho=\frac{1}{1.5}$ and the sample
size $M_t=\lceil 10^{-4} \cdot1.5^t \rceil$. Thus we may calculate that $\tau
\ge 4.9$ and then we choose $t_k=\lceil 5\ln(k+1) \rceil$. In Fig.~\ref{tra2},
we show the trajectories for \uss{(\texttt{ZSOL}$^{\bf 1s}_{\rm cnvx}$)} under various algorithm parameters. 
\begin{figure}[htbp]
\centering
\includegraphics[width=.4\textwidth]{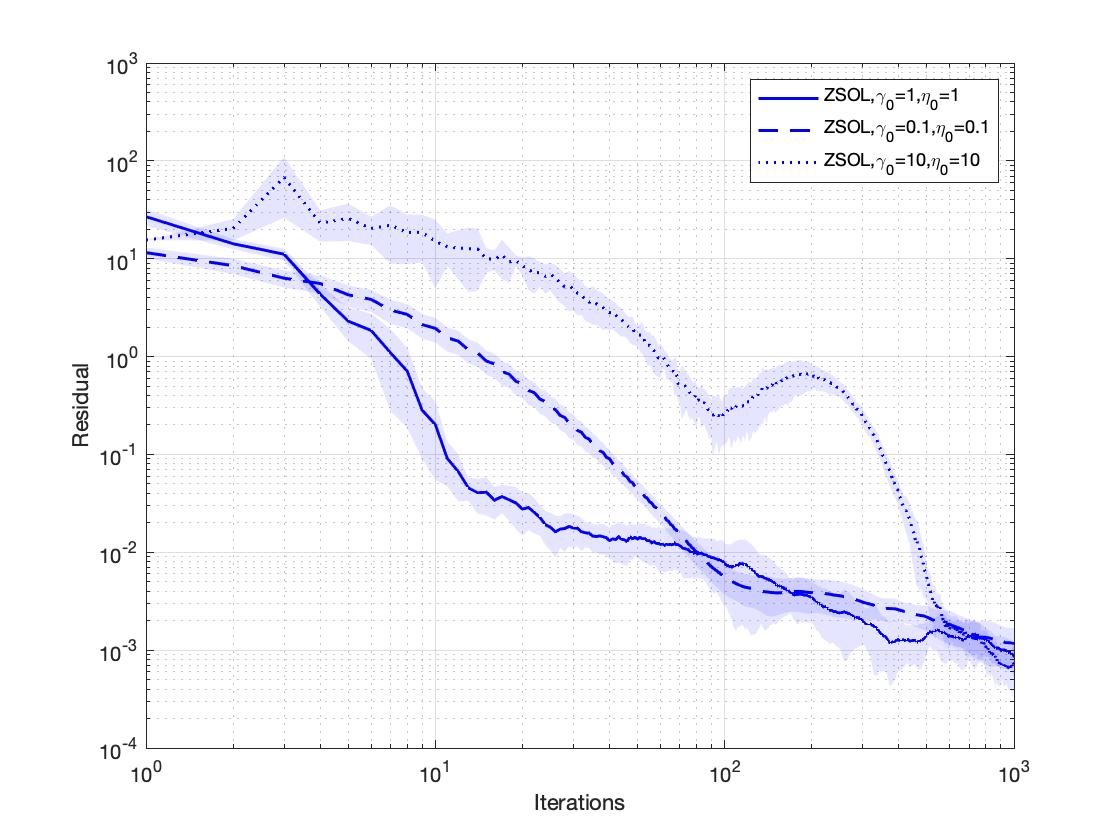}
\includegraphics[width=.4\textwidth]{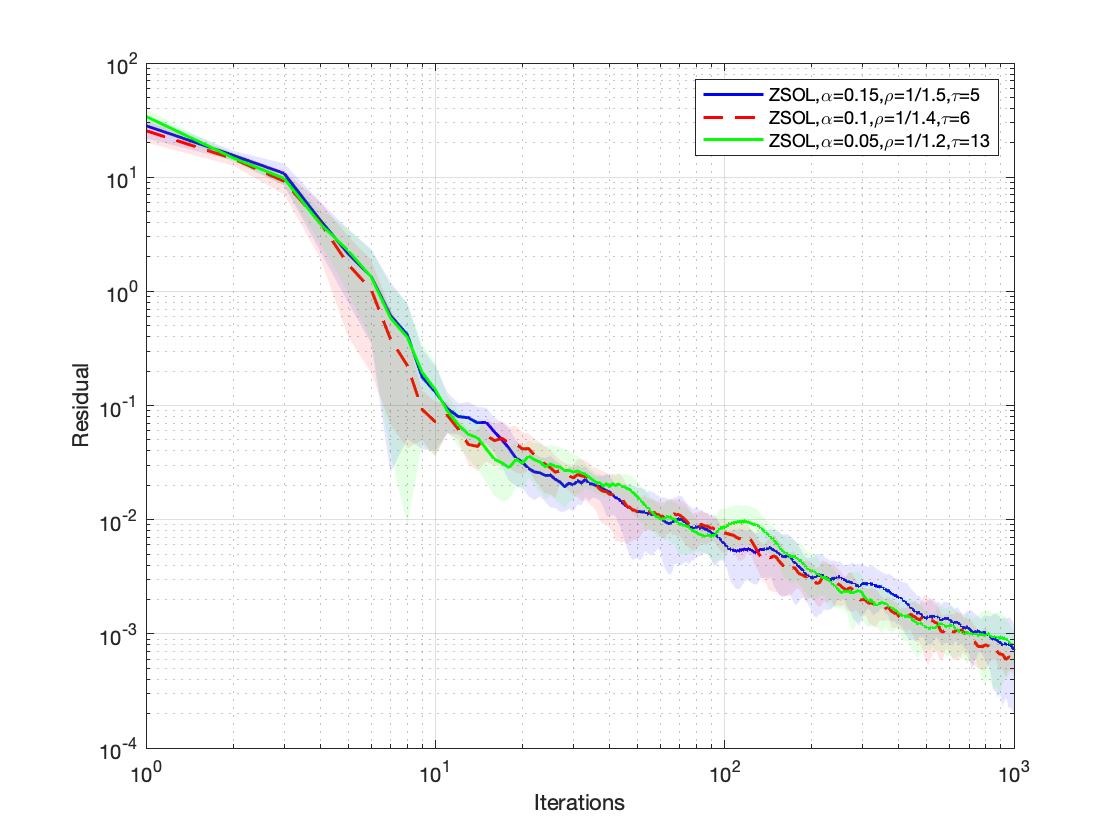} \hfill
\caption{Trajectories for \uss{(\texttt{ZSOL}$^{\bf 1s}_{\rm cnvx}$)} on the convex SMPEC$^{\bf 1s}$}
\label{tra2}
\end{figure}
Again, we compare the errors and time between \uss{(\texttt{ZSOL}$^{\bf 1s}_{\rm cnvx}$)} and (SAA) in Table~\ref{time2}. Here, with (SAA) we solve the following optimization problem
\begin{align*}
{\maximize {0\le x\le x^u}} \  &\tfrac{1}{K}\sum_{k=1}^K\left[x(a(\omega_k)-b(x+Q(x)))\right]-\tfrac{1}{2}dx^2 \\
\st \ & 0 \le q_i \perp \tfrac{1}{L}\mbox{$\sum_{\ell=1}^L$}\left[(c+2b)q_i-a(w_\ell)+b\left(x+\mbox{$\sum_{j=1,j\ne i}^N$}q_j(x)\right)\right] \ge 0, \ \forall i. 
\end{align*}
In (SAA), we use $10^3$ samples in both the upper and lower-level problems.
We also employ a gradient based method (Fig.~\ref{saa2}) to solve the following equivalent (SAA) model:
\begin{align*}
{\max_{0\le x\le x^u}} \  &\tfrac{1}{K}\sum_{k=1}^K\left[x(a(\omega_k)-b(x+Q(x)))\right]-\tfrac{1}{2}dx^2,
\end{align*}
\ic{where $Q(x) \triangleq \sum_{i=1}^Nq_i(x)$ and $q_i(x)$ is the solution to the following optimization problem:
\begin{align*}
{\max_{q_i\ge 0}} \  \mathbb{E}[q_i(a(\omega)-b(q_i+x+\mbox{$\sum_{j\ne i}q_j(x)$})]-\tfrac{1}{2}cq_i^2.
\end{align*}}

\begin{table}[h]
\scriptsize
\caption{Comparison of (SAA) with different solution methods}
\begin{center}
\begin{threeparttable} 
   \ic{ \begin{tabular}{ c | c | c | c | c | c | c  }
    \hline
    \multicolumn{3}{c|}{} & \multicolumn{2}{c|}{SAA(NLPEC)}  & \multicolumn{2}{c}{SAA(Gradient)} \\ \cline{4-7}
       \multicolumn{3}{c|}{}& & & & \\[-0.9em]
     \multicolumn{3}{c|}{} & $f^*-f(\hat{x})$ & Time & $f^*-f(\hat{x})$  & Time \\ \hline
    \multirow{4}{*}{$N=10^2$} & \multirow{2}{*}{$b=0.01$} & $c=3$ & 2.2e-4 & 0.05 & 3.9e-4 & 0.4 \\
     & & $c=5$ & 2.4e-4 & 0.05 & 2.6e-4 & 0.4 \\
    & \multirow{2}{*}{$b=0.02$} & $c=3$ & 7.3e-4 & 0.05 & 5.9e-4 & 0.4 \\
    & & $c=5$ & 4.0e-4 & 0.05 & 3.7e-4 & 0.4 \\ \hline
    \multirow{4}{*}{$N=10^3$} & \multirow{2}{*}{$b=0.01$} & $c=3$ & 7.0e-4 & 1.2 & 6.0e-4 & 2.5 \\
     & & $c=5$ & 5.0e-4 & 1.1 & 4.4e-4 & 2.5 \\
     & \multirow{2}{*}{$b=0.02$} & $c=3$ & 6.8e-4 & 1.2 & 5.9e-4 & 2.6 \\ 
     & & $c=5$ & 4.2e-4 & 1.2 & 3.8e-4 & 2.6 \\ \hline
         \multirow{4}{*}{$N=10^4$} & \multirow{2}{*}{$b=0.01$} & $c=3$ & 7.3e-4 & 88.6 & 5.9e-4 & 25.3 \\
     & & $c=5$ & 5.4e-4 & 85.7 & 4.5e-4 & 25.3 \\
     & \multirow{2}{*}{$b=0.02$} & $c=3$ & 4.3e-4 & 93.5 & 5.2e-4 & 25.2 \\ 
     & & $c=5$ &4.7e-4 & 87.3 & 4.2e-4 & 25.9 \\ \hline
              \multirow{4}{*}{$N=10^5$} & \multirow{2}{*}{$b=0.01$} & $c=3$ & -- & -- & 6.7e-4 & 94.7 \\
     & & $c=5$ & -- & -- & 5.4e-4  & 95.0 \\
     & \multirow{2}{*}{$b=0.02$} & $c=3$ & -- & -- & 8.1e-4 & 96.3 \\ 
     & & $c=5$ & -- & -- & 6.0e-4 & 95.2 \\ \hline
    \end{tabular}}
\begin{tablenotes}
\small
\item The errors and time in the table are \uss{based on averaging over} 20 runs (`--' implies runtime $ > $ 3600s)
    \end{tablenotes}
  \end{threeparttable}
\end{center}
\label{saa2}
\end{table}

\noindent {\bf Insights.} 

\noindent (i) {\em Scalability.} \uss{We observe that the CPU times for  \uss{(\texttt{ZSOL}$^{\bf 1s}_{\rm cnvx}$)} grow by a factor
of approximately 450 when $N$ grows by a factor of $1000$ (from $10^2$ to $10^5$); however (SAA) schemes show a growth in CPU time of $1770$ when $N$ grows by a factor of $100$ (from $10^2$ to $10^4$). In fact, (SAA) schemes cannot process problems for $N=10^5$ in the prescribed time.} 

\noindent (ii) {\em Accuracy.} Both approaches
provide similar accuracy but zeroth-order schemes require less than 6s in CPU
time when $N=10^4$ while the (SAA) framework requires approximately 85s. The
accuracy of \uss{(\texttt{ZSOL}$^{\bf 1s}_{\rm cnvx}$)} is relatively robust to changing steplength and sampling rates
at the lower-level but does tend to be sensitive to changing the initial
steplength at the upper-level; however, as the scheme progresses, the impact of
initial steplengths tends to be muted.

\begin{table}[htb]
\scriptsize
\caption{Errors comparison of the three schemes with different parameters}
\begin{center}
\begin{threeparttable} 
    \begin{tabular}{ c | c | c | c | c    }
    \hline
    \multicolumn{2}{c|}{} & {\texttt{ZSOL}$^{\bf 1s}_{\rm ncvx}$} & NLPEC  & BARON \\ \cline{3-5}
    \multicolumn{2}{c|}{}& & & \\[-0.9em]
     \multicolumn{2}{c|}{} & $f(x_K)$  & \uss{Stationary point}  & global optimum \\ \hline
    \multirow{3}{*}{$(a,b)=(1,0)$} & $(c,d)=(1,1)$  &  -7.50  & -7.20  & -7.50  \\
     & $(c,d)=(2,2)$ & -9.23 & -9.04 & -9.23  \\
    & $(c,d)=(3,3)$ & -9.25 & -9.10 & -9.25 \\ \hline
    \multirow{3}{*}{$(a,b)=(5,0)$} &$(c,d)=(1,1)$ &  -11.50  & -7.20  & -11.50  \\ 
     & $(c,d)=(2,2)$ & -13.23 & -9.04 & -13.23 \\ 
     & $(c,d)=(3,3)$ & -13.25 & -9.10 & -13.25 \\ \hline
     \multirow{3}{*}{$(a,b)=(10,0)$} & $(c,d)=(1,1)$ &  -16.48  & -7.20  & -16.50  \\ 
     & $(c,d)=(2,2)$ & -18.20 & -9.04 & -18.23 \\
     & $(c,d)=(3,3)$ & -18.23 & -9.10 & -18.25 \\ \hline
    \end{tabular}
\begin{tablenotes}
\small
\item The errors of \uss{(\texttt{ZSOL}$^{\bf 1s}_{\rm ncvx}$)} are \uss{based on averaging over} 20 runs 
    \end{tablenotes}
  \end{threeparttable}
\end{center}
\label{time3}
\end{table}

\noindent {\bf II. A nonconvex implicit function.} \label{5.3}
The second example, inspired from \cite{bard1988convex}, is a bilevel problem with a strongly monotone mapping in the lower-level. We add a stochastic component in the lower-level to make the mapping expectation-valued. Formally, this problem is defined as follows.  
\[
\begin{aligned}
&{\minimize {x} } &&{-x_1^2}-3x_2-4y_1\uss{(x)}+(y_2\uss{(x)})^2 \\
&\st &&x_1^2+2x_2 \le 4, \quad  
0\le x_1\le 1, \quad  0\le x_2\le 2,  
\end{aligned}
\]
where $y(x)$ is a solution to the following parametrized optimization problem. 
\[
\begin{aligned}&{\minimize {y} } &&\mathbb{E}\left[2x_1^2+y_1^2+y_2^2-\xi(\omega)y_2\right] \\
&\st &&x_1^2-2x_1+x_2^2-2y_1+y_2 \ge -3, \ x_2+3y_1-y_2 \ge 4,  \
y_1\ge0, y_2\ge 0,
\end{aligned}
\]

\noindent {\bf Problem and algorithm parameters.} We assume $\xi(\omega) \sim \mathcal{U}(4,6)$ and run \uss{(\texttt{ZSOL}$^{\bf 1s}_{\rm ncvx}$)} for 
$10^4$ iterations, \uss{choosing $\eta= 10^{-2}$ and $\gamma =  10^{-3}$} in \uss{(\texttt{ZSOL}$^{\bf 1s}_{\rm ncvx}$)}. In
addition, we choose $\alpha_0=1$ and $\alpha_t = \tfrac{\alpha_0}{t+0.01}$ for
$t=0,1,\dots,t_k-1$ in the  stochastic approximation method applied to the lower-level.
We compare the performance of \uss{(\texttt{ZSOL}$^{\bf 1s}_{\rm ncvx}$)} on this problem in Fig. \ref{bard}
for varying algorithm parameters, all of which suggest that the resulting
sequences steadily converge to the global minimizer. To test the power of
\uss{(\texttt{ZSOL}$^{\bf 1s}_{\rm ncvx}$)} on different problems, we change the objective function of upper-level
and lower-level to ${-ax_1^2}-bx_2^2-3x_2-4y_1+y_2^2$ and
$\mathbb{E}[2x_1^2+cy_1^2+dy_2^2-\xi(\omega)y_2]$, respectively. Then we vary
the values of $a$, $b$, $c$ and $d$. For comparison,  we also run each problem
using solvers \texttt{NLPEC} and \texttt{BARON}
\cite{ts:05,sahinidis:baron:21.1.13} on the
NEOS Server \cite{czyzyk_et_al_1998,dolan_2001,gropp_more_1997}. We record the
empirical errors of each scheme for 9 different settings, as shown in \uss{Table
\ref{time3}}. In \uss{(\texttt{ZSOL}$^{\bf 1s}_{\rm ncvx}$)}, we use $10^4$ samples in each test problem.

\begin{figure}[htb]
\centering
\includegraphics[width=.4\textwidth]{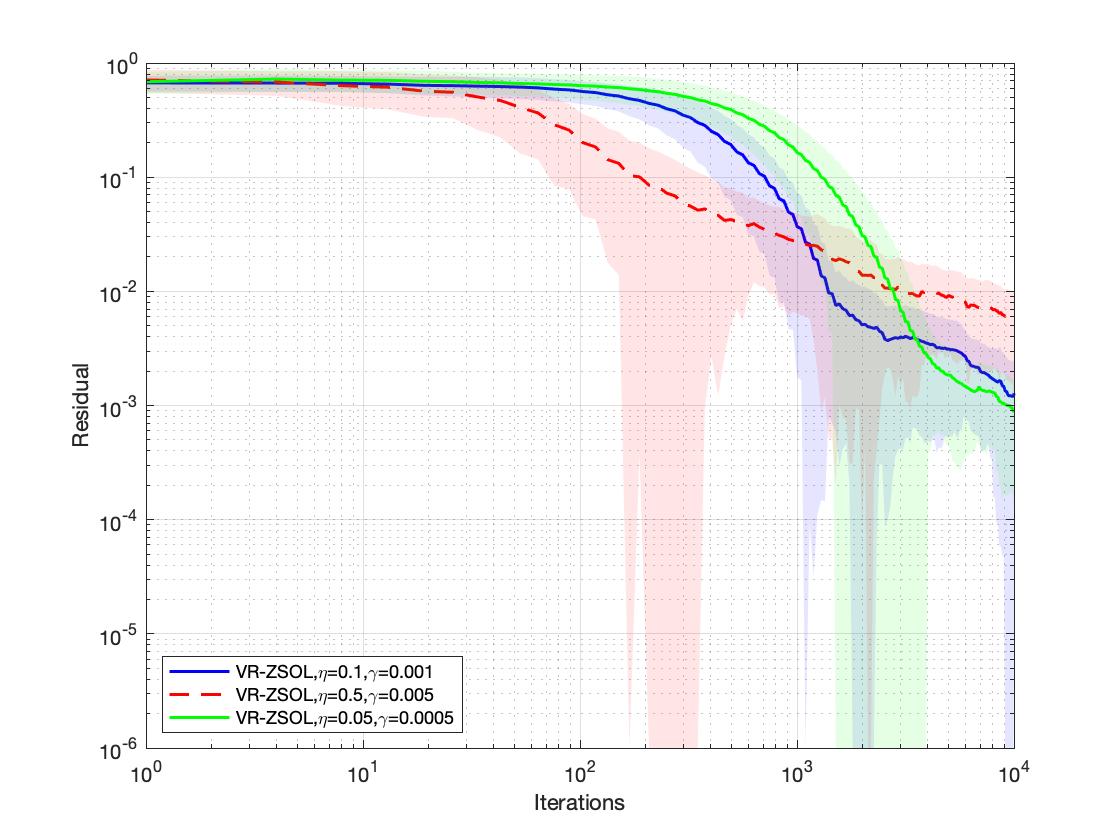}
\includegraphics[width=.4\textwidth]{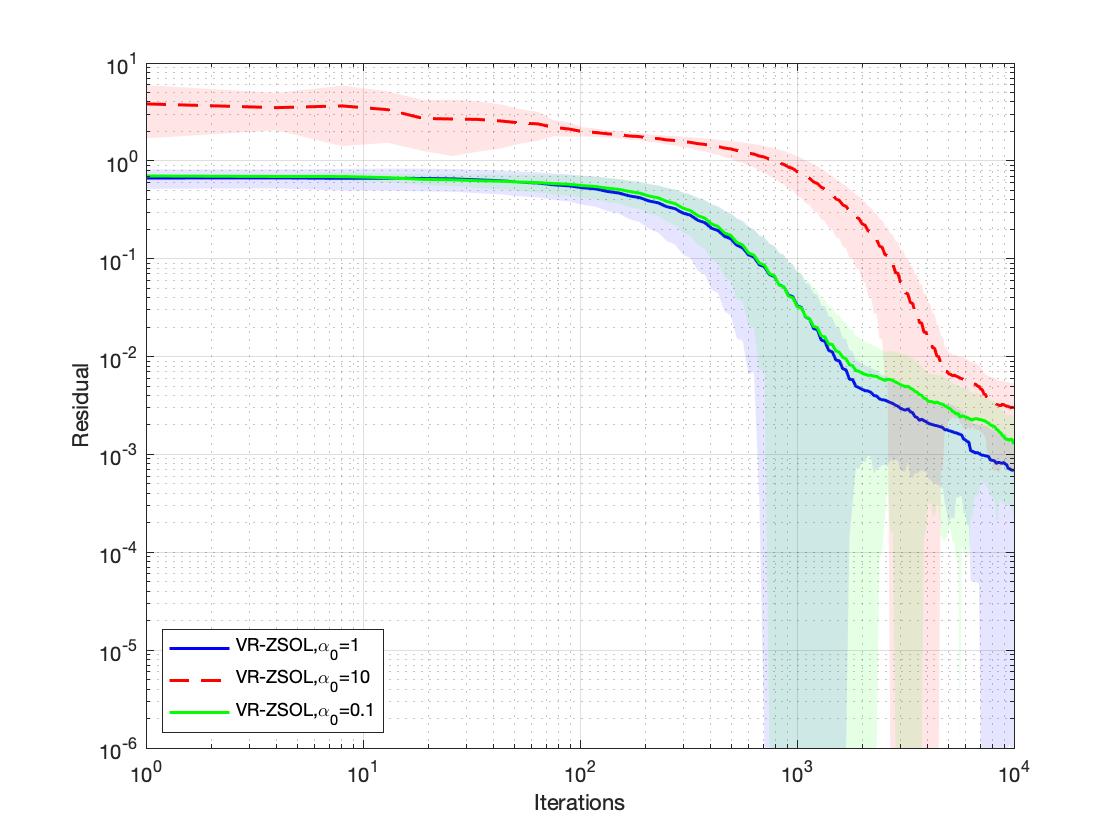} \hfill
\caption{Trajectories for \uss{(\texttt{ZSOL}$^{\bf 1s}_{\rm ncvx}$)} on the nonconvex \fyy{\eqref{eqn:SMPECepx}}}
\label{bard}
\end{figure}

\noindent {\bf Insights.} 

\noindent {\em Global minimizers.} From Fig.~\ref{bard}, we observe that while all of
the implementations perform well, large initial steplengths at the lower-level
tend to lead relatively worse compared to more modest steplengths.
Table~\ref{time3} is instructive in that it shows that \uss{(\texttt{ZSOL}$^{\bf 1s}_{\rm ncvx}$)} produces
values close to the global minimum  as obtained by \texttt{BARON} for all nine
problem instances. Notably, solvers such as \texttt{NLPEC} are equipped with
convergence guarantees to \uss{stationary} points  and provide somewhat
poorer values upon termination. 

\subsection{Confidence intervals for high-dimensional problems} \label{sec:5.3}
To validate the effectiveness of solutions generated by \uss{(\texttt{ZSOL}$^{\bf 1s}_{\rm cnvx}$)} and \uss{(\texttt{ZSOL}$^{\bf 2s}_{\rm cnvx}$)}, we construct 95\% \uss{confidence} intervals for large-scale test problems from Table~\ref{time} and \ref{time2}. The results are shown in Table~\ref{ci}. Note that \uss{(\texttt{ZSOL}$^{\bf 2s}_{\rm acc,cnvx}$)} can process two-stage SMPECs. All confidence intervals presented are relatively narrow, validating the quality of corresponding solutions.
\begin{table}[htb]
\scriptsize
\caption{Errors and confidence intervals for high dimensional problems from Table~\ref{time} and \ref{time2}}
\begin{center}
    \begin{tabular}{ c | c | c | c | c | c | c  }
    \hline
    \multicolumn{3}{c|}{} & \multicolumn{2}{c|}{\uss{(\texttt{ZSOL}$^{\bf 2s}_{\rm cnvx}$) [Table~\ref{time}], (\texttt{ZSOL}$^{\bf 1s}_{\rm cnvx}$)}[Table~\ref{time2}]} & \multicolumn{2}{c}{\uss{(\texttt{ZSOL}$^{\bf 2s}_{\rm acc, cnvx}$)}}   \\ \cline{4-7}
    \multicolumn{3}{c|}{}& & & \\[-0.9em]
     \multicolumn{3}{c|}{} & $f^*-f(\bar{x}_K)$ & CI & $f^*-f(x_K)$ & CI  \\ \hline
     \multirow{2}{*}{Table~\ref{time}} & \multirow{2}{*}{$b=1$} & $c=0.05$ & 1.0e-5 & [0.9e-5,1.1e-5] & 5.2e-7 & [5.0e-7,5.4e-7] \\ \cline{3-7}
     & & $c=0.1$ & 6.0e-6 & [5.9e-6,6.1e-6] & 3.8e-8 & [3.4e-8,4.2e-8]  \\ \cline{2-7}
    \multirow{2}{*}{$N=10^4$} & \multirow{2}{*}{$b=0.5$} & $c=0.05$ & 1.1e-5 & [1.0e-5,1.2e-5] & 5.6e-8 & [5.2e-8,6.0e-8]  \\ \cline{3-7}
     & & $c=0.1$ & 7.1e-6 & [7.0e-6,7.2e-6] & 2.7e-8 & [2.4e-8,3.0e-8] \\ \hline
     \multirow{2}{*}{Table~\ref{time2}} & \multirow{2}{*}{$b=0.01$} & $c=3$ & 8.7e-4 & [7.5e-4,9.9e-4] & n/a & n/a \\ \cline{3-7}
     & & $c=5$ & 6.5e-4 & [5.9e-4,7.1e-4] & n/a & n/a \\ \cline{2-7}
     \multirow{2}{*}{$N=10^5$}& \multirow{2}{*}{$b=0.02$} & $c=3$ & 9.7e-4 & [8.0e-4,1.1e-3] & n/a & n/a \\ \cline{3-7}
     & & $c=5$ & 7.5e-4 & [6.4e-4,8.6e-4] & n/a & n/a \\ \hline
    \end{tabular}
\end{center}
\label{ci}
\end{table}
\begin{table}[h]
\scriptsize
\caption{Results comparison with solutions from the literature}
\begin{center}
    \begin{tabular}{ c  c | c c | c  c  }
    \hline
	\multicolumn{2}{c|}{\multirow{2}{*}{Problem}}  & \multicolumn{2}{c|}{\uss{(\texttt{ZSOL}$^{\bf 2s}_{\rm ncvx}$)}}  & \multicolumn{2}{c}{Literature}  \\ \cline{3-6}
	\multicolumn{1}{c}{} & & $f^*$ &$x^*$ & $f^*$ & $x^*$ \\ \hline
\multirow{3}{*}{Problem 1}	& $L=150$, $\gamma=1.0$  & -343.35 & 55.57 & -343.35 & 55.55 \\ 
	& $L=150$, $\gamma=1.1$  & -203.15 & 42.57 & -203.15 & 42.54 \\
	& $L=150$, $\gamma=1.3$  & -68.14 & 24.19 & -68.14 & 24.14 \\ \hline
	Problem 2 &  & -1.00 & (0.50,0.50) & -1.00 & (0.50,0.50) \\ \hline
	Problem 3 &  & 0.01 & (0.00,0.00) & 0.01 & (0.00,0.00) \\ \hline
	Problem 4 &  & 0.00 & (5.00,8.99) & 0.00 & (5.00,9.00) \\ \hline
	\multirow{3}{*}{Problem 5} & $0.5((y_1-3)^2+(y_2-4)^2)$  & 3.20 & 4.06 & 3.20 & 4.06 \\ 
	& $0.5((y_1-3)^2+(y_2-4)^2+(y_3-1)^2)$  & 3.45 & 5.13 & 3.45 & 5.15 \\ 
	& $0.5((y_1-3)^2+(y_2-4)^2+10y_4^2)$  & 4.60 & 2.39 & 4.60 & 2.39 \\ 
    \hline
    \end{tabular}
\end{center}
\label{p0}
\end{table}

\begin{table}[h]
\scriptsize
\caption{Results of high-dimensional counterparts }
\begin{center}
\begin{threeparttable} 
    \begin{tabular}{ c | c | c | c | c  | c | c | c | c | c }
    \hline
     \multirow{2}{*}{Problem} & \multirow{2}{*}{$N$} & \multicolumn{3}{c|}{\uss{(\texttt{ZSOL}$^{\bf 2s}_{\rm ncvx}$)}}  & \multicolumn{2}{c}{} & \multicolumn{3}{|c}{SAA}  \\ \cline{3-10}
         \multicolumn{1}{c|}{} &\multicolumn{1}{c|}{}& & & & & \\[-0.9em]
      & & $\hat{f}(x_K)$ & CI & Time & $\underline{lb}$ & CI & $\hat{f}(\hat{x})$ & CI & Time \\ \hline \hline
  \multirow{4}{*}{Problem 1} &  $5$ & -462.6 & [-463.1,-462.1] & 0.8 & -462.8  & [-464.0,-461.5] & -461.9 & [-463.1,-460.7] & 5.3 \\
   & $10$& -174.4 & [-174.6,-174.2] & 0.9 & -174.7 & [-175.2,-174.2] & -174.2 & [-174.8,-173.6] & 23.3 \\
   &  $100$ & -5.101 & [-5.105,-5.097] & 1.3 & -- & --  & --  & -- & --\\
      &   $1000$ & -0.071 & [-0.072,-0.071] & 5.2 & -- & --  & -- & -- & -- \\ \hline
   \multirow{4}{*}{Problem 2}     &  2 & -0.882  &[-0.883,-0.881] & 0.6 & -0.883 & [-0.886,-0.880] &  -0.882 & [-0.886,-0.878] & 4.2 \\ 
    & 10& -4.408 & [-4.410,-4.406] & 0.9 & -4.408 & [-4.414,-4.402] & -4.406 & [-4.414,-4.398] & 29.6 \\ 
      & 100 & -44.07 & [-44.08,-44.07] & 5.5 & -- & -- & -- & -- & -- \\ 
        & 1000 & -439.7 & [-439.7,-439.7] & 98.1 & -- & -- & -- & -- & -- \\ \hline
    \end{tabular}
\begin{tablenotes}
\small
    \end{tablenotes}
  \end{threeparttable}
\end{center}
\label{hd}
\end{table}

\subsection{Additional tests on deterministic and two-stage stochastic MPECs}\label{sec:5.4}

We test our schemes on test problems from the literature. In all \uss{of} the test problems, the \uss{lower-level parametrized} VI is strongly monotone, implying that the lower-level decision is uniquely determined by a $\x \in \mathcal{X}$. \\

\noindent {\bf Problem and algorithm parameters.} The problems and their parameters are described in Appendix. We use the same algorithm parameters as those in \ref{5.3}(II). In Table~\ref{p0}, we compare the results generated by \uss{(\texttt{ZSOL}$^{\bf 2s}_{\rm ncvx}$)} and those from the literature, while in Table~\ref{hd}, we extend some of the existing problems to their stochastic counterparts with larger dimensions.  \\

\noindent {\bf Insights.} 

\noindent { (i) \em Scalability.} Again, \uss{(\texttt{ZSOL}$^{\bf
    2s}_{\rm ncvx}$)} shows far better scalability in terms of $N$ with
    modest impact on accuracy and run-time. For both problems in
    Table~\ref{hd}, (SAA) schemes take around 5-20 times more time on small
    scale problems while when $N \ge 100$ on the other hand, no solutions are
    produced within \uss{the imposed} time limit.

\noindent {(ii) \em Accuracy.} For deterministic MPECs,
\uss{(\texttt{ZSOL}$^{\bf 2s}_{\rm ncvx}$)} provides almost the same solutions
\uss{as the globally optimal solutions} in all problems from the literature,
which shows both efficacy and wide applicability of \uss{(\texttt{ZSOL}$^{\bf
        2s}_{\rm ncvx}$)}. \uss{In} high-dimensional SMPECs, \uss{(\texttt{ZSOL}$^{\bf
2s}_{\rm ncvx}$)} provides similar accuracy \uss{as} (SAA) but takes far less
computational time.

\section{Concluding remarks} 
Motivated by the apparent lacuna in non-asymptotic rate guarantees and
efficient first/zeroth-order schemes for MPECs, we consider a subclass of \fyy{stochastic} MPECs
where the parametrized lower-level equilibrium problem is  given by a
deterministic/stochastic variational inequality (VI) problem whose mapping is
strongly monotone, uniformly in upper-level decisions. Under suitable
assumptions, the implicit objective is Lipschitz continuous over a compact and
convex feasibility set, paving the way for developing a gradient-free locally
randomized smoothing framework applied to the implicit form the \fyy{SMPEC}. This
avenue allows for developing complexity guarantees in settings where the
implicit objective is either convex or nonconvex, the lower-level oracle is
exact (allowing for accelerated schemes in convex regimes) or inexact
(requiring the use of stochastic approximation to compute an inexact
lower-level decisions). We believe that this is but the first step in developing a
comprehensive zeroth-order foundation for contending with \fyy{SMPECs} under far
weaker assumptions. Possible extensions include settings where the lower-level map is merely monotone or
possibly non-monotone.
\bibliographystyle{siam}
\bibliography{demobib,wsc11-v03}

\begin{thebibliography}{10}

\bibitem{implicit5}
{\sc R.~P. Agdeppa, N.~Yamashita, and M.~Fukushima}, {\em An implicit
  programming approach for the road pricing problem with nonadditive route
  costs}, J. Ind. Manag. Optim., 4 (2008), pp.~183--197.

\bibitem{anitescu-solving}
{\sc M.~Anitescu}, {\em On solving mathematical programs with complementarity
  constraints as nonlinear programs}, SIAM J. Optim., 15(4) (2005),
  pp.~1203--1236.

\bibitem{bard1988convex}
{\sc J.~F. Bard}, {\em Convex two-level optimization}, Math. Programming, 40
  (1988), pp.~15--27.

\bibitem{baringo2013strategic}
{\sc L.~Baringo and A.~J. Conejo}, {\em Strategic offering for a wind power
  producer}, IEEE Transactions on Power Systems, 28 (2013), pp.~4645--4654.

\bibitem{beck14introduction}
{\sc A.~Beck}, {\em Introduction to nonlinear optimization}, vol.~19 of
  MOS-SIAM Series on Optimization, Society for Industrial and Applied
  Mathematics (SIAM), Philadelphia, PA; Mathematical Optimization Society,
  Philadelphia, PA, 2014.
\newblock Theory, algorithms, and applications with MATLAB.

\bibitem{beck17fom}
{\sc A.~Beck}, {\em First-Order Methods in Optimization}, SIAM, Philadelphia,
  PA, 2017.

\bibitem{implicit7}
{\sc P.~Beremlijski, J.~Haslinger, M.~Ko\v{c}vara, and J.~Outrata}, {\em Shape
  optimization in contact problems with {C}oulomb friction}, SIAM J. Optim., 13
  (2002), pp.~561--587.

\bibitem{burke05robust}
{\sc J.~V. Burke, A.~S. Lewis, and M.~L. Overton}, {\em A robust gradient
  sampling algorithm for nonsmooth, nonconvex optimization}, SIAM J. Optim., 15
  (2005), pp.~751--779.

\bibitem{chen12smoothing}
{\sc X.~Chen}, {\em Smoothing methods for nonsmooth, nonconvex minimization},
  Math. Program., 134 (2012), pp.~71--99.

\bibitem{chen15regularized}
{\sc X.~Chen, H.~Sun, and R.~J.-B. Wets}, {\em Regularized mathematical
  programs with stochastic equilibrium constraints: estimating structural
  demand models}, SIAM J. Optim., 25 (2015), pp.~53--75.

\bibitem{clarke98}
{\sc F.~H. Clarke, Y.~S. Ledyaev, R.~J. Stern, and P.~R. Wolenski}, {\em
  Nonsmooth analysis and control theory}, vol.~178 of Graduate Texts in
  Mathematics, Springer-Verlag, New York, 1998.

\bibitem{ConnScheVice09}
{\sc A.~R. Conn, K.~Scheinberg, and L.~N. Vicente}, {\em Introduction to
  Derivative-Free Optimization}, SIAM, Philadelphia, PA, USA, 2009.

\bibitem{cui2021analysis}
{\sc S.~Cui and U.~V. Shanbhag}, {\em On the analysis of variance-reduced and
  randomized projection variants of single projection schemes for monotone
  stochastic variational inequality problems}, Set-Valued and Variational
  Analysis, 29 (2021), pp.~453--499.

\bibitem{czyzyk_et_al_1998}
{\sc J.~{Czyzyk}, M.~P. {Mesnier}, and J.~J. {Mor{\'{e}}}}, {\em The {NEOS}
  server}, IEEE Journal on Computational Science and Engineering, 5 (1998),
  pp.~68--75.

\bibitem{demiguel2005two}
{\sc V.~DeMiguel, M.~P. Friedlander, F.~J. Nogales, and S.~Scholtes}, {\em A
  two-sided relaxation scheme for mathematical programs with equilibrium
  constraints}, SIAM Journal on Optimization, 16 (2005), pp.~587--609.

\bibitem{demiguel09stochastic}
{\sc V.~DeMiguel and H.~Xu}, {\em A stochastic multiple-leader stackelberg
  model: Analysis, computation, and application}, Operations Research, 57
  (2009), pp.~1220--1235.

\bibitem{dolan_2001}
{\sc E.~D. {Dolan}}, {\em The neos server 4.0 administrative guide}, Technical
  Memorandum ANL/MCS-TM-250, Mathematics and Computer Science Division, Argonne
  National Laboratory, 2001.

\bibitem{Duchi12}
{\sc J.~C. Duchi, P.~L. Bartlett, and M.~J. Wainwright}, {\em Randomized
  smoothing for stochastic optimization}, SIAM Journal on Optimization (SIOPT),
  22 (2012), pp.~674--701.

\bibitem{evgrafov2003stochastic}
{\sc A.~Evgrafov and M.~Patriksson}, {\em Stochastic structural topology
  optimization: discretization and penalty function approach}, Structural and
  Multidisciplinary Optimization, 25 (2003), pp.~174--188.

\bibitem{facchinei1999smoothing}
{\sc F.~Facchinei, H.~Jiang, and L.~Qi}, {\em A smoothing method for
  mathematical programs with equilibrium constraints}, Mathematical
  programming, 85 (1999), pp.~107--134.

\bibitem{facchinei02finite}
{\sc F.~Facchinei and J.-S. Pang}, {\em Finite-dimensional Variational
  Inequalities and Complementarity Problems. {V}ols. {I,II}}, Springer Series
  in Operations Research, Springer-Verlag, New York, 2003.

\bibitem{fang2015coupon}
{\sc X.~Fang, Q.~Hu, F.~Li, B.~Wang, and Y.~Li}, {\em Coupon-based demand
  response considering wind power uncertainty: A strategic bidding model for
  load serving entities}, IEEE Transactions on Power Systems, 31 (2015),
  pp.~1025--1037.

\bibitem{ferris2002mathematical}
{\sc M.~C. Ferris, S.~P. Dirkse, and A.~Meeraus}, {\em Mathematical programs
  with equilibrium constraints: Automatic reformulation and solution via
  constrained optimization},  (2002).

\bibitem{flaxman2005online}
{\sc A.~Flaxman, A.~T. Kalai, and B.~McMahan}, {\em Online convex optimization
  in the bandit setting: Gradient descent without a gradient}, in SODA '05
  Proceedings of the sixteenth annual ACM-SIAM symposium on Discrete
  algorithms, January 2005, pp.~385--394.

\bibitem{fletcher02local}
{\sc R.~Fletcher, S.~Leyffer, D.~Ralph, and S.~Scholtes}, {\em Local
  convergence of sqp methods for mathematical programs with equilibrium
  constraints}, SIAM Journal on Optimization, 17 (2006), pp.~259--286.

\bibitem{ghadimi13zeroth}
{\sc S.~Ghadimi and G.~Lan}, {\em Stochastic first- and zeroth-order methods
  for nonconvex stochastic programming}, SIAM J. Optim., 23 (2013),
  pp.~2341--2368.

\bibitem{ghadimi15}
{\sc S.~Ghadimi, G.~Lan, and H.~Zhang}, {\em Mini-batch stochastic
  approximation methods for nonconvex stochastic composite optimization},
  Mathematical Programming, 155 (2016), pp.~267--–305.

\bibitem{goldstein77}
{\sc A.~A. Goldstein}, {\em Optimization of {L}ipschitz continuous functions},
  Math. Programming, 13 (1977), pp.~14--22.

\bibitem{gropp_more_1997}
{\sc W.~{Gropp} and J.~J. {Mor{\'{e}}}}, {\em Optimization environments and the
  neos server}, in Approximation Theory and Optimization, M.~D. {Buhman} and
  A.~{Iserles}, eds., Cambridge University Press, 1997, p.~167.

\bibitem{implicit2}
{\sc M.~Hinterm\"{u}ller and T.~Surowiec}, {\em A bundle-free implicit
  programming approach for a class of elliptic {MPEC}s in function space},
  Math. Program., 160 (2016), pp.~271--305.

\bibitem{hobbs-strategic}
{\sc B.~F. Hobbs, C.~B. Metzler, and J.-S. Pang}, {\em Strategic gaming
  analysis for electric power systems: An {MPEC} approach}, IEEE Transactions
  on Power Systems, 15 (2000), pp.~638--645.

\bibitem{xu04convergence}
{\sc X.~Hu and D.~Ralph}, {\em Convergence of a penalty method for mathematical
  programming with complementarity constraints}, Journal of Optimization Theory
  and Applications, 123 (2004), pp.~365--398.

\bibitem{iusem19variance}
{\sc A.~N. Iusem, A.~Jofr{\'{e}}, R.~I. Oliveira, and P.~Thompson}, {\em
  Variance-based extragradient methods with line search for stochastic
  variational inequalities}, {SIAM} J. Optim., 29 (2019), pp.~175--206.

\bibitem{jalilzadeh19proximal}
{\sc A.~Jalilzadeh and U.~V. Shanbhag}, {\em A proximal-point algorithm with
  variable sample-sizes {(PPAWSS)} for monotone stochastic variational
  inequality problems}, in 2019 Winter Simulation Conference, {WSC} 2019,
  National Harbor, MD, USA, December 8-11, 2019, {IEEE}, 2019, pp.~3551--3562.

\bibitem{jalilzadeh2018smoothed}
{\sc A.~Jalilzadeh, U.~V. Shanbhag, J.~H. Blanchet, and P.~W. Glynn}, {\em
  Smoothed variable sample-size accelerated proximal methods for nonsmooth
  stochastic convex programs}, arXiv preprint arXiv:1803.00718,  (2018).

\bibitem{jiang00smooth}
{\sc H.~Jiang and D.~Ralph}, {\em Smooth {SQP} methods for mathematical
  programs with nonlinear complementarity constraints}, SIAM Journal on
  Optimization, 10(3) (2000), pp.~779--808.

\bibitem{Houyuan08}
{\sc H.~Jiang and H.~Xu}, {\em Stochastic approximation approaches to the
  stochastic variational inequality problem}, IEEE Transactions in Automatic
  Control, 53 (2008), pp.~1462--1475.

\bibitem{Nem11}
{\sc A.~Juditsky, A.~Nemirovski, and C.~Tauvel}, {\em Solving variational
  inequalities with stochastic mirror-prox algorithm}, Stochastic Systems, 1
  (2011), pp.~17--58.

\bibitem{implicit3}
{\sc Y.~Kanno}, {\em An implicit formulation of mathematical program with
  complementarity constraints for application to robust structural
  optimization}, J. Oper. Res. Soc. Japan, 54 (2011), pp.~65--85.

\bibitem{KaushikYousefian2021}
{\sc H.~D. Kaushik and F.~Yousefian}, {\em A method with convergence rates for
  optimization problems with variational inequality constraints},
  arXiv:2007.15845v2,  (2021).

\bibitem{Knopp_1951}
{\sc K.~Knopp}, {\em Theory and applications of infinite series}, Blackie \&
  Son Ltd., Bishopbriggs, Glasgow G64 2NZ, Scotland, 1951.

\bibitem{implicit6}
{\sc M.~Ko\v{c}vara and J.~V. Outrata}, {\em Optimization problems with
  equilibrium constraints and their numerical solution}, Math. Program., 101
  (2004), pp.~119--149.

\bibitem{implicit1}
\leavevmode\vrule height 2pt depth -1.6pt width 23pt, {\em Inverse truss design
  as a conic mathematical program with equilibrium constraints}, Discrete
  Contin. Dyn. Syst. Ser. S, 10 (2017), pp.~1329--1350.

\bibitem{DeFarias08}
{\sc H.~Lakshmanan and D.~Farias}, {\em Decentralized recourse allocation in
  dynamic networks of agents}, SIAM Journal on Optimization, 19 (2008),
  pp.~911--940.

\bibitem{hearn04mpec}
{\sc S.~Lawphongpanich and D.~W. Hearn}, {\em An {MPEC} approach to second-best
  toll pricing}, Math. Program., 101 (2004), pp.~33--55.

\bibitem{leyffer2006interior}
{\sc S.~Leyffer, G.~L{\'o}pez-Calva, and J.~Nocedal}, {\em Interior methods for
  mathematical programs with complementarity constraints}, SIAM Journal on
  Optimization, 17 (2006), pp.~52--77.

\bibitem{lin09solving}
{\sc G.-H. Lin, X.~Chen, and M.~Fukushima}, {\em Solving stochastic
  mathematical programs with equilibrium constraints via approximation and
  smoothing implicit programming with penalization}, Math. Program., 116
  (2009), pp.~343--368.

\bibitem{liu2019successive}
{\sc T.~Liu, T.~K. Pong, and A.~Takeda}, {\em A successive difference-of-convex
  approximation method for a class of nonconvex nonsmooth optimization
  problems}, Mathematical Programming, 176 (2019), pp.~339--367.

\bibitem{liu11convergence}
{\sc Y.~Liu and G.-H. Lin}, {\em Convergence analysis of a regularized sample
  average approximation method for stochastic mathematical programs with
  complementarity constraints}, Asia-Pac. J. Oper. Res., 28 (2011),
  pp.~755--771.

\bibitem{luo96mathematical}
{\sc Z.-Q. Luo, J.-S. Pang, and D.~Ralph}, {\em Mathematical Programs with
  Equilibrium Constraints}, Cambridge University Press, Cambridge, 1996.

\bibitem{mayne84}
{\sc D.~Q. Mayne and E.~Polak}, {\em Nondifferential optimization via adaptive
  smoothing}, J. Optim. Theory Appl., 43 (1984), pp.~601--613.

\bibitem{migdalas1998multilevel}
{\sc A.~Migdalas, P.~M. Pardalos, and P.~V{\"a}rbrand}, {\em Multilevel
  optimization: algorithms and applications}, vol.~20, Springer Science \&
  Business Media, 1998.

\bibitem{implicit4}
{\sc B.~S. Mordukhovich}, {\em Characterizations of linear suboptimality for
  mathematical programs with equilibrium constraints}, Math. Program., 120
  (2009), pp.~261--283.

\bibitem{murphy82mathematical}
{\sc F.~H. Murphy, H.~D. Sherali, and A.~L. Soyster}, {\em A mathematical
  programming approach for determining oligopolistic market equilibrium}, Math.
  Programming, 24 (1982), pp.~92--106.

\bibitem{nemirovski_robust_2009}
{\sc A.~Nemirovski, A.~Juditsky, G.~Lan, and A.~Shapiro}, {\em Robust
  stochastic approximation approach to stochastic programming}, {SIAM} Journal
  on Optimization, 19 (2009), pp.~1574--1609.

\bibitem{nemirovskij1983problem}
{\sc A.~S. Nemirovskij and D.~B. Yudin}, {\em Problem complexity and method
  efficiency in optimization},  (1983).

\bibitem{nesterov83}
{\sc Y.~Nesterov}, {\em A method for unconstrained convex minimization problem
  with the rate of convergence ${{O}(1/k^2)}$}, Doklady AN USSR, 269 (1983),
  pp.~543--547.

\bibitem{nesterov1998introductory}
{\sc Y.~Nesterov}, {\em Introductory lectures on convex programming volume i:
  Basic course}, Lecture notes,  (1998).

\bibitem{nesterov17}
{\sc Y.~Nesterov and V.~Spokoiny}, {\em Random gradient-free minimization of
  convex functions}, Found. Comput. Math., 17 (2017), pp.~527--566.

\bibitem{outrata98nonsmooth}
{\sc J.~Outrata, M.~Ko{\v{c}}vara, and J.~Zowe}, {\em Nonsmooth {A}pproach to
  {O}ptimization {P}roblems with {E}quilibrium {C}onstraints}, vol.~28 of
  Nonconvex Optimization and its Applications, Kluwer Academic Publishers,
  Dordrecht, 1998.
\newblock Theory, applications and numerical results.

\bibitem{outrata1995numerical}
{\sc J.~Outrata and J.~Zowe}, {\em A numerical approach to optimization
  problems with variational inequality constraints}, Mathematical Programming,
  68 (1995), pp.~105--130.

\bibitem{outrata1994optimization}
{\sc J.~V. Outrata}, {\em On optimization problems with variational inequality
  constraints}, SIAM Journal on optimization, 4 (1994), pp.~340--357.

\bibitem{patriksson2008applicability}
{\sc M.~Patriksson}, {\em On the applicability and solution of bilevel
  optimization models in transportation science: A study on the existence,
  stability and computation of optimal solutions to stochastic mathematical
  programs with equilibrium constraints}, Transportation Research Part B:
  Methodological, 42 (2008), pp.~843--860.

\bibitem{patriksson99stochastic}
{\sc M.~Patriksson and L.~Wynter}, {\em Stochastic mathematical programs with
  equilibrium constraints}, Operations Research Letters, 25 (1999),
  pp.~159--167.

\bibitem{Polyak87}
{\sc B.~T. Polyak}, {\em Introduction to Optimization}, Optimization Software,
  Inc., New York, 1987.

\bibitem{ragunathan05interior}
{\sc A.~U. Raghunathan and L.~T. Biegler}, {\em An interior point method for
  mathematical programs with complementarity constraints ({MPCC}s)}, SIAM J.
  Optim., 15 (2005), pp.~720--750 (electronic).

\bibitem{robbins51sa}
{\sc H.~Robbins and S.~Monro}, {\em A stochastic approximation method}, Ann.
  Math. Statistics, 22 (1951), pp.~400--407.

\bibitem{rockafellar17stochastic}
{\sc R.~T. Rockafellar and R.~J.-B. Wets}, {\em Stochastic variational
  inequalities: Single-stage to multistage}, Math. Program., 165 (2017),
  p.~331–360.

\bibitem{sahinidis:baron:21.1.13}
{\sc N.~V. Sahinidis}, {\em {BARON 21.1.13: Global Optimization of
  Mixed-Integer Nonlinear Programs, {\em User's Manual}}}, 2017.

\bibitem{scheel2000mathematical}
{\sc H.~Scheel and S.~Scholtes}, {\em Mathematical programs with
  complementarity constraints: stationarity, optimality, and sensitivity},
  Math. Oper. Res., 25 (2000), pp.~1--22.

\bibitem{shapiro2006}
{\sc A.~Shapiro}, {\em Stochastic programming with equilibrium constraints},
  Journal of optimization theory and applications, 128 (2006), pp.~223--243.

\bibitem{shapiro08stochastic}
{\sc A.~Shapiro and H.~Xu}, {\em Stochastic mathematical programs with
  equilibrium constraints, modelling and sample average approximation},
  Optimization, 57 (2008), pp.~395--418.

\bibitem{sherali84multiple}
{\sc H.~D. Sherali}, {\em A multiple leader {S}tackelberg model and analysis},
  Oper. Res., 32 (1984), pp.~390--404.

\bibitem{sherali83stackelberg}
{\sc H.~D. Sherali, A.~L. Soyster, and F.~H. Murphy}, {\em
  Stackelberg-{N}ash-{C}ournot equilibria: characterizations and computations},
  Oper. Res., 31 (1983), pp.~253--276.

\bibitem{steklov1}
{\sc V.~A. Steklov}, {\em Sur les expressions asymptotiques decertaines
  fonctions définies par les équations différentielles du second ordre et
  leers applications au problème du dévelopement d'une fonction arbitraire en
  séries procédant suivant les diverses fonctions}, Comm. Charkov Math. Soc.,
  2 (1907), pp.~97--199.

\bibitem{su2007analysis}
{\sc C.-L. Su}, {\em Analysis on the forward market equilibrium model},
  Operations Research Letters, 35 (2007), pp.~74--82.

\bibitem{su07analysis}
{\sc C.-L. Su}, {\em Analysis on the forward market equilibrium model}, Oper.
  Res. Lett., 35 (2007), pp.~74--82.

\bibitem{ts:05}
{\sc M.~Tawarmalani and N.~V. Sahinidis}, {\em {A polyhedral branch-and-cut
  approach to global optimization}}, Mathematical Programming, 103 (2005),
  pp.~225--249.

\bibitem{xu06implicit}
{\sc H.~Xu}, {\em An implicit programming approach for a class of stochastic
  mathematical programs with complementarity constraints}, SIAM J. Optim., 16
  (2006), pp.~670--696.

\bibitem{xuye11}
{\sc H.~Xu and J.~Ye}, {\em Approximating stationary points of stochastic
  mathematical programs with equilibrium constraints via sample averaging},
  Set-Valued and Variational Analysis, 128 (2011), pp.~283--309.

\bibitem{xu2019stochastic}
{\sc Y.~Xu, Q.~Qi, Q.~Lin, R.~Jin, and T.~Yang}, {\em Stochastic optimization
  for dc functions and non-smooth non-convex regularizers with non-asymptotic
  convergence}, in International conference on machine learning, PMLR, 2019,
  pp.~6942--6951.

\bibitem{Farzad1}
{\sc F.~Yousefian, A.~Nedi\'c, and U.~V. Shanbhag}, {\em On stochastic gradient
  and subgradient methods with adaptive steplength sequences}, Automatica, 48
  (2012), pp.~56--67.

\bibitem{FarzadAngeliaUday_MathProg17}
{\sc F.~Yousefian, A.~Nedic, and U.~V. Shanbhag}, {\em On smoothing,
  regularization, and averaging in stochastic approximation methods for
  stochastic variational inequality problems}, Math. Program., 165 (2017),
  pp.~391--431.

\bibitem{yousefian10convex}
{\sc F.~{Yousefian}, A.~{Nedić}, and U.~V. {Shanbhag}}, {\em Convex
  nondifferentiable stochastic optimization: A local randomized smoothing
  technique}, in Proceedings of the 2010 American Control Conference, 2010,
  pp.~4875--4880.

\end{thebibliography}
\vspace{-0.2in}

\section{Appendix}
\begin{lemma}[cf. Lemma 10 in~\cite{FarzadAngeliaUday_MathProg17} and Lemma 2.14 in~\cite{KaushikYousefian2021}]\label{lem:harmonic_bounds}
\em Let $\ell$ and $N$ be arbitrary integers where $0\leq \ell \leq N-1$. The following hold.
\begin{itemize}
\item [(a)] $\ln\left( \frac{N+1}{\ell+1}\right) \leq \sum_{k=\ell}^{N-1}\frac{1}{k+1} \leq \frac{1}{\ell+1}+\ln\left( \frac{N}{\ell+1}\right)$.
\item [(b)] If $0\leq \alpha <1$, then for any $N \geq 2^{\frac{1}{1-\alpha}}-1$, we have 
$\frac{(N+1)^{1-\alpha}}{2(1-\alpha)}  \leq \sum_{k=0}^{N}\frac{1}{(k+1)^\alpha}\leq \frac{(N+1)^{1-\alpha}}{1-\alpha}$.
\end{itemize}
\end{lemma}
\begin{lemma}[Theorem 6, page 75 in~\cite{Knopp_1951}]\label{lem:convergence_sum}\em
	Let $\{u_t\}\subset \mathbb{R}^n$ denote a sequence of vectors where $\lim_{t \to \infty}u_t=\hat{u}$. Also, let $\{\alpha_k\}$ denote a sequence of strictly positive scalars such that $\sum_{k=0}^{\infty} \alpha_k = \infty$. Suppose $v_k\in \mathbb{R}^n$ is defined by {$ v_k \triangleq  \frac{\sum_{t=0}^{k}\alpha_t u_t}{\sum_{t=0}^{k}\alpha_t}$ for all $k\geq 0$}. Then{,} $ \lim\limits_{k\rightarrow\infty}v_k = \hat{u}$.
\end{lemma}
\begin{lemma}[{cf.~\cite{Polyak87}}]\label{lemma:supermartingale}\em Let $v_k,$ $u_k,$ $\alpha_k,$ and  $\beta_k$ be
nonnegative random variables, and let the
following relations hold almost surely: 
\begin{align*}
&\EXP{v_{k+1}\mid {\tilde \sF_k}} 
\le (1+\alpha_k)v_k - u_k + \beta_k \quad\hbox{ for all } k, \qquad  \sum_{k=0}^\infty \alpha_k < \infty,\qquad
\sum_{k=0}^\infty \beta_k < \infty,
\end{align*} 
where $\tilde \sF_k$ denotes the collection $v_0,\ldots,v_k$, $u_0,\ldots,u_k$,
$\alpha_0,\ldots,\alpha_k$, $\beta_0,\ldots,\beta_k$. 
Then, we have almost surely
$\lim_{k\to\infty}v_k = v$ and $\sum_{k=0}^\infty u_k < \infty,$
where $v \geq 0$ is some random variable.
\end{lemma}
\noindent {\bf Proof of Lemma~\ref{lem:sublinear}.} We use induction on $k$ for $k\geq 0$. We have $e_0  =\tfrac{\Gamma e_0}{0+\Gamma} \leq  \tfrac{\max \left\{\tfrac{\beta\gamma^2}{\alpha \gamma-1},\Gamma e_0\right\}}{0+\Gamma}$
implying that the hypothesis statement holds for $k=0$. Let us assume that $e_k \leq \tfrac{\theta_0}{k+\Gamma}$ for some $k\geq 0$ where $\theta_0 \triangleq \max \left\{\tfrac{\beta\gamma^2}{\alpha \gamma-1},\Gamma e_0\right\}$. Let the induction hypothesis \fyy{hold} for $k\geq 0$. We show that it holds for $k+1$ as well. We have
\begin{align*}
& \theta_0  \geq \tfrac{\beta\gamma^2}{\alpha \gamma-1}
\Rightarrow \quad  \theta_0 \leq \gamma(\theta_0 \alpha - \beta \gamma)
\Rightarrow \quad  \tfrac{\theta_0}{k+\Gamma} \leq \tfrac{\gamma(\theta_0 \alpha - \beta \gamma)}{k+\Gamma}
\Rightarrow \quad  \tfrac{\theta_0}{k+\Gamma+1} \leq \tfrac{\gamma(\theta_0 \alpha - \beta \gamma)}{k+\Gamma}\\
\Rightarrow \quad & \tfrac{\theta_0}{(k+\Gamma+1)(k+\Gamma)} \leq \tfrac{\gamma(\theta_0 \alpha - \beta \gamma)}{(k+\Gamma)^2}
\Rightarrow \quad  \theta_0\left(\tfrac{1}{k+\Gamma}-\tfrac{1}{k+\Gamma+1}\right)  \leq \tfrac{\gamma(\theta_0 \alpha - \beta \gamma)}{(k+\Gamma)^2}
\Rightarrow \quad   \tfrac{\theta_0}{k+\Gamma}- \tfrac{\gamma(\theta_0 \alpha - \beta \gamma)}{(k+\Gamma)^2}  \leq \tfrac{\theta_0}{k+\Gamma+1}\\
\Rightarrow \quad &  \left(1-\alpha \tfrac{\gamma}{k+\Gamma}\right) \tfrac{\theta_0}{k+\Gamma}+ \tfrac{\beta \gamma^2}{(k+\Gamma)^2}  \leq \tfrac{\theta_0}{k+\Gamma+1}
\Rightarrow \quad   \left(1-\alpha \gamma_k\right) \tfrac{\theta_0}{k+\Gamma}+ \beta \gamma_k^2 \leq \tfrac{\theta_0}{k+\Gamma+1}\\
\Rightarrow \quad &  \left(1-\alpha \gamma_k\right) e_k+ \beta \gamma_k^2 \leq \tfrac{\theta_0}{k+\Gamma+1}
\Rightarrow \quad   e_{k+1} \leq \tfrac{\theta_0}{k+\Gamma+1}.
\end{align*}

\noindent{\bf Academic examples and their stochastic counterparts in Section~\ref{sec:5.4}}
\begin{enumerate}
    \item[] {\bf Problem 1. This problem is described in \cite[Definition 4.1]{outrata1995numerical}}
\begin{align*}
f(\x,\y)=r_1(x)-xp(x+y_1+y_2+y_3+y_4),
\end{align*}
where $r_i(v)=c_iv+\tfrac{\beta_i}{\beta_i+1}K_i^{1/\beta_i}v^{(1+\beta_i)/\beta_i}$, $p(Q)=5000^{1/\gamma}Q^{-1/\gamma}$, $c_i$, $\beta_i$, $K_i$, $i=1,\cdots,5$ are given positive parameters in Table~\ref{p1}, $\gamma$ is a positive parameter, $Q=x+y_1+y_2+y_3+y_4$.
\begin{gather*}
\mathcal{X}=\{0\le x \le L\}. \\
F(\x,\y)=\left(\begin{aligned}
\nabla r_2(y_1)-p(&Q)-y_1\nabla p(Q) \\
&\vdots \\
\nabla r_5(y_4)-p(&Q)-y_4\nabla p(Q) 
\end{aligned}
\right). \\
\mathcal{Y}=\{0\le y_j \le L, \quad j=1,2,3,4\}.
\end{gather*}

\begin{table}[h]
\caption{Parameter specification for \fyy{P}roblem 1}
\begin{center}
    \begin{tabular}{ | c | c | c | c | c | c |  }
    \hline
	i & 1& 2 & 3 & 4 & 5 \\\hline
	$c_i$ & 10 & 8 & 6 & 4 & 2 \\
	$K_i$ & 5 & 5 & 5 & 5 & 5 \\
	$\beta_i$ & 1.2 & 1.1 & 1.0 & 0.9 & 0.8 \\ \hline
    \end{tabular}
\end{center}
\label{p1}
\end{table}

\noindent The following three examples were tested in \cite{outrata1995numerical,facchinei1999smoothing}. \\
\item[] {\bf Problem 2.}
\begin{gather*}
f(\x,\y)=x_1^2-2x_1+x_2^2-2x_2+y_1^2+y_2^2. \\
\mathcal{X}=\{0\le x_i \le 2, \quad i=1,2\}. \\
F(\x,\y)=\left(\begin{aligned}
2y_1-2x_1 \\
2y_2-2x_2 
\end{aligned}
\right). \\
\mathcal{Y}=\{(y_j-1)^2 \le 0.25, \quad j=1,2\}.
\end{gather*}
\item[] {\bf Problem 3.}
\begin{gather*}
f(\x,\y)=2x_1+2x_2-3y_1-3y_2-60+R[\max\{0,x_1+x_2+y_1-2y_2-40\}]^2. \\
\mathcal{X}=\{0\le x_i \le 50, \quad i=1,2\}. \\
F(\x,\y)=\left(\begin{aligned}
2y_1-2x_1+40 \\
2y_2-2x_2+40 
\end{aligned}
\right). \\
\mathcal{Y}=\{-10\le y_j \le 20, \ x_j-2y_j-10 \ge 0, \quad j=1,2\}.
\end{gather*}

\item[] {\bf Problem 4.}
\begin{gather*}
f(\x,\y)=\tfrac{1}{2}((x_1-y_1)^2+(x_2-y_2)^2). \\
\mathcal{X}=\{0\le x_i \le 10, \quad i=1,2\}. \\
F(\x,\y)=\left(\begin{aligned}
-34&+2y_1+\tfrac{8}{3}y_2 \\
-24.25&+1.25y_1+2y_2 
\end{aligned}
\right). \\
\mathcal{Y}=\{-x_{3-j}-y_j+15 \ge 0, \quad j=1,2\}.
\end{gather*}

\noindent The next problem is taken from~\cite{outrata1994optimization,facchinei1999smoothing}. In all tests, the only difference lies in the objective function.

\item[] {\bf Problem 5.}
\begin{gather*}
\mathcal{X}=\{0\le x \le 10\}. \\
F(\x,\y)=\left(\begin{gathered}
(1+0.2x)y_1-(3+1.333x)-0.333y_3+2y_1y_4-y_5 \\
(1+0.1x)y_2-x+y_3+2y_2y_4-y_6 \\
0.333y_1-y_2+1-0.1x \\
9+0.1x-y_1^2-y_2^2 \\
y_1 \\
y_2 
\end{gathered}
\right). \\
\mathcal{Y}=\{y_j \ge 0, \quad j=3,4,5,6\}.
\end{gather*}

\item[] {\bf High-dimensional stochastic counterparts.}

\noindent Consider the stochastic $N$-dimensional counterpart of Problem 1, defined as follows.
\begin{align*}
    f(\x,\y)=\mathbb{E}\left[r_1(x)-xp\left(x+\sum_{i=1}^n y_i,\omega\right)\right],
\end{align*}
where $r_i(v)=c_iv+\tfrac{\beta_i}{\beta_i+1}K_i^{1/\beta_i}v^{(1+\beta_i)/\beta_i}$, $p(Q,\omega)=5000^{1/\gamma(\omega)}Q^{-1/\gamma(\omega)}$, $c_i = 6$, $\beta_i=1$, $K_i = 5$, $i=1,\cdots,5$, $\gamma(\omega) \in \mathcal{U}(0.9,1.1)$ is a positive parameter, $Q=x+\sum_{i=1}^N y_i$. 
\begin{gather*}
\mathcal{X}=\{0\le x \le L\}. \\
F(\x,\y,\omega)=\left(\begin{aligned}
\nabla r_2(y_1)-p(&Q,\omega)-y_1\nabla p(Q,\omega) \\
&\vdots \\
\nabla r_n(y_n)-p(&Q,\omega)-y_n\nabla p(Q,\omega) 
\end{aligned}
\right). \\
\mathcal{Y}=\{0\le y_j \le L, \quad j=1,\cdots, n\}.
\end{gather*}
\noindent The stochastic $N$-dimensional counterpart of Problem 2.
\begin{gather*}
\mathbb{E}[f(\x,\y(\omega))], \mbox{ where } f(x,y(\omega)) =\|x - {\bf 1}\|^2 + \|y(\omega)\|^2.  \\
\mathcal{X}=\{0\le x_i \le 2, \quad i=1,\cdots, n\}. \\
F(\x,\y,\omega)=\left(\begin{aligned}
2y - 2x + \omega
\end{aligned}
\right). \\
\mathcal{Y}=\{\| y-{\bf 1}\|^2 \le 0.25 \}\fyyy{, \hbox {where }\omega \in \mathcal{U}(-0.5,0.5).}
\end{gather*}

\end{enumerate}

\end{document}